\documentclass[10pt]{article}

\usepackage[utf8]{inputenc}
\usepackage[T1]{fontenc}
\usepackage{amsmath,amsfonts,amssymb}
\usepackage{amsthm}
\usepackage[a4paper]{geometry}
\usepackage{color}
\usepackage{transparent}
\usepackage[dvipsnames]{xcolor}
\usepackage{graphicx}
\usepackage{tabularx,float}
\usepackage[round]{natbib}
\usepackage[hidelinks]{hyperref}
\usepackage{comment}
\usepackage{enumerate,mathrsfs}
\usepackage{mathtools}
\usepackage{enumitem}
\usepackage{orcidlink}
\usepackage{titling}
\usepackage[misc]{ifsym}
\usepackage[titletoc]{appendix}
\usepackage[intoc]{nomencl}
\usepackage{soul}
\usepackage{caption}
\usepackage{bbm}
\usepackage{overpic}
\usepackage{lmodern}
\usepackage[subrefformat=parens,labelformat=parens]{subcaption}
\DeclareRobustCommand{\doi}[1]{\href{https://doi.org/#1}{doi:\nolinkurl{#1}}}
\newcommand{\overbar}[1]{\mkern 1.5mu\overline{\mkern-1.5mu#1\mkern-1.5mu}\mkern 1.5mu}

\numberwithin{equation}{section}

\reversemarginpar

\newcommand*{\dd}{\mathrm{d}}

\newtheorem{theorem}{Theorem}[section]

\newtheorem{proposition}[theorem]{Proposition}
\newtheorem{corollary}[theorem]{Corollary}
\newtheorem{lemma}[theorem]{Lemma}

\theoremstyle{definition}
\newtheorem{definition}[theorem]{Definition}
\newtheorem{remark}[theorem]{Remark}
\newtheorem{example}[theorem]{Example}

\newtheorem{assumpA}{Assumption}

\counterwithin{figure}{section}

\newcommand{\eps}{\varepsilon}

\newcommand\norm[1]{\left\lVert#1\right\rVert}

\newcommand\abs[1]{\left\lvert#1\right\rvert}
\newcommand{\R}{\mathbb{R}}

\newcommand\blfootnote[1]{%
	\begingroup
	\renewcommand\thefootnote{}\footnote{#1}%
	\addtocounter{footnote}{-1}%
	\endgroup
}

\usepackage{fancyhdr}
\title{\textbf{Embedding Capabilities of Neural ODEs}}

\author{Christian Kuehn \orcidlink{0000-0002-7063-6173}$^{1,2,3}$ 
	\& Sara-Viola Kuntz \orcidlink{0009-0000-4611-9742}$^{1,2,3}$
}

\date{
	\small{$^1$\textit{Technical University of Munich, School of Computation, Information and Technology, \\ Department of Mathematics, Boltzmannstraße 3, 85748 Garching, Germany} \\
		$^2$\textit{Munich Data Science Institute (MDSI), Garching, Germany } \\
		$^3$\textit{Munich Center for Machine Learning (MCML), München, Germany }}\\[5mm]
	\large{November 23, 2026}
}
\begin{document}
	\maketitle
	
	\begin{abstract}%
		Neural ordinary differential equations (neural ODEs) are a class of neural networks that have gained particular interest in recent years. We study input-output relations of neural ODEs using dynamical systems theory and prove several results about the exact embedding of maps in different neural ODE architectures in low and high dimensions. The embedding capability of a neural ODE architecture can be increased by adding, for example, affine linear layers or augmenting the phase space. However, there is currently no systematic theory available, and our work contributes towards this goal by developing various embedding results as well as identifying situations where no embedding is possible. Our core mathematical techniques include iterative functional equations, Morse functions, and suspension flows, as well as several further ideas from analysis. 
		Although in practice mostly universal approximation theorems are used, our geometric dynamical systems viewpoint on universal embedding provides a fundamental understanding of why certain neural ODE architectures perform better than others.
	\end{abstract}
	
	\vspace{1mm}
	
	
	\vspace{2mm}
	
	
	\blfootnote{\hspace{-5.4mm}\Letter \; ckuehn@ma.tum.de (Christian Kuehn)  \\
		\Letter\; saraviola.kuntz@tum.de (Sara-Viola Kuntz, corresponding author)}
	
	\section{Introduction}
	
	Neural networks are a machine learning technique inspired by the human brain. The goal is to create an artificial intelligence that is in theory capable of learning any mathematical function. A general neural network consists of neurons, which can be represented as nodes of a graph, and weighted connections in between, which can be represented as edges of the graph. Based on an input and the weights and biases that are used as parameters, a neural network computes an output. The process of adapting the weighted connections to data is called learning~\citep{Aggarwal2018}.\medskip 
	
	One of the simplest neural network models is the multilayer perceptron (MLP), studied already by~\citet{Rosenblatt1957}. MLPs are feed-forward neural networks, which are structured in layers $h_l$ for $l \in \{0,1,\ldots,L\}$ with input layer $h_0$ and output layer $h_L$. The layers $h_1,h_2,\ldots,h_{L-1}$ are called hidden layers. Each layer consists of a number of neurons \mbox{$n_l \in \mathbb{N}$}, $l \in \{0,1,\ldots,L\}$. The state of each neuron is represented by a real number, and the states of all neurons in layer $l$ are denoted by $h_l \in \mathbb{R}^{n_l}$. The links between neurons of consecutive layers are characterized by weight matrices $W_l \in \mathbb{R}^{n_{l+1}\times n_l}$ and biases $b_l \in \R^{n_{l+1}}$, building the parameters $\theta_l\in \mathbb{R}^{p_l}$ for $l \in \{0,1,\ldots,L-1\}$. For feed-forward neural networks, the layers are typically computed by an iterative update rule of the preceding layer. Each layer $h_{l+1}$ is calculated by a layer map \mbox{$f_{l}: \mathbb{R}^{n_l} \times \mathbb{R}^{p_l} \rightarrow \mathbb{R}^{n_{l+1}}$} of the preceding layer $h_l$ and the parameters~$\theta_l$:
	\begin{equation} \label{eq:intro_perceptron}
		h_{l+1} = f_{l}(h_l,\theta_l), \qquad l \in \{0,1,\ldots,L-1\}.
	\end{equation}  
	In summary, the feed-forward neural network maps the input $h_0$ to the output $h_L$. In the case of MLPs, the nonlinear activation function is often chosen to be a $\tanh$, a sigmoid or a (normal, leaky or parametric) $\textup{ReLU}$ applied component-wise to the affine linear transformation $W_l h_l + b_l \in \mathbb{R}^{n_{l+1}}$ with weight matrix $W_l$ and bias $b_l$.
	
	More advanced classes of neural networks are residual neural networks (ResNets)~\citep[see][]{He2016} and recurrent neural networks (RNNs)~\citep[see][]{Rumelhart1985}. Since RNNs can be viewed as a ResNet with weights shared across all layers~\citep{Liao2016}, the ResNet architecture represents a more general class that relaxes this constraint. In the following, we therefore focus on the broader class of ResNets. In contrast to MLPs, the layer structure is weakened and additional shortcut/skip connections are allowed. As the additive skip connection requires matching dimensions, all layers within a residual block share a common width $m \in \mathbb{N}$ and parameter dimension $p \in \mathbb{N}$. Each layer $h_{l+1} \in \mathbb{R}^{m}$ is computed as the sum of the preceding layer $h_l$ and the residual map $f_\textup{ResNet}: \mathbb{R}^m \times \mathbb{R}^{p} \rightarrow \mathbb{R}^m$, which is independent of $l$:
	\begin{equation}\label{eq:intro_ResNet}
		h_{l+1} = h_{l} + f_\textup{ResNet}(h_l, \theta_l), \qquad l \in \{0,1,\ldots,L-1\}.
	\end{equation} 
	As before, the neural network is a function mapping the input layer $h_0$ to the output layer~$h_L$. We focus exclusively on the constant-dimension regime. Any dimension-changing transitions, for example via strided convolutions or projection shortcuts, are structurally distinct from the update rule~\eqref{eq:intro_ResNet} and fall outside the scope of this work. \medskip
	
	The iterative updates~\eqref{eq:intro_ResNet} can be obtained as an Euler discretization of the ordinary differential equation (ODE) 
	\begin{equation} \label{eq:intro_ODE_1}
		\frac{\dd h}{\dd t} = f_\textup{ODE}(h(t),\theta(t)), \qquad h(0) = x,
	\end{equation}
	on the time interval $ [0,T]$ with step size $\delta = \frac{T}{L}$ and vector field $$f_\textup{ODE}(\cdot,\cdot) = \frac{1}{\delta} f_\textup{ResNet}(\cdot,\cdot): \mathbb{R}^m \times \mathbb{R}^{p} \rightarrow \mathbb{R}^m$$ \citep{Chen2018, Weinan2017}. The function $h: [0,T] \rightarrow \mathbb{R}^m$ can hereby be interpreted as hidden states, and $\theta: [0,T] \rightarrow \mathbb{R}^{p}$ as the parameters. The initial condition $x \in \mathbb{R}^m$ corresponds to the input layer $h_0 \in \mathbb{R}^m$ of the neural network and the output layer $h_L \in \mathbb{R}^m$ corresponds to the time-$T$ map  $h_x(T)$~\citep[see][]{Guckenheimer2002} of the ODE~\eqref{eq:intro_ODE_1}. The Euler discretization of~\eqref{eq:intro_ODE_1} is
	\begin{equation*}
		h(t+\delta) \approx h(t) + \delta f_\textup{ODE}(h(t),\theta(t))
	\end{equation*}
	for $t \in \{0,\delta,2\delta,\ldots,T-\delta\}$. As such a discretization subdivides the time interval $[0,T]$ into~$L$ intervals, it can be interpreted as a ResNet in which each layer with index $l$ corresponds to the discrete time step~$l \delta \in [0,T]$:
	\begin{align*}
		h_{(t+\delta)/\delta} &= h_{t/\delta} + \delta f_\textup{ODE}(h_{t/\delta},\theta_{t/\delta}) \\
		\Leftrightarrow  \hspace{11mm} h_{l+1} &= h_{l} + f_\textup{ResNet}(h_{l},\theta_l), 
	\end{align*}
	with $l \in  \{0,\ldots,L-1\}$. This shows that ResNets of the form~\eqref{eq:intro_ResNet} with finite width can be obtained as Euler discretizations of the ODE~\eqref{eq:intro_ODE_1}. The Euler approximation becomes more accurate the smaller the step size~$\delta$, i.e., the larger the depth~$L$ of the neural network for fixed $T$. To better understand the behavior of deep ResNets, i.e., ResNets with a large number of layers $L \gg 1$, it is helpful to study the solutions of the underlying ODE~\eqref{eq:intro_ODE_1} mapping the input $h(0) = x$ to the output $h_x(T)$. 

    As $\theta: [0,T] \rightarrow \mathbb{R}^p$ is itself a function of time, the right-hand side of~\eqref{eq:intro_ODE_1} depends on both $h(t)$ and $t$. It is therefore natural to study the more general class of non-autonomous ODEs with vector field $f: [0,T] \times \mathbb{R}^m \rightarrow \mathbb{R}^m$:
    \begin{equation} \label{eq:intro_ODE_2}
        \frac{\dd h}{\dd t} = f(t, h(t)), \qquad h(0) = x.
    \end{equation}
	In machine learning, the ODE~\eqref{eq:intro_ODE_2} with $f$ parameterized by a neural network architecture is referred to as a neural ordinary differential equation (neural ODE)~\citep[see][]{Chen2018}, but it is also a classical object in the theory of differential equations, studied in many contexts. The main difference in the context of artificial intelligence is that the focus lies on input-output relations of neural ODEs on finite time-scales. In this work, we expand upon this viewpoint using techniques from the theory of dynamical systems. Crucially, our analysis does not depend on a specific parameterization of $f$. We study the general class of non-autonomous vector fields in~\eqref{eq:intro_ODE_2}, of which the neural ODE framework of~\citet{Chen2018} is a special case. Different parameterizations such as piecewise constant weights, spectral discretization, or hypernetworks~\citep[see][]{Kidger2022} are special cases of the vector field $f$, which may be discontinuous in the time variable. As will be specified in Section~\ref{sec:overview}, we consider all vector fields $f$ for which the solution $h$ of the initial value problem \eqref{eq:intro_ODE_2} depends continuously on $t$ and on the initial condition $x$.
    
	Neural ODEs can be trained with the adjoint sensitivity method studied already by~\citet{Pontryagin1986} and then adapted to neural ODEs by~\citet{Chen2018}. The idea is to numerically solve a second augmented ODE backwards in time to compute the gradients needed to update the parameters. Hence, neural ODEs can be trained without storing intermediate quantities and thus the memory requirement is constant. In contrast, the memory cost of training feed-forward neural networks increases with the depth $L$ of the network. Another advantage of neural ODEs is that they can not only embed functions as the time-$T$ map of the ODE, but also model time-series data via the solution function $h(t)$. Compared to discrete networks, the data can lie on a continuous time-scale and does not need to be spaced equally.\medskip 
	
	In the context of expressivity, an important concept for neural networks is universal approximation, which means that the set of functions a neural network can represent is dense in the space of underlying functions. 
    The universal approximation property has been studied in great generality by~\citet{Kratsios2021} and for differentiable manifolds by~\citet{Kratsios2022}. For the purpose of this paper, we use the following definition, adapted to the space $C^k(\mathcal{X},\mathcal{Y})$ of $k$-times continuously differentiable functions on $C^k$-manifolds. Note that while this definition applies to $C^k$-manifolds $\mathcal{X}$ and $\mathcal{Y}$, in most of our subsequent results, $\mathcal{X}$ and $\mathcal{Y}$ are subsets of Euclidean spaces.

    \begin{definition}[Universal Approximation]
        Let $\mathcal{N} = \{\Phi_\theta \colon \mathcal{X} \to \mathcal{Y}\}_{\theta \in \Theta}$ be a neural network family with $C^k$-manifolds $\mathcal{X}$ and $\mathcal{Y}$, parameters $\theta \in \Theta$, and let $\mathrm{dist}_{\mathcal{Y}}: \mathcal{Y} \times \mathcal{Y} \rightarrow \R$ be a metric on $\mathcal{Y}$ inducing the topology of $\mathcal{Y}$. Then $\mathcal{N}$ has the \emph{universal approximation property} with respect to the space $C^k(\mathcal{X}, \mathcal{Y})$, $k \geq 0$, if for every $\varepsilon > 0$, every compact subset $\mathcal{K} \subset \mathcal{X}$, and every $\Psi \in C^k(\mathcal{X}, \mathcal{Y})$, there exists $\theta \in \Theta$ such that $\sup_{x \in \mathcal{K}}\, \mathrm{dist}_{\mathcal{Y}}(\Phi_\theta(x), \Psi(x)) < \varepsilon$.
    \end{definition}

	For feed-forward neural networks like MLPs, ResNets, and RNNs, various universal approximation theorems exist~\citep{Hornik1989, Kidger2020, Lin2018, Pinkus1999, Schaefer2006}, stating that by increasing the width or depth of the network and the number of parameters, any function $\Psi \in C^k( \mathcal{X}, \mathcal{Y})$ can be approximated arbitrarily well on every compact subset of $\mathcal{X}$. Although universal approximation is practically extremely useful, its proofs tend to require careful tracking of intermediate approximation errors. In contrast, if we demand an exact representation, the mathematical arguments gain clarity. We define a neural network to have the universal embedding property if every $C^k$-function can be represented exactly:
	
	\begin{definition}[Universal Embedding]
		  Let $\mathcal{N} = \{\Phi_\theta \colon \mathcal{X} \to \mathcal{Y}\}_{\theta \in \Theta}$ be a neural network family with $C^k$-manifolds $\mathcal{X}$ and $\mathcal{Y}$ and parameters $\theta \in \Theta$. Then $\mathcal{N}$ has the \emph{universal embedding property} with respect to the space $C^k(\mathcal{X}, \mathcal{Y})$, $k \geq 0$, if for every $\Psi \in C^k(\mathcal{X}, \mathcal{Y})$, there exists $\theta \in \Theta$ such that $\Phi_\theta(x) = \Psi(x)$ for all $x \in \mathcal{X}$.
	\end{definition}
	
	Embedding capabilities are already interesting on their own and can help to understand the approximation capability of a network. The relationship between the exact representation of critical points and the universal approximation property is also discussed for MLPs, neural ODEs, ResNets, and neural delay differential equations (neural DDEs) in the recent works of~\citet{Kuehn2025,Kuehn2024} and \citet{Kuehn2026}. In this work, we study general neural ODE architectures, which are compositions of functions including the time-$T$ map $h_x(T)$ of the ODE~\eqref{eq:intro_ODE_2}~\citep[see for example][]{Dupont2019, Kidger2022, RuizBalet2023,  Teshima2020, Zhang2020a}. 
	
	We aim to contribute to the study of neural ODEs with a dynamical systems viewpoint. In this work, we systematically study if and which functions can be embedded in different neural ODE architectures. In particular, we do not consider how the parameters of the right-hand side can be learned. In this paper, we introduce different neural ODE architectures, generalize and mathematically sharpen existing results, prove several completely new structure theorems, and develop a more transparent context for the embedding capabilities of neural ODEs. In particular, for each neural ODE architecture we contribute to at least two of the following fundamental questions:
	
	\begin{itemize}
		\item[(Q1)] How does the neural ODE architecture perform in low dimensions? 
		\item[(Q2)] Are there function classes that cannot be embedded in the neural ODE architecture in an arbitrary dimension?
		\item[(Q3)] How large and of which type does the neural ODE architecture need to be to have the universal embedding or approximation property?
	\end{itemize}

    Even though neural ODEs in low dimensions are not the primary use case in applications, their study helps to understand, illustrate, and compare how different neural ODE architectures perform.
    The following sections review the existing literature most relevant to these questions, describe our contributions, and outline the structure of this work.
	
    \subsection*{Related Work}

    \paragraph{Neural ODEs and their connection to ResNets.}
    The interpretation of deep residual neural networks as discretizations of continuous dynamical systems was formalized independently by~\citet{Weinan2017},~\citet{Haber2017}, and~\citet{Chen2018}. The latter introduced the term neural ODE and proposed training via the adjoint sensitivity method~\citep{Pontryagin1986}. Our work takes this ODE perspective as a starting point and studies the input-output maps of neural ODEs from a dynamical systems viewpoint, rather than from an optimization or training perspective.

    \paragraph{Universal approximation for neural networks.}
    Classical universal approximation theorems for shallow (single-hidden-layer) multilayer perceptrons were established by \citet{Cybenko1989}, \citet{Hornik1989}, \citet{Funahashi1989}, and~\citet{Pinkus1999}. More recently, \citet{Kidger2020} and~\citet{Park2020} established universal approximation theorems for narrow networks, provided the network has arbitrary depth. An extension to residual networks with bounded width was given by~\citet{Lin2018}. Notably, the ResNet of~\citet{Lin2018} with one neuron per hidden layer can be viewed as a coarse Euler discretization of a neural ODE. At this level of coarseness, the discrete iterates are not constrained to approximate a continuous flow, and the topological obstructions arising from the flow structure of neural ODEs are absent. As shown in~\citet{Kuehn2026}, these obstructions reappear as the discretization becomes finer. More generally,~\citet{Kuehn2024} establish that the universal embedding properties of general neural ODEs carry over directly to universal approximation properties of parameterized neural ODEs, provided the parameterized vector field is sufficiently expressive. The results of the present paper can therefore be interpreted as expressivity results for practical parameterized and discrete architectures via this transfer.

    \paragraph{Expressivity limitations of neural ODEs.}
    While the above results establish strong approximation capabilities for classical architectures, neural ODEs described by the basic ODE~\eqref{eq:intro_ODE_2} face structural constraints. \citet{Dupont2019} identified the topological obstruction that the time-$T$ map is a homeomorphism, and proposed augmenting the phase space to overcome this restriction. \citet{Massaroli2020} provided a detailed analysis of augmented neural ODE architectures, examining how different augmentation strategies affect their expressivity and dynamics. \citet{Zhang2020a} studied this problem more systematically and proved the universal approximation property for augmented neural ODEs followed by a linear layer. Furthermore,~\citet{Teshima2020} leveraged structure theorems from differential geometry to prove that non-augmented neural ODEs combined with an affine transformation act as sup-universal approximators for a broad class of diffeomorphisms.~\citet{RuizBalet2023} analyzed controllability in relation to these approximation capabilities, while~\citet{Kidger2022} provided a comprehensive treatment of neural differential equations, detailing how various parameterization strategies and augmented phase spaces impact expressivity. Offering a different approach to approximation theory, \citet{Weinan2021} developed a norm-based framework for residual networks by constructing flow-induced function spaces defined via an ODE in the deep limit, proving direct and inverse approximation theorems.

    \paragraph{Critical points and neural ODE expressivity.}
    The relationship between critical points in the input-output map of a neural network and its expressive power was investigated in~\citet{Kuehn2024}: non-augmented neural ODEs and non-augmented multilayer perceptrons cannot represent functions with critical points, a structural limitation with direct consequences for the universal approximation property. The interplay between critical points and expressivity was further extended to neural delay differential equations in~\citet{Kuehn2025} and to ResNets in different parameter regimes in~\citet{Kuehn2026}.

    \paragraph{Dynamical systems tools.}
    As the expressivity limitations of neural ODEs are based on the constraints of continuous flows, classical tools from differential topology and dynamical systems are well suited to analyze them. The suspension flow and the associated mapping torus~\citep{Brin2002, Katok1995} are classical constructions that encode a discrete map into a continuous flow. Whitney's embedding theorem~\citep{Whitney1944} allows such abstract constructions to be transferred into Euclidean space. Morse theory~\citep{Milnor1963} and the Borsuk-Ulam theorem~\citep{Borsuk1933} are further classical tools from differential topology that are relevant to the study of embedding obstructions of critical points. To the best of our knowledge, the systematic use of these tools in the context of neural ODE expressivity is new.
	
	\subsection*{Our Contributions}

    We systematically study the embedding capabilities of neural ODEs from a dynamical systems perspective. This viewpoint yields concise, mathematically rigorous explanations of the key structural properties governing neural ODE embeddings. We consider four neural ODE architectures that differ in two aspects: whether the phase space is augmented with additional dimensions beyond the input dimension (non-augmented vs. augmented), and whether the ODE is composed with surrounding affine linear layers (basic vs. general). The four resulting architectures $\textup{NODE}_{\textup{N,b}}$, $\textup{NODE}_{\textup{N}}$, $\textup{NODE}_{\textup{A,b}}$, and $\textup{NODE}_{\textup{A}}$ are introduced in Section~\ref{sec:overview} and we address systematically for each architecture at least two of the three questions (Q1)-(Q3). Figure~\ref{fig:flowmap} visualizes the relationship between the four different neural ODE architectures and collects the most important results regarding their respective embedding capabilities.

\begin{figure}[t]
		\centering
		\vspace{-5mm}
		\begin{overpic}[scale = 0.9,,tics=10]
			{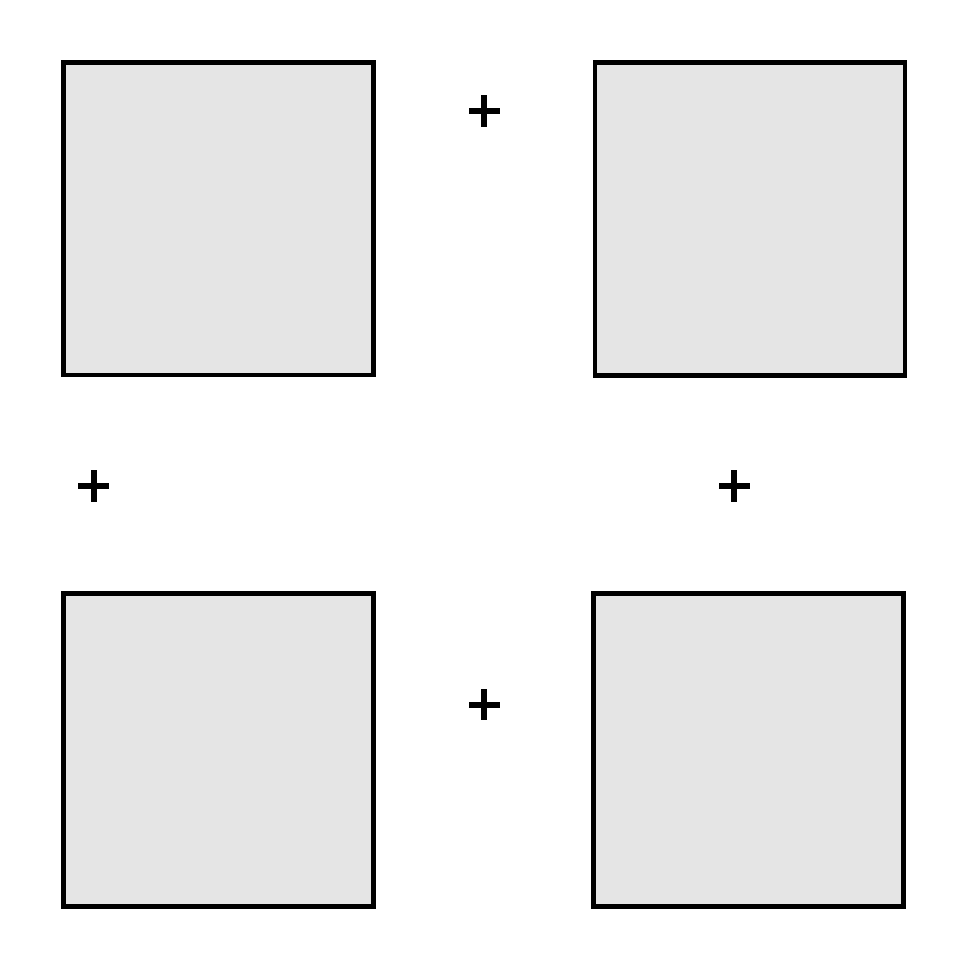}
			\put(8,89.5){\textbf{Section 2.1: $\textup{NODE}_{\textup{N,b}}$}}
			\put(63,89.5){\textbf{Section 2.2: $\textup{NODE}_{\textup{N}}$}}
			\put(8,34.5){\textbf{Section 2.3: $\textup{NODE}_{\textup{A,b}}$}}
			\put(63,34.5){\textbf{Section 2.4: $\textup{NODE}_{\textup{A}}$}}
			\put(14,52){\textbf{phase}}
			\put(14,49){\textbf{augmen-}}
			\put(14,46){\textbf{tation}}
			\put(80,52){\textbf{phase}}
			\put(80,49){\textbf{augmen-}}
			\put(80,46){\textbf{tation}}
			\put(47.5,84){\textbf{two}}
			\put(46.4,81){\textbf{affine}}
			\put(46.2,78){\textbf{linear}}
			\put(46.1,75){\textbf{layers}}
			\put(47.5,22.5){\textbf{two}}
			\put(46.4,19.5){\textbf{affine}}
			\put(46.2,16.5){\textbf{linear}}
			\put(46.1,13.5){\textbf{layers}}
			\put(8,85.5){\small (Q1): Example \ref{ex:minusxembedding1}: $x \mapsto -x$ }
			\put(8,81.5){\small(Q2): Theorem \ref{th:zhang_generalized} }
			\put(8,79){\small(Generalization of Th. \hspace{-3mm} 1}
			\put(8,76.5){\small in~\citet{Zhang2020a})}
			\put(8,72.5){\small(Q2): Theorem \ref{th:orientationpreserving} }
			\put(8,70){\small(orientation-preserving)}
			\put(8,66){\small (Q1): Theorem \ref{th:julia_monome} }
			\put(8,63.5){\small (monomial embedding)}
			\put(63,85.5){\small (Q1): Example \ref{ex:minusxembedding2}: $x \mapsto -x$ }
			\put(63,81.5){\small(Q2): Theorem \ref{th:dupont_nonembeddable} }
			\put(63,79){\small (Generalization of Prop. \hspace{-3mm} 2} 
			\put(63,76.5){\small in~\citet{Dupont2019})}
			\put(63,72.5){\small (Q2): Theorem \ref{th:nocriticalpoint}}
			\put(63,70){\small (critical points)}
			\put(8,30.5){\small (Q1): Example \ref{ex:minusxembedding3}: $x \mapsto -x$ }
			\put(8,26.5){\small (Q3): Theorem \ref{th:suspension}}
			\put(8,24){\small (suspension flow)}
			\put(8,20){\small (Q2): Theorem \ref{th:neg}}
			\put(8,17.5){\small (Morse functions with}
			\put(8,15){\small critical points)}
			\put(8,11){\small (Q3): Theorem \ref{th:twononlinearlayers} }
			\put(8,8.5){\small (embedded suspension flow)}
			\put(63,30.5){\small (Q1): Example \ref{ex:minusxembedding4}: $x \mapsto -x$ }
			\put(63,26.5){\small (Q3): Theorem \ref{th:nodeA}}
			\put(63,24){\small (universal embedding)}
			\put(63,20){\small (Q3): Corollary \ref{cor:approximatepolynomial}}
			\put(63,17.5){\small (universal approximation}
			\put(63,15){\small of continuous functions)}
			\put(36,56.5){\small (Q1): performance in low}
			\put(42.5,54){\small dimensions}
			\put(36,50){\small (Q2): non-embeddable}
			\put(42.5,47.5){\small function classes}
			\put(36,43.5){\small (Q3): universal embedding/}
			\put(42.5,41){\small approximation}
		\end{overpic}
		\vspace{-2mm}
		\caption{Overview of the four neural ODE architectures introduced in Section~\ref{sec:overview} and the most important results regarding their embedding capability.}
		\label{fig:flowmap}
	\end{figure}

    \paragraph{A unified framework for neural ODE architectures.}
    We introduce in Definition~\ref{def:nodeNB} a single general neural ODE architecture that encompasses all four architectures as special cases, making their embedding capabilities directly comparable. In that way, we can explain how each modification, i.e., adding affine linear layers or augmenting the phase space, expands the class of representable functions.

    \paragraph{Structural limitations of non-augmented architectures.} We generalize two foundational non-embeddability results from the literature. Theorem~\ref{th:zhang_generalized} extends Theorem~1 of~\citet{Zhang2020a} to hold under strictly weaker assumptions on both the map $\Psi$ and the vector field $f$. Theorem~\ref{th:dupont_nonembeddable} generalizes Proposition~2 of~\citet{Dupont2019} to the full class $\textup{NODE}_{\textup{N}}$ with arbitrary affine linear layers. Beyond these generalizations, Theorem~\ref{th:orientationpreserving} shows that every map in $\textup{NODE}^k_{\textup{N,b}}$, $k \geq 1$, is an orientation-preserving diffeomorphism, providing a necessary condition for a differentiable map to be embeddable. Together with Theorem~\ref{th:nocriticalpoint} that $\textup{NODE}_{\textup{N}}$ with full-rank weight matrices and $q \leq m$ cannot represent any function with a critical point, these results answer Question~(Q2) for non-augmented architectures and explain their fundamental expressivity limitations.

    \paragraph{Morse-theoretic obstructions for basic augmented neural ODEs.}
    We introduce topological Morse functions as a tool to construct non-embeddable function classes for basic augmented neural ODEs. In Theorem~\ref{th:neg}, combining the local structure of topological Morse functions with the Borsuk-Ulam theorem, we prove that any map $\Psi$, where one component is a topological Morse function with a topologically critical point, cannot be embedded in $\textup{NODE}_{\textup{A,b}}$. This result is particularly significant because, as shown in Corollary~\ref{cor:morsefunctionsdense}, Morse functions are dense in the Banach space $C^k(\overbar{\mathcal{X}},\mathbb{R})$, $k \geq 2$, for bounded $\mathcal{X}$. This answers Question (Q2) for basic augmented architectures, requiring mathematical tools not previously used in the neural ODE literature.

    \paragraph{Suspension flow construction and embedding into Euclidean space.}
    We prove via the classical suspension flow on the mapping torus 
    (Theorem~\ref{th:suspension}) that every diffeomorphism $\Psi \in C^1(\mathcal{X},\mathcal{X})$ can be embedded in the time-$T$ map of a flow on the $(n+1)$-dimensional mapping torus, i.e., with only one additional dimension. To make this construction applicable in machine learning, where Euclidean representations are required, we further show in Theorem~\ref{th:twononlinearlayers} that for smooth diffeomorphisms the mapping torus can be embedded in a $(2n+2)$-dimensional Euclidean space via Whitney's embedding theorem, yielding a general augmented neural ODE architecture with two (possibly nonlinear) layers. This provides a geometrically transparent answer to Question (Q3) for diffeomorphisms with a direct dynamical systems interpretation.

    \paragraph{Universal embedding and approximation for augmented neural ODEs.}
    Theorem~\ref{th:nodeA} shows that the augmented architecture $\textup{NODE}_{\textup{A}}$ with $m \geq n+q$ has the universal embedding property. This answers Question (Q3) for general augmented architectures and immediately implies the universal approximation property. As a complement, Corollary~\ref{cor:approximatepolynomial} establishes universal approximation of continuous functions on compact intervals via an independent argument: by combining explicit vector fields for monomial embeddings (Theorem~\ref{th:julia_monome}), it uses a single vector field that is fixed a priori and independent of the target function, and only the output layer depends on the target function. In contrast to Theorem~\ref{th:nodeA}, where the vector field is constructed from the target map itself, this makes the approximation mechanism transparent for unknown targets that are to be learned from data.

    \subsection*{Structure of the Paper}
	
	In Section~\ref{sec:overview}, the four different neural ODE architectures are introduced, relevant existing results are collected and generalized, with complete proofs provided for all stated results. Furthermore, we state our new theorems that require more complex mathematical arguments, which are postponed to later sections.  In Section~\ref{sec:restrictedembedding} we discuss the restricted embedding problem, i.e., how to choose the vector field of a basic non-augmented neural ODE in order to embed a given map, including explicit autonomous vector fields for monomial maps and a resulting universal approximation result. The following Section~\ref{sec:morsesection} introduces Morse functions, which allow us to define a function class that is non-embeddable in basic augmented neural ODEs. In Section~\ref{sec:suspensionflows}, we prove how to embed a basic augmented neural ODE on a special manifold, called the mapping torus, in a Euclidean space in order to use it in machine learning applications. In all three Sections \ref{sec:restrictedembedding},~\ref{sec:morsesection}, and \ref{sec:suspensionflows}, the mathematical theory is followed by the proofs of the main results. Appendix~\ref{app:julia} derives the explicit vector fields of Section~\ref{sec:restrictedembedding} via Jabotinsky's and Julia's functional equations, and Appendix~\ref{app:onedmorse} collects results on critical points and topological Morse functions in the one-dimensional case. 
	
	\section{Overview and Results}
	\label{sec:overview}
	
	In this section, several common and fundamental neural ODE architectures are introduced. A neural ODE architecture is a composition of functions, whereby one of these functions is the time-$T$ map~$h_x(T)$ of an initial value problem. 
	The different neural ODE architectures introduced here are among the most studied in the literature~\citep[see for example][]{Chen2018,Dupont2019,Kuehn2024,RuizBalet2023,Zhang2020a}. In each case, we build on existing architectures from the literature and extend them in several directions, proving new embedding and approximation results. The mathematical foundations and the proofs of the new theorems can be found in Sections~\ref{sec:restrictedembedding},~\ref{sec:morsesection}, and~\ref{sec:suspensionflows}. \medskip
	
	In this work we study the embedding of $k$-times continuously differentiable maps $\Psi\in C^k(\mathcal{X},\mathbb{R}^{q})$, $k \geq 0$, mapping an input $x \in \mathcal{X} \subset \mathbb{R}^{n}$, $\mathcal{X}$ open, to some output $\Psi(x) \in \mathbb{R}^{q}$, into different neural ODE architectures. Each neural ODE is based on the solution $h_a: \mathcal{I}_a \rightarrow \mathbb{R}^m$ of an $m$-dimensional initial value problem
	\begin{equation} \label{eq:IVP} 
		\frac{\dd h}{\dd t} = f(t,h(t)), \qquad h(0) = a \in \mathcal{A} \subset \mathbb{R}^m, \tag{IVP}
	\end{equation}
	with vector field $f: \Omega \rightarrow \mathbb{R}^m$, $\Omega \subset \R \times \R^m$ open, with maximal time interval of existence $\mathcal{I}_a$ for $a \in \mathcal{A}$. To take into account the dependence on the initial condition, we denote the solution with $h(0) = a$ by $h_a(t)$. As we do not optimize the neural ODE with respect to its parameters, we suppress the dependence of $f$ on its parameters.
    For the study of universal embedding and universal approximation, one asks whether there exists a vector field $f$, without fixing a finite-dimensional parameterization, such that the resulting neural ODE realizes a given target map $\Psi$ exactly or approximately. In this sense, $f$ itself is the free variable, instead of a parameter vector $\theta$ to be optimized. Specific parameterizations of $f$ are then special cases of this general setting, as discussed in the introduction. Throughout this work, we are interested in the existence of an embedding and not how it can be learned. The relationship between general and parameterized vector fields and their implications on universal embedding and universal approximation are discussed in~\citet{Kuehn2024}. \medskip
	
	In the case that the vector field $f$ is continuous in both variables $h$ and $t$, the solution map $h$ is continuously differentiable and \eqref{eq:IVP} is equivalent to the integral formulation
	\begin{equation} \label{eq:INT}
		h(t) = a + \int_0^t f(s,h(s)) \; \dd s. \tag{INT}
	\end{equation}
	Typical activation functions are continuous with respect to the layers $h$, but the weights~$\theta$ can change from layer to layer, leading to possible discontinuities with respect to the time variable $t$. Often the weight function $\theta$ in~\eqref{eq:intro_ODE_1} is chosen piecewise constant to mimic the sudden changes of the weights in ResNet, also in neural ODEs~\citep{Chen2018}. Therefore, it is reasonable to assume that the resulting non-autonomous vector field~$f$ in~\eqref{eq:IVP} is for fixed $t \in \mathcal{I}_a$ continuous in $h$, but to allow for discontinuities in $t$. Hence, we only require the solution~$h$ to fulfill \eqref{eq:INT} and not to be continuously differentiable with respect to the time variable~$t$, which would be necessary to fulfill~\eqref{eq:IVP}.
	
	If a continuous solution $h$ fulfills the integral equation \eqref{eq:INT}, it satisfies the initial value problem~\eqref{eq:IVP} for almost all $t \in \mathcal{I}_a$ and is called a Carathéodory solution of \eqref{eq:IVP}. The Carathéodory conditions guaranteeing the existence of a Carathéodory solution of \eqref{eq:IVP} can be found, for example, in~\citet{Coddington1955}. The conditions require that the vector field $f$ is continuous in~$h$ for each fixed $t$, that $f$ is Borel measurable in $t$ for each fixed $h$, and that $f$ is locally bounded by an integrable function, i.e., for every compact set $K \subset \Omega$ there exists a Lebesgue integrable function~$m_K$ with $\norm{f(t,h)} \leq m_K(t)$ for all $(t,h) \in K$.
	
	Throughout this work, a `solution~$h$ of~\eqref{eq:IVP}' always refers to a solution in the extended sense, including classical and Carathéodory solutions. The following three assumptions on the solution of \eqref{eq:IVP} form the basis for a well-defined and continuous neural ODE architecture.
	
	\begin{assumpA} \label{assA:ode_existence}
		For all $a \in \mathcal{A}$, the solution $h_a:\mathcal{I}_a\rightarrow \mathbb{R}^m$ of the initial value problem \eqref{eq:IVP} contained in the neural ODE is continuous and exists for all $t \in [0,T] \subset \mathcal{I}_a$ with $(t,h_a(t)) \in \Omega$.
	\end{assumpA}
	
	In analogy to the classical Peano ODE solution existence theory~\citep[see for example][]{Hartman2002}, global existence of Carathéodory solutions of~\eqref{eq:IVP} can be proven if the vector field $f$ is bounded. To have a well-defined neural ODE, we have to additionally assume uniqueness of solution curves, such that there exists a unique time-$T$ map for every initial condition. In analogy to the classical Picard-Lindelöf theorem~\citep[see for example][]{Hartman2002}, unique Carathéodory solutions can be proven under the additional assumption of Lipschitz continuity with respect to the variable $h$~\citep{Coddington1955}. Typical activation functions of feed-forward neural networks are globally Lipschitz continuous, leading to a vector field of \eqref{eq:IVP}, which is globally Lipschitz continuous with respect to the second variable. 
	
	\begin{assumpA} \label{assA:ode_unique}
		For all $a \in \mathcal{A}$, the solution $h_a:\mathcal{I}_a\rightarrow \mathbb{R}^m$ of the initial value problem \eqref{eq:IVP} contained in the neural ODE architecture is unique for all $t \in \mathcal{I}_a$, i.e., if $h_{a_1},h_{a_2}$ are solutions of  \eqref{eq:IVP} with $a_1 \neq a_2$, then $h_{a_1}(t) \neq h_{a_2}(t)$ for all $t \in \mathcal{I}_{a_1}\cap \mathcal{I}_{a_2}$.
	\end{assumpA}

	The last assumption we need for the upcoming results is continuous dependence on initial data. For vector fields fulfilling the assumptions of the Picard-Lindelöf theorem or the Carathéodory conditions, Assumption~\ref{assA:continuous_dependence} follows from the uniqueness Assumption~\ref{assA:ode_unique}~\citep{Coddington1955,Hartman2002}.
	
	\begin{assumpA} \label{assA:continuous_dependence} 
		For each fixed $t\in [0,T]$, the time-$t$ map $\mathcal{A} \rightarrow \mathbb{R}^m$, $a \mapsto h_a(t)$ of the solution of the initial value problem \eqref{eq:IVP} contained in the neural ODE architecture is a continuous function of the initial data $a\in \mathcal{A}$. 
	\end{assumpA}
	
	The general neural ODE architectures we consider allow for an affine linear transformation before and after the time-$T$ map of the initial value problem \eqref{eq:IVP}, resulting in the neural network
	\begin{equation} \label{eq:NODE}
		\Phi: \mathcal{X} \rightarrow \mathbb{R}^q, \qquad \Phi(x)\coloneqq \tilde{\lambda}( h_{\lambda(x)}(T)) = \widetilde{W} \cdot h_{Wx+b}(T) + \tilde{b} = \widetilde{W} \cdot h_{a}(T) + \tilde{b}.
	\end{equation} 
	Here, the first affine linear layer $\lambda: \mathbb{R}^n \rightarrow \mathbb{R}^m$ is represented by a matrix $W \in \mathbb{R}^{m \times n}$ and a vector $b \in \mathbb{R}^m$ and the second affine linear layer $\tilde{\lambda}: \mathbb{R}^m\rightarrow \mathbb{R}^q$ is represented by a matrix $\widetilde{W} \in \mathbb{R}^{q \times m}$ and a vector $\tilde{b} \in \mathbb{R}^q$. The activated state $a$, which defines the initial condition of~\eqref{eq:IVP}, is calculated by $a = \lambda(x) = Wx+b \in \mathbb{R}^m$ and the set of initial conditions  of~\eqref{eq:IVP} is defined to be $\mathcal{A} \coloneqq \lambda(\mathcal{X})$.  
	
	The three Assumptions \ref{assA:ode_existence}, \ref{assA:ode_unique}, and \ref{assA:continuous_dependence} form the most general framework to have a well-defined and continuous neural ODE $\Phi$ defined by \eqref{eq:NODE}. If the vector field $f$ of \eqref{eq:IVP} fulfills additional regularity assumptions, the neural ODE architecture \eqref{eq:NODE} also becomes more regular, as the following lemma shows. Here, for $\Omega \subset \R \times \R^m$ open, $C^{0,k}(\Omega,\R^m)$ denotes the space of functions $f: \Omega \rightarrow \R^m$ such that $f$ and its partial derivatives with respect to the state $h$ up to order $k$ exist and are jointly continuous in time and state on $\Omega$.
	
	\begin{lemma}[Neural ODE Regularity]\label{lem:node_regularity}
		Let $f \in C^{0,k}(\Omega,\mathbb{R}^m)$ with $\Omega \subset \R \times \R^m$ open, $k \geq 1$, and let Assumption \ref{assA:ode_existence} be fulfilled. Then the Assumptions \ref{assA:ode_unique} and \ref{assA:continuous_dependence} follow and it holds for the map $\Phi: \mathcal{X}\rightarrow \R^q$ defined by \eqref{eq:NODE} that $\Phi\in C^k(\mathcal{X},\mathbb{R}^q)$.
	\end{lemma}
	
	\begin{proof}
		The result follows directly from standard ODE results~\citep[see for example][]{Hale1980} and the smoothness of the affine linear layers $\lambda$ and $\tilde \lambda$.
	\end{proof}

	Under the three Assumptions \ref{assA:ode_existence}, \ref{assA:ode_unique}, and \ref{assA:continuous_dependence}, we can define the class of neural ODE architectures considered in this work, depending on the regularity of the vector field.

	\begin{definition}[Neural ODE] \label{def:nodeNB}
		The set of all neural ODE architectures $\Phi: \mathcal{X} \rightarrow \mathbb{R}^q$, $\mathcal{X}\subset \mathbb{R}^n$ open, defined by~\eqref{eq:NODE}, where the solution map of the underlying initial value problem \eqref{eq:IVP} fulfills the Assumptions \ref{assA:ode_existence}, \ref{assA:ode_unique}, and \ref{assA:continuous_dependence}, is denoted by $\textup{NODE}(\mathcal{X},\mathbb{R}^q) \subset C^0(\mathcal{X},\mathbb{R}^q)$. 
		
		If the underlying vector field additionally fulfills $f \in C^{0,k}(\Omega,\mathbb{R}^m)$, $k \geq 1$, on the open set~$\Omega$ of~\eqref{eq:IVP}, then the neural ODE architecture is denoted by $\textup{NODE}^k(\mathcal{X},\mathbb{R}^q) \subset C^k(\mathcal{X},\mathbb{R}^q)$. For $k = 0$, we set $\textup{NODE}^0(\mathcal{X},\mathbb{R}^q) \coloneqq \textup{NODE}(\mathcal{X},\mathbb{R}^q)$.
	\end{definition}
	
	In the upcoming sections, we study neural ODE architectures of the form \eqref{eq:NODE} with respect to the property of universal embedding and their implications on universal approximation. The parameters of a neural ODE $\Phi \in \textup{NODE}(\mathcal{X},\R^q)$ are the choice of the vector field~$f$ and the weights and biases $W$, $\widetilde{W}$, $b$, and~$\tilde{b}$ of the affine linear transformations. With this general framework in place, we now introduce the four specific architectures studied in this work. The different types of architectures are distinguished by the dimensions $n$,~$m$, and $q$, and the explicit form of the affine linear layers $\lambda$, $\tilde{\lambda}$. We distinguish between non-augmented architectures, where $m \leq n$, and augmented architectures with $m>n$. In the special case that the affine linear layers $\lambda$, $\tilde{\lambda}$ are identity transformations, we call the corresponding architecture `basic', otherwise `general'. The architectures introduced are basic non-augmented neural ODEs in Section~\ref{sec:nodeNB}, general non-augmented neural ODEs in Section~\ref{sec:nodeN}, basic augmented neural ODEs in Section~\ref{sec:nodeAB}, and the combination of both, general augmented neural ODEs in Section~\ref{sec:nodeA}.

	\subsection{Basic Non-Augmented Neural ODEs}
	\label{sec:nodeNB}
	
	A basic non-augmented neural ODE $\Phi$ is based on the $m$-dimensional initial value problem~\eqref{eq:IVP} with  $h(0) = x \in \mathcal{X} \subset \mathbb{R}^m$. The resulting neural ODE architecture is defined to be the map from the initial condition $x$ to the time-$T$ map, i.e., $\Phi(x) = h_x(T)$. In terms of the general architecture \eqref{eq:NODE}, basic non-augmented neural ODEs are defined as follows:
	
	\begin{definition}[Basic Non-Augmented Neural ODE]
		The set of all basic non-augmented neural ODE architectures $\Phi\in \textup{NODE}(\mathcal{X},\mathbb{R}^n)$, $\mathcal{X}\subset \mathbb{R}^n$ open, defined by \eqref{eq:NODE} with $n = m = q$ and identity transformations $\lambda(x) = \tilde{\lambda}(x) = x$ for all $x \in \mathbb{R}^n$ is denoted by $\textup{NODE}_{\textup{N,b}}(\mathcal{X},\mathbb{R}^n)$. 
	\end{definition}

	The architecture $\textup{NODE}_{\textup{N,b}}(\mathcal{X},\mathbb{R}^n)$ can only be used to embed maps $\Psi: \mathcal{X} \rightarrow \mathbb{R}^n$, $\mathcal{X} \subset \mathbb{R}^n$ where the input and output dimension agree with the dimension of the ODE, see Figure~\ref{fig:nodeNB}. As the phase space of the ODE is not augmented, the problem of embedding a map $\Psi: \mathcal{X} \rightarrow \mathbb{R}^n$ in the neural ODE architecture $\textup{NODE}_{\textup{N,b}}(\mathcal{X},\mathbb{R}^n)$ is also called the restricted embedding problem~\citep{Zhang2020a}.

	Due to the topological structure of unique solution curves of ODEs, the class of functions that can be embedded in basic non-augmented neural ODEs is restricted. In the following example, a simple one-dimensional non-embeddable map is given.

	\begin{example} \label{ex:minusxembedding1}
		The map $\Psi: \mathbb{R} \rightarrow \mathbb{R}$, $x \mapsto -x$ cannot be embedded in the neural ODE architecture $\textup{NODE}_{\textup{N,b}}(\mathbb{R},\mathbb{R})$: as by Assumptions \ref{assA:ode_existence} and \ref{assA:ode_unique}, the continuous solutions of~\eqref{eq:IVP} are unique, solution curves do not cross. This is a contradiction to the fact that solution curves going from $h_0(0) = 0$ to $h_0(T) = 0$ and from $h_{x^*}(0) = x^*$ to $h_{x^*}(T) = -x^*$ for some $x^* \neq 0$ need to cross by the intermediate value theorem~\citep[see for example][]{Forster2017}. The setting is visualized in Figure~\ref{fig:node_example}.
	\end{example}
	
	This counterexample can be generalized to higher dimensions, contributing to Question~(Q2). The following theorem is based on ideas of Theorem 1 in~\citet{Zhang2020a}, with weakened assumptions on the map~$\Psi$ to embed and the vector field~$f$ of~\eqref{eq:IVP}. 

	\begin{theorem} \label{th:zhang_generalized}
    Let $\mathcal{Z} \subset \mathbb{R}^n$ be a set subdividing $\mathbb{R}^n$ into at least two, but finitely many disjoint, connected subsets $\mathcal{C}_i$, $i \in \{1,2,\ldots,K\}$, such that every curve from $x \in \mathcal{C}_i$ to $y \in \mathcal{C}_j$, $i \neq j$, has to intersect the set $\mathcal{Z}$. Consider a continuous map $\Psi: \mathcal{X} \rightarrow \mathbb{R}^n$, $\mathcal{Z} \subset \mathcal{X} \subset \mathbb{R}^n$, $\mathcal{X}$ open, which is the identity transformation on $\mathcal{Z}$ (i.e., $\Psi(x) = x$ for $x \in \mathcal{Z}$), and for which there exists a point $x^\ast \in \mathcal{X} \, \cap \, \mathcal{C}_{i^\ast}$ being mapped to $\Psi(x^\ast) \in \mathcal{C}_{j^\ast}$ with $i^\ast \neq j^\ast$. Then the map $\Psi$ cannot be embedded in the neural ODE architecture $\textup{NODE}_{\textup{N,b}}(\mathcal{X},\mathbb{R}^n)$, $\mathcal{X}\subset \R^n$, by an autonomous vector field~$f$, for which the solutions of~\eqref{eq:IVP} are unique for all initial values in the domain of $f$.
\end{theorem}

\begin{proof}
    Suppose there exists an embedding of the continuous map $\Psi$ in the neural ODE architecture $\textup{NODE}_{\textup{N,b}}(\mathcal{X},\R^n)$, so that $h_x(T) = \Psi(x)$ for all $x \in \mathcal{X}$. By the assumptions of the theorem, it holds $h_{x^\ast}(0) = x^\ast \in \mathcal{C}_{i^\ast}$, \mbox{$h_{x^\ast}(T) = \Psi(x^\ast) \in \mathcal{C}_{j^\ast}$}, such that $h_{x^\ast}|_{[0,T]}$ is a curve from $\mathcal{C}_{i^\ast}$ to $\mathcal{C}_{j^\ast}$ and the separation property of $\mathcal{Z}$ implies $z \coloneqq h_{x^\ast}(\tau) \in \mathcal{Z}$ for some $\tau \in (0,T)$. As $z \in \mathcal{Z} \subset \mathcal{X}$, the solution $h_z$ exists on $[0,T]$, and as $f$ is autonomous, $h_{x^\ast}$ can be continued by $h_{x^\ast}(\tau + s) = h_z(s)$ for $s \in [0,T]$. As $\Psi$ is an identity transformation on~$\mathcal{Z}$, it holds $h_{x^\ast}(\tau) = \Psi(h_{x^\ast}(\tau)) =$ $h_z(T) =$ $h_{x^\ast}(\tau+T)$, i.e., the trajectory starting at $h_{x^\ast}(\tau)$ forms a closed loop~$\gamma$ ending at the same point in $\mathcal{Z}$ after the time $T$. The function $p(t) \coloneqq h_z(t+T-\tau)$, $t \in [0,\tau]$, is a solution of the autonomous ODE with $p(\tau) = h_z(T) = z = h_{x^\ast}(\tau)$ and initial value $y \coloneqq p(0) = h_z(T-\tau) \in \gamma$. As the solutions are unique for all initial values in the domain of $f$, it follows $x^\ast = y \in \gamma$. As every point of $\gamma$ is $T$-periodic, it follows $h_{x^\ast}(T) = h_{x^\ast}(0)$, which is a contradiction to $h_{x^\ast}(0) = x^\ast \in \mathcal{C}_{i^\ast}$ and $h_{x^\ast}(T) = \Psi(x^\ast) \in \mathcal{C}_{j^\ast}$.
\end{proof}

	\begin{figure}
		\centering
		\begin{minipage}[c]{0.33\textwidth}
			\centering
			\begin{overpic}[scale = 0.35,,tics=10]
				{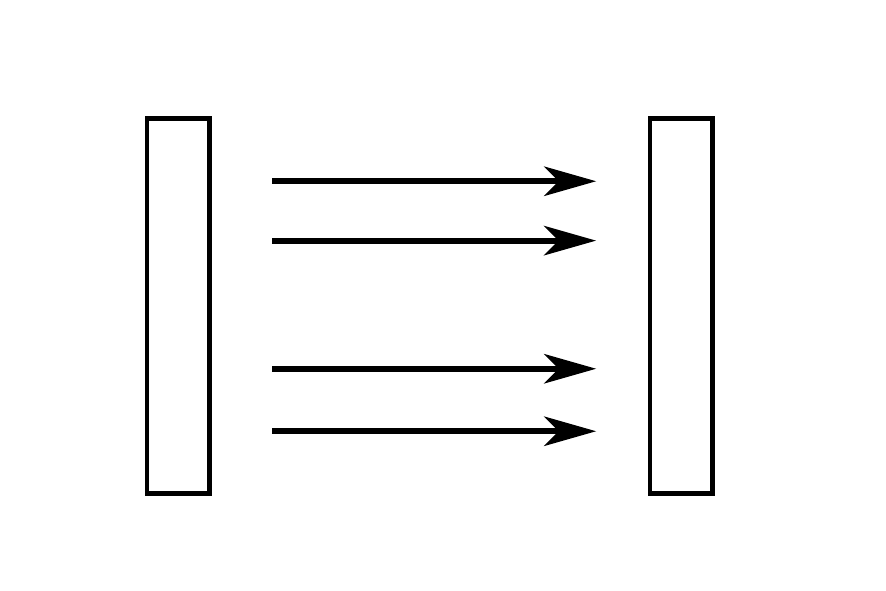}
				\put(8,1){$\mathcal{X}\subset \R^n$}
				\put(18,58){$x$}
				\put(57,58){$\Phi(x) = h_x(T)$}
				\put(73,1){$\R^n$}
				\put(40,30){$\textup{ODE}$}
			\end{overpic}
			\vspace{2mm}
		\end{minipage}
		\begin{minipage}[c]{0.08\textwidth}
			\textcolor{white}{.}
		\end{minipage}
		\begin{minipage}[c]{0.57\textwidth}
			\centering
			\begin{overpic}[scale = 0.35,,tics=10]
				{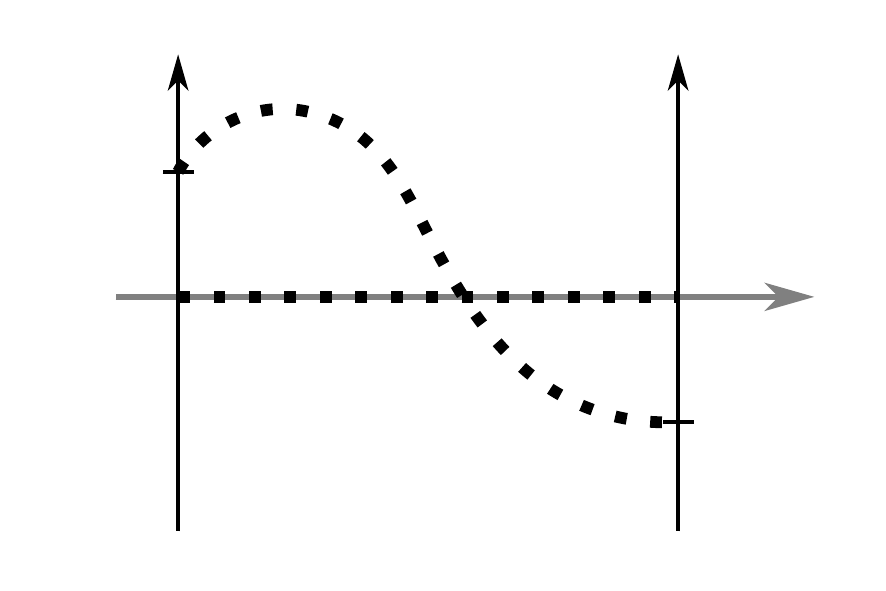}
				\put(74,-1){$\R$}
				\put(18,-1){$\R$}
				\put(94,31){$\textcolor{gray}{t}$}
				\put(79,35){$\textcolor{gray}{T}$}
				\put(22,35){$\textcolor{gray}{0}$}
				\put(11,46){$x^\ast$}
				\put(80,17){$-x^\ast$}
			\end{overpic}
			\vspace{2mm}
		\end{minipage}
		\begin{minipage}[t]{0.33\textwidth}
			\caption{Structure of a basic non-augmented neural ODE $\Phi \in \textup{NODE}_{\textup{N,b}}(\mathcal{X},\R^n)$.}
			\label{fig:nodeNB}
		\end{minipage}
		\begin{minipage}[t]{0.08\textwidth}
			\textcolor{white}{.}
		\end{minipage}
		\begin{minipage}[t]{0.57\textwidth}
			\caption{The map $\Psi(x) = -x$ cannot be embedded in  $\textup{NODE}_{\textup{N,b}}(\R,\R)$. The dotted lines represent possible trajectories from  $h_0(0) = 0$ to $h_0(T) = 0$ and from $h_{x^*}(0) = x^*$ to $h_{x^*}(T) = -x^*$ for some $x^* \neq 0$.}
			\label{fig:node_example}
		\end{minipage}
	\end{figure}

	\begin{remark}\label{rem:autonomy}
		The autonomy of the vector field $f$ is essential for the proof of Theorem~\ref{th:zhang_generalized}: it identifies the time-$T$ map $\Psi$ started at time $0$ with the evolution from time $\tau$ to time $\tau+T$, which does not hold for general non-autonomous vector fields. The assumption cannot be removed by the standard reduction of a non-autonomous ODE to an autonomous ODE, in which the time is appended to the state as an additional coordinate $h_{n+1}$ with $h_{n+1}'(t) = 1$ and $h_{n+1}(0) = 0$. First, this reduction enlarges the phase space to $\R^{n+1}$ while the input remains in $\R^n$, such that the resulting architecture is augmented and no longer lies in $\textup{NODE}_{\textup{N,b}}(\mathcal{X},\R^n)$. Second, the appended coordinate is strictly increasing along every trajectory, such that the autonomous system has no periodic orbit: the point $(z,\tau)$ is mapped after time $T$ to $(z,\tau+T) \neq (z,\tau)$, and the closed loop $\gamma$ on which the proof relies does not exist in the enlarged phase space. 
		
		For $n = 1$, the autonomy assumption can be dropped: by the separation property of $\mathcal{Z}$, there exists $z \in \mathcal{Z}$ strictly between $x^\ast$ and $\Psi(x^\ast)$, such that $h_{x^\ast} - h_z$ changes sign on $[0,T]$ as $h_z(T) = z$. By the intermediate value theorem, the solutions $h_{x^\ast}$ and $h_z$ intersect, contradicting Assumption~\ref{assA:ode_unique}. This is the argument of Example~\ref{ex:minusxembedding1}, which relies on the order structure of $\mathbb{R}$ and has no analogue for $n \geq 2$.
	\end{remark}

	The one-dimensional map $\Psi: \mathbb{R} \rightarrow \mathbb{R}$, $x \mapsto -x$ of Example~\ref{ex:minusxembedding1} is the special case $n = 1$ of Theorem~\ref{th:zhang_generalized} with $\mathcal{X} = \mathbb{R}$, $\mathcal{C}_1 = (-\infty,0)$, $\mathcal{Z} = \{0\}$ and $\mathcal{C}_2 = (0,\infty)$, where by Remark~\ref{rem:autonomy} the autonomy of the vector field~$f$ is not needed. Example~\ref{ex:minusxembedding1} and Theorem~\ref{th:zhang_generalized} only use topological arguments to show the non-embeddability of certain functions. Theorem~\ref{th:zhang_generalized} provides a topological obstruction to embedding, requiring an autonomous vector field~$f$, but no regularity beyond continuity. Under the additional differentiability assumption $k \geq 1$, a stronger and more general obstruction can be derived, showing that every map embeddable in $\textup{NODE}^k_\textup{N,b}(\mathcal{X},\R^n)$ has to be an orientation-preserving diffeomorphism, further addressing Question (Q2).

	\begin{theorem}[Orientation-Preserving Diffeomorphism] \label{th:orientationpreserving}
		Let $\Psi \in C^1(\mathcal{X},\R^n)$, $\mathcal{X}\subset \R^n$ open, with $p \in \mathcal{X}$, such that $\det(J_\Psi(p))\leq0$. Then $\Psi$ cannot be embedded in the basic non-augmented neural ODE architecture $\textup{NODE}^k_\textup{N,b}(\mathcal{X},\R^n)$, $k \geq 1$.
	\end{theorem}
	
	\begin{proof}
		By the proof of Proposition 4.10 in~\citet{Kuehn2024}, the Jacobian matrix $\frac{\partial h_a(T)}{\partial a}\in \R^{m \times m}$ of the initial value problem \eqref{eq:IVP} has for every $a\in \mathcal{A}$ full rank $m$. Liouville's formula~\citep[see for example][]{Chicone2006} additionally implies that $\det\left(\frac{\partial h_a(T)}{\partial a}\right)>0$ for all $a \in \mathcal{A}$. As the affine linear layers of a basic non-augmented neural ODE are identity transformations, it holds $n = m$ and $a =x$, such that $\Phi(x) = h_x(T)$ and $J_\Phi(x) = \frac{\partial h_x(T)}{\partial x}$. Hence, $\Phi:\mathcal{X}\rightarrow \Phi(\mathcal{X})$ is an orientation-preserving diffeomorphism, i.e., for all $x\in \mathcal{X}$ it holds $\det(J_\Phi(x))>0$, which implies the result.
	\end{proof}	

    \begin{remark}
        The orientation-preserving property of Theorem~\ref{th:orientationpreserving} can alternatively be understood via homotopy theory. The flow $h(t): \mathcal{X} \rightarrow \mathbb{R}^n$ defines a homotopy between the identity map $h(0) = \textup{id}$ and the time-$T$ map $\Phi = h(T)$. Since homotopic maps have the same topological degree, and the identity has degree one, the time-$T$ map $\Phi$ has degree one for any $k \geq 0$. In the case $k \geq 1$, the topological degree additionally coincides with the sign of $\det(J_\Phi(x))$ at regular points $x \in \mathcal{X}$, agreeing with $\det(J_\Phi(x)) > 0$ for all $x \in \mathcal{X}$, consistent with the conclusion of Theorem~\ref{th:orientationpreserving} \citep{Milnor1963}.
    \end{remark}

    In Section~\ref{sec:restrictedembedding}, we continue to discuss the restricted embedding problem in the case that an embedding is possible. We first give in Theorem~\ref{th:monotoneembedding} a constructive condition for embedding strictly increasing functions in one-dimensional basic non-augmented neural ODEs. In Theorem~\ref{th:julia_monome}, we derive explicit vector fields embedding monomial maps $\Psi(x) = cx^\alpha$ with $c, x, \alpha \in \R_{>0}$, contributing to Question~(Q1). In Corollary~\ref{cor:approximatepolynomial}, we combine the results of Theorem~\ref{th:julia_monome} to approximate continuous functions on compact intervals with a general augmented neural ODE as introduced in Section~\ref{sec:nodeA}.

	\subsection{General Non-Augmented Neural ODEs}
	\label{sec:nodeN}
	
	We have seen in Section~\ref{sec:nodeNB} that basic non-augmented neural ODEs are restricted to embed maps, where the input and the output dimension are the same. Furthermore, it was shown that the embedding capabilities of basic non-augmented neural ODEs are, in the differentiable case, restricted to orientation-preserving diffeomorphisms. To embed general maps $\Psi: \mathcal{X} \rightarrow \mathbb{R}^{q}$, $\mathcal{X} \subset \mathbb{R}^{n}$, two affine linear transformations $\lambda: \mathbb{R}^n\rightarrow \mathbb{R}^m$ and $\tilde\lambda: \mathbb{R}^m\rightarrow \mathbb{R}^q$ can be added to an initial value problem in dimension~$m$, leading to the neural ODE architecture \eqref{eq:NODE}. In the case of a scalar output $q = 1$, this neural ODE architecture is often used for regression and classification tasks~\citep{Dupont2019}. If the dimension $m$ of the underlying initial value problem does not exceed the input dimension~$n$, the architecture is called non-augmented.
	
	\begin{definition}[Non-Augmented Neural ODE]
		The set of all non-augmented neural ODE architectures $\Phi\in \textup{NODE}(\mathcal{X},\mathbb{R}^q)$, $\mathcal{X}\subset \mathbb{R}^n$ open, defined by \eqref{eq:NODE} with vector field dimension $m \leq n$ is denoted by $\textup{NODE}_{\textup{N}}(\mathcal{X},\mathbb{R}^q)$.
	\end{definition}
	
	The structure of the neural ODE architecture $\textup{NODE}_\textup{N}(\mathcal{X},\R^q)$ with two affine linear layers is visualized in Figure \ref{fig:nodeN}. 
	
    Even a single additional non-identity affine linear transformation allows one to embed maps that cannot be embedded in basic non-augmented neural ODEs. We demonstrate this for the map of Example~\ref{ex:minusxembedding1}, illustrating the impact on Question (Q1). 
	
	\begin{example}\label{ex:minusxembedding2}
		The map $\Psi: \mathbb{R} \rightarrow \mathbb{R}$, $x \mapsto -x$ can be embedded in the neural ODE architecture $\textup{NODE}_{\textup{N}}(\R,\R)$ by choosing the vector field $f \equiv 0$ in~\eqref{eq:IVP}, such that for $x \in \mathbb{R}$ it holds $h_x(T) = x$. For the affine linear transformations $\lambda$, $\tilde{\lambda}$, we  choose $\lambda: \R\rightarrow \R$, $ x \mapsto x$ and $\tilde \lambda: \R\rightarrow \R$, $x \mapsto -x$, resulting in the neural ODE 
		\begin{equation*}
			\Phi(x) = \tilde{\lambda}(h_{\lambda(x)}(T)) = - x.
		\end{equation*}
	\end{example}
	
	Based on the idea of the proof of Proposition 2 in~\citet{Dupont2019}, the following theorem shows that there exist continuous functions $\Psi: \mathbb{R}^n \rightarrow \mathbb{R}^{q}$, which cannot be embedded in non-augmented neural ODEs, contributing to Question (Q2).  Proposition~2 in~\citet{Dupont2019} is restricted to non-augmented architectures, where the first affine linear layer $\lambda$ is the identity transformation and the second affine linear layer $\tilde \lambda$ is a linear function with $\tilde b = 0$. We additionally generalize the result by weakening the assumptions on the map~$\Psi$ to embed and on the vector field~$f$.

	\begin{theorem}\label{th:dupont_nonembeddable}
		Let $\Psi: \mathcal{X} \rightarrow \mathbb{R}^{q}$, $\mathcal{X}\subset \mathbb{R}^n$ open, be a continuous map and $\mathcal{U},\mathcal{V},\mathcal{W}$ be non-empty subsets of $\mathbb{R}^n$ with $\mathcal{V}$ bounded, $\mathcal{U} \subset \mathcal{V} \subset \mathcal{X}$, $\partial \mathcal{V} \subset \mathcal{W} \subset \mathcal{X}$ and $\mathcal{U} \cap \mathcal{W} = \varnothing$, such that 
		\begin{equation*}
			\begin{cases}
				\Psi_i(x) >c \qquad  \textup{ if } x \in \mathcal{U}, \\
				\Psi_i(x) <c \qquad \textup{ if } x \in \mathcal{W},
			\end{cases}
		\end{equation*}
		for some constant $c \in \mathbb{R}$ and some index $i \in \{1,\ldots,q\}$. Here, $\partial \mathcal{V}$ denotes the boundary of~$\mathcal{V}$ in $\R^n$ and $\Psi_i$ the $i$-th component of the map $\Psi$. Then the map $\Psi$ cannot be embedded in the neural ODE architecture $\textup{NODE}_{\textup{N}}(\mathcal{X},\R^q)$.
	\end{theorem}

	\begin{figure}
		\centering
		\begin{minipage}[c]{0.5\textwidth}
			\centering
			\begin{overpic}[scale = 0.3,,tics=10]
				{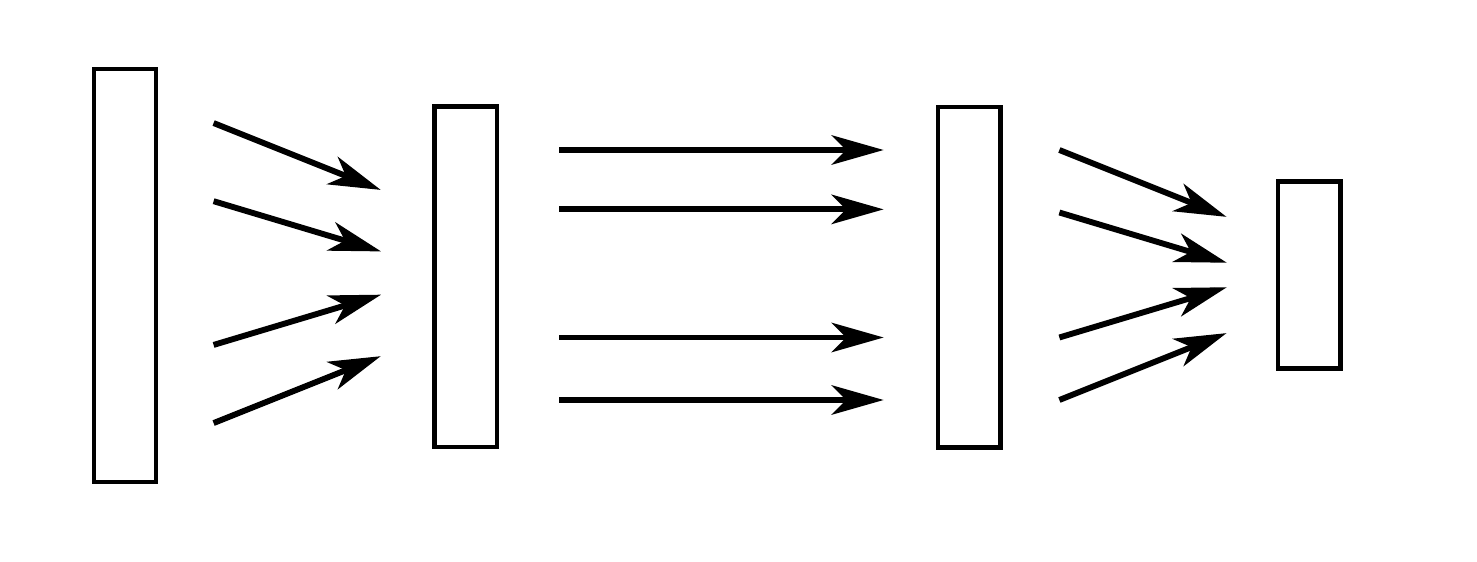}
				\put(1,0.5){$\mathcal{X}\subset \R^n$}
				\put(24.5,2){$\mathcal{A}\subset \R^m$}
				\put(7,36){$x$}
				\put(27,34){$\lambda(x)$}
				\put(58,34){$h_{\lambda(x)}(T)$}
				\put(63,2){$\R^m$}
				\put(42.5,18){$\textup{ODE}$}
				\put(84,28.5){$\Phi(x)$}
				\put(86.5,7){$\R^q$}
			\end{overpic}
			\vspace{1mm}
		\end{minipage}
		\begin{minipage}[c]{0.08\textwidth}
			\textcolor{white}{.}
		\end{minipage}
		\begin{minipage}[c]{0.4\textwidth}
			\centering
			\begin{overpic}[scale = 0.9,,tics=10]
				{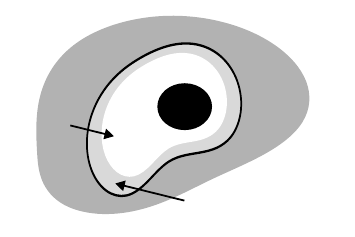}
				\put(38,38){$\mathcal{U}$}
				\put(80,18){\textcolor{gray}{$\mathcal{W}$}}
				\put(67,47){$\partial \mathcal{V}$}
				\put(13,29){\textcolor{white}{$\mathcal{V}$}}
				\put(54,7){\textcolor{lightgray}{$\mathcal{V}\cap \mathcal{W}$}}
			\end{overpic}
			\vspace{1mm}
		\end{minipage}
		\begin{minipage}[t]{0.52\textwidth}
			\caption{Structure of a general non-augmented neural ODE  $\Phi \in \textup{NODE}_\textup{N}(\mathcal{X},\R^q)$.}
			\label{fig:nodeN}
		\end{minipage}
		\begin{minipage}[t]{0.08\textwidth}
			\textcolor{white}{.}
		\end{minipage}
		\begin{minipage}[t]{0.38\textwidth}
			\caption{Example of possible two-dimensional sets  $\mathcal{U}$, $\mathcal{V}$, and $\mathcal{W}$ of Theorem \ref{th:dupont_nonembeddable}.}
			\label{fig:dupont_nonembeddable}
		\end{minipage}
	\end{figure}

	\begin{proof}
		Suppose there exists a neural ODE $\Phi \in \textup{NODE}_{\textup{N}}(\mathcal{X},\R^q)$ embedding the map $\Psi$, then it holds $$\Psi(x) = \widetilde{W} \cdot h_{Wx+b}(T) + \tilde{b} \quad \textup{for all } x \in \mathcal{X},$$ with $n \geq m$, weight matrices $W \in \R^{m \times n}$, $\widetilde{W} \in \R^{q \times m}$, biases $b\in \R^m$, $\tilde{b}\in \R^q$ and time-$T$ map $h_a(T) \in \mathbb{R}^m$ of the initial value problem~\eqref{eq:IVP}. It is sufficient to restrict the proof to the case $m = n$: if no neural ODE $\Phi \in \textup{NODE}_{\textup{N}}(\mathcal{X},\R^q) $ with vector field dimension $m=n$ can embed the given map~$\Psi$, then also no neural ODE with $m<n$ can embed~$\Psi$, as every vector field with $m<n$ can be written as a vector field in dimension $m = n$ by adding zero components to the vector field and to the corresponding entries of the weight matrices $W$ and~$\widetilde{W}$, without changing the input-output map~$\Phi$.
		
		In the following, we additionally assume that $\textup{rank}(W) = n$. If $\textup{rank}(W) < n$, then for fixed $u \in \mathcal{U}$, there exists a vector $\alpha \in \R^n \setminus \{0\}$, defining a one-dimensional affine line $L_\alpha \coloneqq \{u + \mu \cdot \alpha: \mu \in \R\}$, such that $Wu = Wv$ for every $v \in L_\alpha$. As $\mathcal{U} \subset \mathcal{V}$ and $\mathcal{V}$ is bounded, the affine line $L_\alpha$ has to intersect the boundary $\partial \mathcal{V}$. Let $v^\ast \in \partial \mathcal{V} \cap L_\alpha \subset \mathcal{W}$, then $Wu = W v^\ast$, which implies that $\Phi(u) = \Phi(v^\ast)$, which contradicts the assumption that $\Phi_i(u)>c$ and $\Phi_i(v^\ast)<c$.
		
		Assumptions~\ref{assA:ode_existence}, \ref{assA:ode_unique}, and \ref{assA:continuous_dependence} imply that the time-$T$ map $a \mapsto h_a(T)$ is continuous and injective on $\mathcal{A} = \lambda(\mathcal{X})$ with image $\mathcal{B} \coloneqq \{h_a(T): a \in \mathcal{A}\}$. As we assume $m = n$ and $\textup{rank}(W) = n$, the affine linear transformation $\lambda: \mathcal{X} \rightarrow \lambda(\mathcal{X})$, $\lambda(x) = Wx+b$, is a homeomorphism and $\mathcal{A}$ is open. As the set $\mathcal{B}$ is open and the time-$T$ map $\mathcal{A} \rightarrow \mathcal{B}$, $a \mapsto h_a(T)$, is a homeomorphism, the composition $h_{\lambda(\cdot)}(T): \mathcal{X} \rightarrow \mathcal{B}$ is a homeomorphism~\citep[see for example][]{Hatcher2002}. By the assumptions of the theorem, the interior $\textup{int}(\mathcal{V})$ of $\mathcal{V}$ is non-empty, otherwise it would hold $\mathcal{U} \subset \mathcal{V} \subset \partial \mathcal{V} \subset \mathcal{W}$, which contradicts $\mathcal{U} \cap \mathcal{W} = \varnothing$. Furthermore, as $\mathcal{U} \cap \partial\mathcal{V} \subset \mathcal{U} \cap \mathcal{W} = \varnothing$, it holds $\mathcal{U} \subset \textup{int}(\mathcal{V})$. As homeomorphisms in~$\mathbb{R}^m$ map interiors of sets to interiors and boundaries to boundaries~\citep[see][]{Armstrong1983}, it holds for $w \in \partial \mathcal{V}$ that $$h_{\lambda(w)}(T) \in \partial h_{\lambda(\mathcal{V})}(T), \qquad h_{\lambda(\mathcal{V})}(T) \coloneqq \{h_{\lambda(v)}(T): v \in \mathcal{V}\},$$ and for $u \in \mathcal{U} \subset \textup{int}(\mathcal{V})$ that $h_{\lambda(u)}(T) \in \textup{int}(h_{\lambda(\mathcal{V})}(T))$, hence $h_{\lambda(\mathcal{U})}(T) \subset \textup{int}(h_{\lambda(\mathcal{V})}(T))$. As $\overbar{\mathcal{V}} = \mathcal{V} \cup \partial\mathcal{V} \subset \mathcal{X}$ is compact and $h_{\lambda(\cdot)}(T)$ is continuous, the set $h_{\lambda(\mathcal{V})}(T)$ is bounded, such that every $\bar{u} \in h_{\lambda(\mathcal{U})}(T)$ can be written as a convex combination of two boundary points $\bar{w}_1,\bar{w}_2 \in \partial h_{\lambda(\mathcal{V})}(T)$:
		\begin{equation*}
			\bar{u} = \mu \bar w_1 + (1-\mu) \bar w_2 \qquad \textup{with } \mu\in (0,1).
		\end{equation*}
		As $h_{\lambda(\cdot)}(T): \mathcal{X} \rightarrow \mathcal{B}$ is a homeomorphism, there exist $u \in \mathcal{U}$ with $h_{\lambda(u)}(T) = \bar{u}$ and $w_1,w_2 \in \partial\mathcal{V}$ with $h_{\lambda(w_1)}(T) = \bar{w}_1$, $h_{\lambda(w_2)}(T) = \bar{w}_2$ yielding
		\begin{equation*}
			h_{\lambda(u)}(T) = \mu h_{\lambda(w_1)}(T) + (1-\mu) h_{\lambda(w_2)}(T) \qquad \textup{with } \mu\in (0,1).
		\end{equation*}
		The assumption $\Psi(x) = \widetilde{W} \cdot h_{\lambda(x)}(T) + \tilde{b}$ for all $x \in \mathcal{X}$ implies
		\begin{align*}
			\Psi_i(u) &= [ \widetilde{W} \cdot h_{\lambda(u)}(T) + \tilde{b}]_i = \mu [ \widetilde{W} \cdot h_{\lambda(w_1)}(T) + \tilde{b}]_i + (1-\mu)  [\widetilde{W} \cdot h_{\lambda(w_2)}(T) + \tilde{b}]_i \\
			&= \mu \Psi_i(w_1) + (1-\mu) \Psi_i(w_2) <c,
		\end{align*}
		as $\Psi_i(w)<c$ for $w \in \partial \mathcal{V}\subset \mathcal{W}$, which contradicts $\Psi_i(u)>c$ for $u \in \mathcal{U}$.
	\end{proof}
	
	An example of possible two-dimensional sets $\mathcal{U}$, $\mathcal{V}$ and $\mathcal{W}$, which fulfill the assumptions of Theorem~\ref{th:dupont_nonembeddable}, is shown in Figure \ref{fig:dupont_nonembeddable}.
	Theorem~\ref{th:dupont_nonembeddable} only uses the uniqueness and continuous dependence on initial data of the neural ODE. Under some regularity assumptions on the vector field $f$ and the weight matrices $W, \widetilde{W}$ of the affine linear layers, the class of non-embeddable functions can be enlarged, as the following theorem shows, contributing again to Question (Q2). We assume that both weight matrices $W, \widetilde{W}$ have full rank. This is the generic case, as rank-deficient matrices form a set of Lebesgue measure zero in the parameter space~\citep{Kuehn2024}.

	\begin{theorem}[No Critical Points]\label{th:nocriticalpoint}
		Let $\Psi\in C^1(\mathcal{X}, \mathbb{R}^{q})$ be a function with  $\mathcal{X}\subset \mathbb{R}^n$ open, where one component map has a critical point, i.e., $\nabla \Psi_i(p) = 0$ for some $i\in \{1,\ldots,q\}$ and $p \in \mathcal{X}$. Then $\Psi$ cannot be embedded in the neural ODE architecture $\textup{NODE}_\textup{N}(\mathcal{X},\mathbb{R}^q)$ with  vector field $f \in C^{0,1}(\Omega,\R^m)$, $\Omega \subset \R \times \R^m$ open, and full-rank weight matrices $W, \widetilde{W}$, if $q \leq m$.
	\end{theorem}
	
	\begin{proof}
		The result follows directly from Theorem 4.11 in~\citet{Kuehn2024}. The dimension condition $q \leq m$ is necessary, as for $m < q$, a full-rank matrix $\widetilde{W}$ can have a zero row, such that the corresponding component of $\Phi$ is constant, and every point is a critical point.
	\end{proof}
	
	As shown in Theorem \ref{th:orientationpreserving}, every basic non-augmented neural ODE $\Phi \in \textup{NODE}^k_\textup{N,b}(\mathcal{X},\R^n)$, $k \geq 1$, is an orientation-preserving diffeomorphism. By Theorem~\ref{th:nocriticalpoint}, no general non-augmented neural ODE with vector field $f \in C^{0,1}(\Omega,\R^m)$, full-rank weight matrices, and $q \leq m$ has a critical point. Note that diffeomorphisms are a special case of functions without critical points, such that for basic non-augmented neural ODEs, Theorem~\ref{th:nocriticalpoint} is a special case of Theorem~\ref{th:orientationpreserving} under the same regularity assumptions.
	
	Beyond restricting the class of embeddable functions, Theorem~\ref{th:nocriticalpoint} also has consequences for universal approximation. In Corollary 2.7 in~\citet{Kuehn2024} it is shown that Theorem~\ref{th:nocriticalpoint} additionally implies that non-augmented neural ODEs with full-rank weight matrices and $q \leq m$ cannot have the universal approximation property with respect to the space $C^k(\mathcal{X},\R^q)$ for every $k \geq 0$. The rank-deficient case does not restore universal approximation either. If $W$ does not have full rank, it has a non-trivial kernel, and $\Phi$ is constant along the directions of $\ker W$. If $\widetilde{W}$ does not have full rank, there exists $v \in \R^{q} \setminus \{0\}$ with $v^\top \widetilde{W} = 0$, so that $v^\top \Phi(x) = v^\top \tilde{b}$ for all $x \in \mathcal{X}$, i.e.,~$\Phi$ maps into an affine hyperplane of $\R^{q}$. For $m < q$, the universal approximation property fails for all weight matrices, as the affine linear layer $\tilde{\lambda}$ maps into an affine subspace of $\R^q$ of dimension at most $m$. Consequently, adding up to two affine linear layers to a basic non-augmented neural ODE increases the embedding capability, but is not sufficient to have the universal approximation property with respect to the space $C^k(\mathcal{X},\R^q)$, $k\geq 0$, if either $m < q$, or $q \leq m$ and the weight matrices $W$ and~$\widetilde{W}$ are restricted to have full rank.
	
	\subsection{Basic Augmented Neural ODEs}
	\label{sec:nodeAB}
	
	As the embedding capability of the non-augmented neural ODE architectures presented in Section~\ref{sec:nodeNB} and Section~\ref{sec:nodeN} is restricted, one can extend the phase space and consider augmented neural ODEs, as introduced in~\citet{Dupont2019}. The idea is to embed a given map $\Psi: \mathcal{X} \rightarrow \mathbb{R}^{n}$, $\mathcal{X} \subset \mathbb{R}^n$ with output dimension $q = n$ in a neural ODE in dimension $m > n$, by initializing the additional components by zeros, see Figure~\ref{fig:nodeAB}. In terms of the general neural ODE architecture~\eqref{eq:NODE}, augmented neural ODEs are defined as follows.
	
	\begin{definition}[Basic Augmented Neural ODE]
		The set of all basic augmented neural ODE architectures $\Phi\in \textup{NODE}(\mathcal{X},\mathbb{R}^n)$, $\mathcal{X}\subset \mathbb{R}^n$ open, defined by \eqref{eq:NODE} with $n = q$, $m>n$, identity transformations $\lambda: x \mapsto (x,0)^{\top}$, $\tilde \lambda: (x,0)^{\top} \mapsto x$ for all $x \in \mathbb{R}^n$ and the property that $h_{(x,0)^\top}(T) \in \R^n \times \{0\}^{m-n}$ for all $x \in \mathcal{X}$, is denoted by $\textup{NODE}_{\textup{A,b}}(\mathcal{X},\mathbb{R}^n)$. \label{def:nodeAB}
	\end{definition}

	To ensure that, under iteration of $\Psi$, points are still represented as vectors in $\mathbb{R}^{n} \times \{0\}^{m-n}$, we need to assume that the last $m-n$ components of the time-$T$ map $h_{(x,0)^\top}(T)$ are zeros~\citep{Zhang2020a}. In this sense, augmented means that trajectories starting in the $n$-dimensional subspace $\mathcal{X} \times \{0\}^{m-n}$ evolve in the full $m$-dimensional space and then return after the time~$T$ to the $n$-dimensional subspace $\mathbb{R}^{n} \times \{0\}^{m-n}$. The idea of considering an augmented (or extended) differential equation is well known in various contexts in dynamical systems, e.g., turning non-autonomous systems into autonomous ones, second-order systems into first-order ones, appending parameters as constant phase space variables, or using geometric desingularization near non-hyperbolic points~\citep[see for example][and references therein]{Guckenheimer2002,JardonKuehn}. The subspace condition can classically be interpreted as a finite-time-$T$ invariance of a subspace, which is frequently important in non-autonomous dynamics~\citep[see for example][]{AnagnostopoulouPoetzscheRasmussen,KloedenRasmussen}.

	Basic augmented neural ODEs allow one to embed more functions than basic non-augmented neural ODEs, for instance, the map of Example~\ref{ex:minusxembedding1}, illustrating Question (Q1).
	
	\begin{example} \label{ex:minusxembedding3}
		The map $\Psi: \mathbb{R} \rightarrow \mathbb{R}$, $x \mapsto -x$ can be embedded in the neural ODE architecture $\textup{NODE}_{\textup{A,b}}(\R,\R)$ by choosing
		\begin{equation*}
			\frac{\dd}{\dd t}\begin{pmatrix}
				h_1 \\ h_2
			\end{pmatrix}
			= \frac{\pi}{T}\cdot\begin{pmatrix}
				-h_2 \\ \hspace{7pt} h_1
			\end{pmatrix},
			\quad 
			\begin{pmatrix}
				h_1(0) \\ h_2(0)
			\end{pmatrix}
			= \begin{pmatrix}
				x \\ 0
			\end{pmatrix},
			\quad \Rightarrow  \quad 
			\begin{pmatrix}
				h_1(t) \\ h_2(t)
			\end{pmatrix} = 
			\begin{pmatrix}
				x \cdot \cos(\pi t/T)  \\ x \cdot \sin(\pi t/T)
			\end{pmatrix},
		\end{equation*}
		such that 
		\begin{equation*}
			\Phi(x) = [h_{(x,0)^\top}(T)]_{1} = \left[\begin{pmatrix}
				-x \\ 0
			\end{pmatrix} \right]_1 = -x.
		\end{equation*}
	\end{example}
	
	By Theorem~\ref{th:orientationpreserving}, only maps $\Psi\in C^1(\mathcal{X},\R^n)$, $\mathcal{X}\subset \R^n$ open, that are orientation-preserving diffeomorphisms can possibly be embedded in a basic non-augmented neural ODE $\textup{NODE}^k_\textup{N,b}(\mathcal{X},\R^n)$, $k \geq 1$. Using the suspension flow construction on the mapping torus, we show in the following that all diffeomorphisms $\Psi \in C^1(\mathcal{X}, \mathcal{X})$, $\mathcal{X} \subset \mathbb{R}^n$ open, can be embedded in the time-$T$ map of a flow on a manifold of dimension $n+1$, i.e., with one additional dimension. As the suspension flow is defined on a manifold and not on a Euclidean space, it is strictly speaking not a basic augmented neural ODE in the sense of Definition~\ref{def:nodeAB}. The Euclidean setting is recovered for smooth diffeomorphisms in Theorem~\ref{th:twononlinearlayers}, at the price of two possibly nonlinear layers.
	
	\begin{definition}[\citet{Brin2002, Katok1995}]
		Let $\Psi \in C^0(\mathcal{X}, \mathcal{X})$, $\mathcal{X} \subset \mathbb{R}^n$ open, be a homeomorphism. The $(n+1)$-dimensional manifold 
		\begin{equation*}
			\mathcal{M} \coloneqq \frac{\mathcal{X}\times [0,T]}{(\Psi(x),0)^\top \sim (x,T)^\top}
		\end{equation*}
		is called the mapping torus of $\Psi$. The $\sim$ here means that  $\mathcal{M}$ is a quotient space, where the points $(\Psi(x),0)^\top$ and $(x,T)^\top$ are identified with each other.
	\end{definition}
	
	\begin{theorem}[Suspension Flow Theorem~\citep{Brin2002, Katok1995}] \label{th:suspension}
		Let $\Psi \in C^1(\mathcal{X}, \mathcal{X})$, $\mathcal{X} \subset \mathbb{R}^n$ open, be a diffeomorphism. Then the ODE 
		\begin{equation*}
			\frac{\dd }{\dd t}\begin{pmatrix}
				h \\ \tilde{t}
			\end{pmatrix} = 
			\begin{pmatrix}
				0 \\ 1
			\end{pmatrix}, \qquad
			\begin{pmatrix}
				h(0) \\ \tilde{t}(0)
			\end{pmatrix}
			= \begin{pmatrix}
				x \\ 0
			\end{pmatrix}
		\end{equation*}
		has on the $(n+1)$-dimensional mapping torus $\mathcal{M}$ the time-$T$ map $(x,0)^\top \mapsto (\Psi(x),0)^\top$. Hence, $\Psi$ is embedded, with one additional dimension, in the time-$T$ map of a flow on $\mathcal{M}$.
	\end{theorem}

	\begin{proof}
		The mapping torus $\mathcal{M}$ is well-defined as the map $\Psi$ is bijective. The solution of the initial value problem with initial condition $(x,0)^\top$ is $(h(t),\tilde{t}(t)) = (x,t)$, such that it reaches $(x,T)^\top$ at time~$T$, which is by definition of the mapping torus identified with $(\Psi(x),0)^\top$ in $\mathcal{M}$. In particular, the section $\{\tilde{t} = 0\} \subset \mathcal{M}$ is invariant under the time-$T$ map.
	\end{proof}

	\begin{figure}
		\centering
		\begin{minipage}[c]{0.45\textwidth}
			\centering
			\begin{overpic}[scale = 0.35,,tics=10]
				{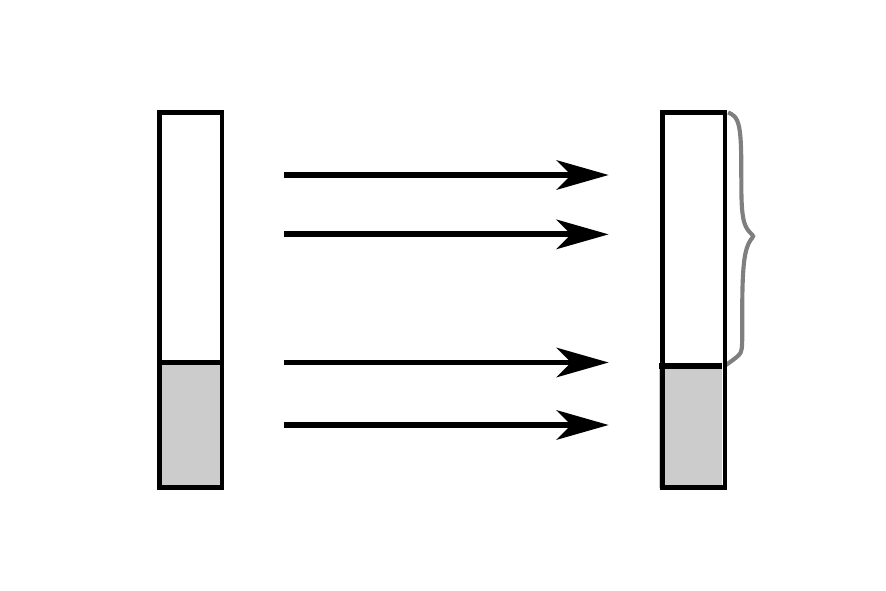}
				\put(11,1.5){$\mathcal{X}\times \{0\}^{m-n}$}
				\put(13,58){$(x,0)^\top$}
				\put(62,58){$h_{(x,0)^\top}(T)$}
				\put(65,1.5){$\R^n\times \{0\}^{m-n}$}
				\put(40,30){$\textup{ODE}$}
				\put(87,37.5){$\Phi(x)$}
			\end{overpic}
		\end{minipage}
		\begin{minipage}[c]{0.08\textwidth}
			\textcolor{white}{.}
		\end{minipage}
		\begin{minipage}[c]{0.45\textwidth}
			\centering
			\begin{overpic}[scale = 0.3,,tics=10]
				{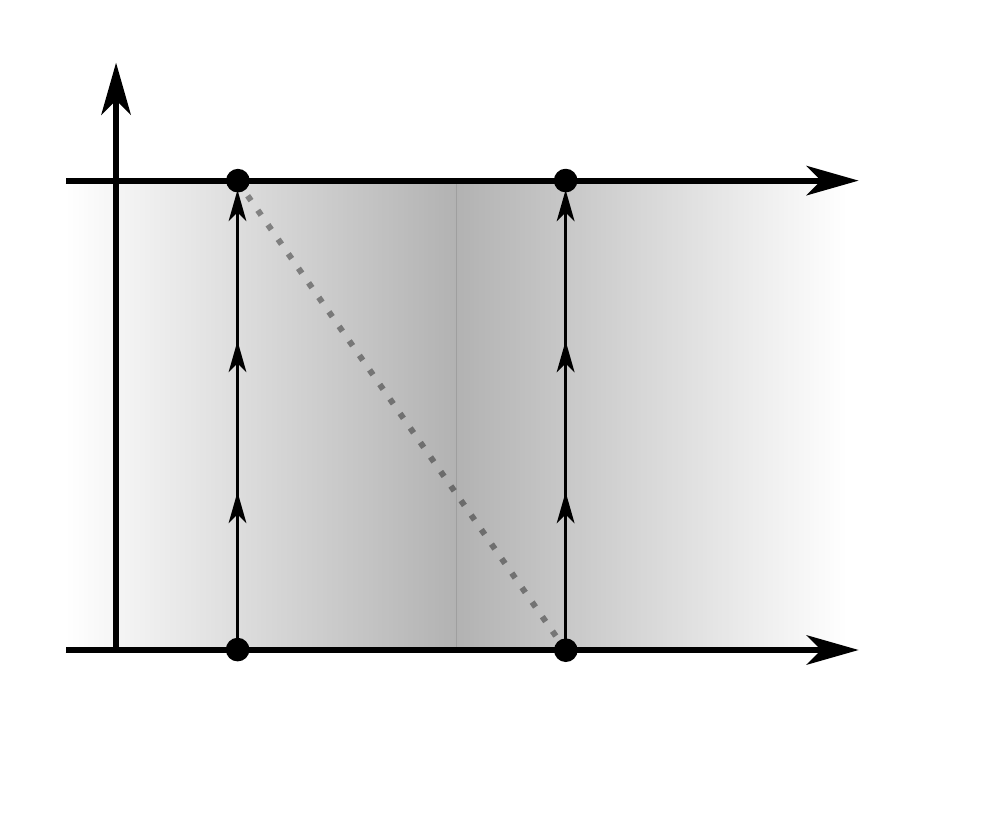}
				\put(9.5,78){$t$}
				\put(0,62){$T$}
				\put(0,15){$0$}
				\put(70,39){$\textcolor{gray}{\mathcal{M}}$}
				\put(88,62){$\mathcal{X}$}
				\put(88,15){$\mathcal{X}$}
				\put(16,8){$(x,0)^\top$}
				\put(45,8){$(\Psi(x),0)^\top$}
				\put(16,69){$(x,T)^\top$}
				\put(45,69){$(\Psi(x),T)^\top$}
			\end{overpic}
		\end{minipage}
		\begin{minipage}[t]{0.46\textwidth}
			\caption{Structure of a basic augmented neural ODE $\Phi \in \textup{NODE}_{\textup{A,b}}(\mathcal{X},\R^n)$.}
			\label{fig:nodeAB}
		\end{minipage}
		\begin{minipage}[t]{0.06\textwidth}
			\textcolor{white}{.}
		\end{minipage}
		\begin{minipage}[t]{0.46\textwidth}
			\caption{Visualization of the suspension flow of Theorem \ref{th:suspension} on the mapping torus~$\mathcal{M}$.}
			\label{fig:suspension}
		\end{minipage}
	\end{figure}

	The construction of the suspension flow on the mapping torus $\mathcal{M}$ is visualized in Figure~\ref{fig:suspension}. In machine learning applications, it is often not practical to work with non-Euclidean manifolds like the mapping torus $\mathcal{M}$. To resolve this problem, the mapping torus $\mathcal{M}$ of a smooth diffeomorphism $\Psi$ can be embedded in the $(2n+2)$-dimensional Euclidean space, which we show in Theorem~\ref{th:twononlinearlayers} in Section~\ref{sec:suspensionflows} via Whitney's embedding theorem. Hence, the embedded suspension flow answers Question~(Q3) for Euclidean spaces. The embedding makes use of two additional transformations, which can be interpreted as (possibly) nonlinear layers, hence the resulting neural ODE architecture is even more general than the general augmented architecture introduced in the following Section~\ref{sec:nodeA}.

	\medskip
	
	Theorem~\ref{th:suspension} and Theorem~\ref{th:twononlinearlayers} give positive embedding results for diffeomorphisms, at the cost of one additional dimension on the mapping torus $\mathcal{M}$, and of $n+2$ additional dimensions and two possibly nonlinear layers in the Euclidean case. Nevertheless, function classes exist that cannot be embedded in the considered architecture. For non-augmented neural ODEs with full-rank weight matrices and $q \leq m$, we showed in Theorem~\ref{th:nocriticalpoint} that they cannot represent functions with critical points. In analogy, we show in Section~\ref{sec:morsesection} for basic augmented neural ODEs that there exists a function class with critical points, which cannot be embedded. 
	
	To define that function class, we introduce in Section~\ref{sec:morsefunctions} (topological) Morse functions. Morse functions are twice continuously differentiable maps $\Psi \in C^2(\mathcal{X},\R)$, $\mathcal{X}\subset \R^n$ open, where every critical point is non-degenerate, i.e., for $p\in \mathcal{X}$ with $\nabla \Psi(p) = 0$ it holds that the Hessian matrix~$H_{\Psi}(p)$ has full rank. By the Morse-Palais lemma (see Theorem~\ref{th:morse}), every Morse function can locally be written as a quadratic form. Topological Morse functions generalize this property to maps with degenerate critical points. Morse functions are particularly relevant here, since they are dense in the Banach space of $k$-times continuously differentiable functions, $k \geq 2$, see Corollary~\ref{cor:morsefunctionsdense}. 
	
	In Section~\ref{sec:nonembeddable}, we combine in Theorem~\ref{th:neg} the properties of Morse functions with the Borsuk-Ulam theorem to prove that if one component of a continuous map $\Psi$ is a topological Morse function with a topologically critical point, then $\Psi$ cannot be embedded in a basic augmented neural ODE. In that way, we answer Question (Q2) for quite a large class of functions for basic augmented neural ODEs.

	\subsection{General Augmented Neural ODEs}
	\label{sec:nodeA}
	
	As for basic non-augmented neural ODEs, one can add up to two affine linear layers to basic augmented neural ODEs in dimension $m$ to embed maps $\Psi: \mathcal{X} \rightarrow \mathbb{R}^{q}$, $\mathcal{X} \subset \mathbb{R}^{n}$ open, with different input and output dimensions. If the dimension $m$ of the underlying initial value problem is larger than the input dimension $n$, the architecture is called augmented.
	
	\begin{definition}[Augmented Neural ODE]
		The set of all augmented neural ODE architectures $\Phi\in \textup{NODE}(\mathcal{X},\mathbb{R}^q)$, $\mathcal{X}\subset \mathbb{R}^n$ open, defined by \eqref{eq:NODE} with vector field dimension $m > n$ is denoted by $\textup{NODE}_{\textup{A}}(\mathcal{X},\mathbb{R}^q)$.
	\end{definition}

	In contrast to the basic augmented neural ODE architecture $\textup{NODE}_{\textup{A,b}}(\mathcal{X},\mathbb{R}^n)$, it is not necessary for the general augmented architecture $\textup{NODE}_{\textup{A}}(\mathcal{X},\mathbb{R}^q)$ to assume the subspace invariance $h_{(x,0)^\top}(T) \in \mathbb{R}^{n} \times \{0\}^{m-n}$, as the neural ODE is followed by an affine linear layer mapping $h_{(x,0)^\top}(T)$ back into a $q$-dimensional space, see Figure~\ref{fig:nodeA}.

	\medskip

	The following theorem shows that the combination of two affine linear layers with an augmented phase space with dimension $m\geq n+q$  is sufficient to be able to embed any continuous map $\Psi: \mathcal{X} \rightarrow \mathbb{R}^{q}$, $\mathcal{X} \subset \mathbb{R}^{n}$, answering Question (Q3). The following theorem is based on the ideas of the proof of Theorem 7 in~\citet{Zhang2020a}, adapted to our setting.
	
	\begin{theorem}[Universal Embedding] \label{th:nodeA}
		If $m \geq n+q$, any $\Psi \in C^k(\mathcal{X},\R^q)$, $k \geq 0$, $\mathcal{X}\subset \mathbb{R}^n$ open, can be embedded in the neural ODE architecture $\textup{NODE}^k_{\textup{A}}(\mathcal{X},\mathbb{R}^q)$.
	\end{theorem}
	
	\begin{proof}
		Fix $T>0$, choose $m\geq n+q$ and define the augmented neural ODE $\Phi$ by the weights and biases
		\begin{equation*}
			W = \begin{pmatrix}
				\textup{Id}_n \\ 0
			\end{pmatrix} \in \mathbb{R}^{m\times n}, \quad b = 0\in\mathbb{R}^m,\quad \widetilde{W} = (0_{q \times (m-q)}, \textup{Id}_q) \in \mathbb{R}^{q \times m}, \quad \tilde{b} = 0 \in \mathbb{R}^q
		\end{equation*}
		and the ODE
		\begin{equation*}
			\frac{\dd }{\dd t} \begin{pmatrix}
				h_{1,\ldots,m-q} \\  h_{m-q+1,\ldots,m}
			\end{pmatrix}
			= \begin{pmatrix}
				0 \\ 
				\frac{1}{T} \cdot \Psi(h_1,\ldots,h_n)
			\end{pmatrix} \eqqcolon f(t,h).
		\end{equation*} 
		As the vector field is zero in the first $m-q$ components it holds $h_i(t) = h_i(0)$ for all $t \geq 0$ and $i\in \{1,\ldots,m-q\}$. As $\Psi \in C^0(\mathcal{X},\mathbb{R}^q)$, the last $q$ components are solvable by direct integration and we obtain $$h_i(t) = h_i(0) + \frac{t}{T}\Psi_{i-m+q}(h_1(0),\ldots,h_n(0))$$ for all $t \geq 0$ and $i \in \{m-q+1,\ldots,m\}$. This solution is unique, as the first $m-q \geq n$ components are constant and the remaining ones are determined by integration of a function of these constants. It depends continuously on the initial data, and hence Assumptions~\ref{assA:ode_existence}, \ref{assA:ode_unique}, and~\ref{assA:continuous_dependence} are fulfilled and the constructed augmented neural ODE $\Phi \in \textup{NODE}_{\textup{A}}(\mathcal{X},\mathbb{R}^q)$ is well-defined. In the case that $\Psi \in C^k(\mathcal{X},\R^q)$, $k \geq 1$, the vector field~$f$ is autonomous and $k$-times continuously differentiable with respect to $h$, hence it holds for the defined architecture that $\Phi \in \textup{NODE}^k_{\textup{A}}(\mathcal{X},\mathbb{R}^q)$. It follows
		\begin{equation*}
			\Phi(x) = \widetilde{W} \cdot h_{Wx+b}(T) +\tilde{b}= \widetilde{W} \cdot h_{(x,0)^\top}(T) = \widetilde{W}\begin{pmatrix}
				x \\ 0 \\ \Psi(x)
			\end{pmatrix} = \Psi(x)
		\end{equation*}
		for all $x \in \mathcal{X}$, such that $\Psi$ is embedded in the neural ODE $\Phi$.
	\end{proof}

	\begin{figure}
		\centering
		\begin{minipage}[c]{0.5\textwidth}
			\centering
			\begin{overpic}[scale = 0.3,,tics=10]
				{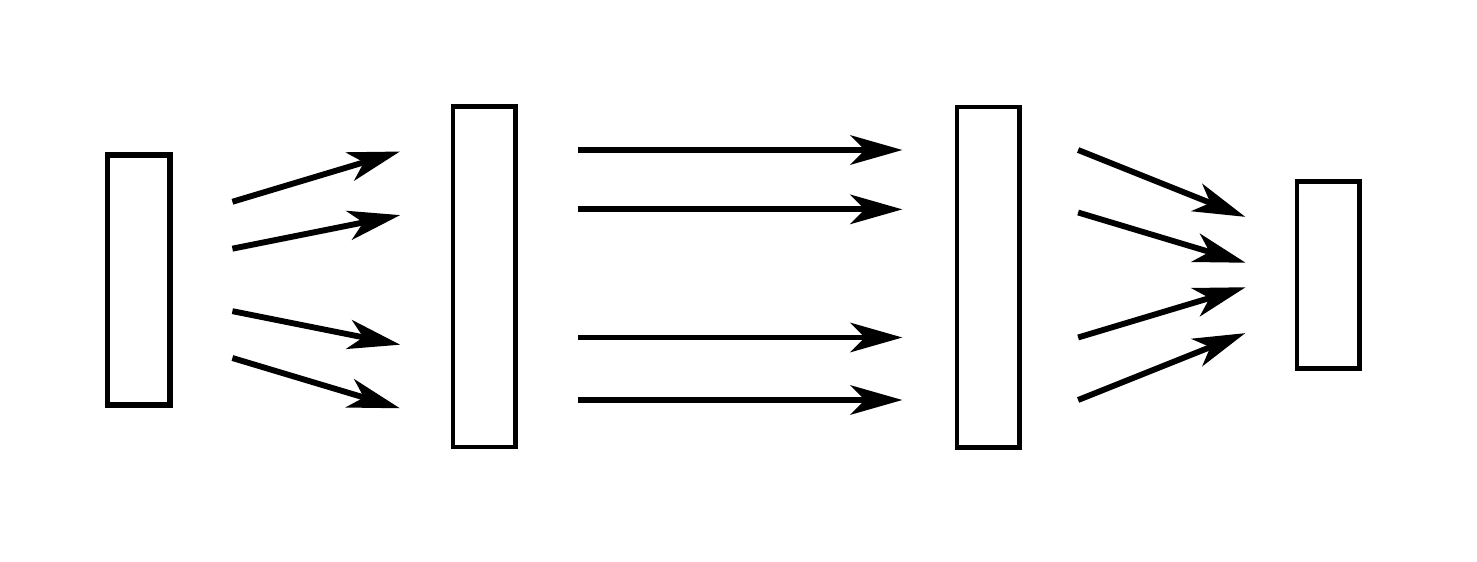}
				\put(2,5){$\mathcal{X}\subset \R^n$}
				\put(25.5,2){$\mathcal{A}\subset \R^m$}
				\put(8,29.5){$x$}
				\put(28,34){$\lambda(x)$}
				\put(59,34){$h_{\lambda(x)}(T)$}
				\put(64,2){$\R^m$}
				\put(43,17.5){$\textup{ODE}$}
				\put(85,28.5){$\Phi(x)$}
				\put(87.5,7){$\R^q$}
			\end{overpic}
			\vspace{2mm}
		\end{minipage}
		\begin{minipage}[c]{0.08\textwidth}
			\textcolor{white}{.}
		\end{minipage}
		\begin{minipage}[c]{0.4\textwidth}
			\centering
			\begin{overpic}[scale = 0.27,,tics=10]
				{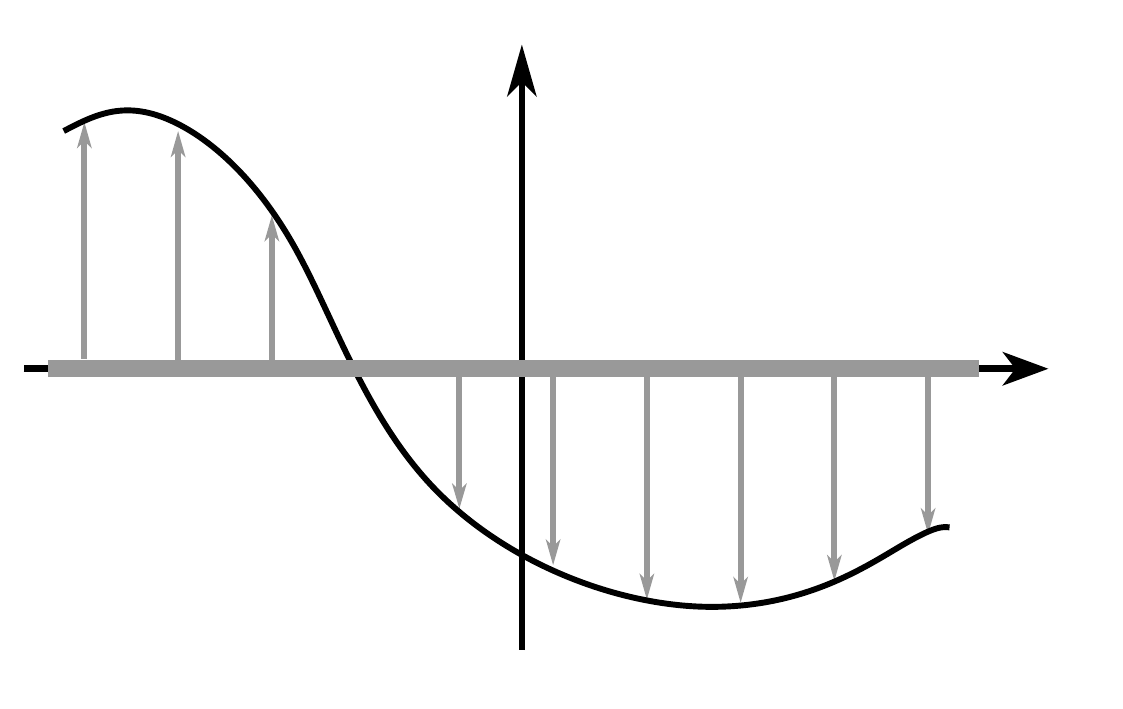}
				\put(96,28){$\R^n$}
				\put(44,61){$\R^q$}
				\put(65,40){\textcolor{gray}{$\scriptsize \begin{pmatrix}
							x \\ 0
						\end{pmatrix}$}}
				\put(15,13){$\scriptsize \begin{pmatrix}
						x \\ \Psi(x)
					\end{pmatrix}$}
			\end{overpic}
			\vspace{2mm}
		\end{minipage}
		\begin{minipage}[t]{0.5\textwidth}
			\caption{Structure of a general augmented neural ODE  $\Phi \in \textup{NODE}_\textup{A}(\mathcal{X},\R^q)$.}
			\label{fig:nodeA}
		\end{minipage}
		\begin{minipage}[t]{0.08\textwidth}
			\textcolor{white}{.}
		\end{minipage}
		\begin{minipage}[t]{0.4\textwidth}
			\caption{Visualization of the proof idea of Theorem~\ref{th:nodeA}.}
			\label{fig:universal_proof}
		\end{minipage}
	\end{figure}
	
	The proof idea of Theorem~\ref{th:nodeA} is visualized in Figure~\ref{fig:universal_proof}: the first $n$ components $x\in \R^n$ of the initial data $(x,0)^\top$ stay constant, while the solution $h$ of the ODE maps the last $q$ components linearly over the time interval $[0,T]$ from $0\in \R^q$ to $\Psi(x)\in \R^q$. 
	
	In particular, Theorem~\ref{th:nodeA} holds for all augmented neural ODE architectures when the output is scalar, i.e., $q=1$, as the condition $m \geq n + 1$ is satisfied for any $m > n$. In the case of a multi-dimensional output $q>1$, Theorem \ref{th:nodeA} shows the universal embedding property only for the case $m \geq n+q$. As the universal embedding property implies the universal approximation property, Theorem~\ref{th:nodeA} also shows the universal approximation property of general augmented neural ODEs if $m \geq n+q$.\medskip
	
	In the previous sections, we have seen that the one-dimensional map $\Psi: \mathbb{R} \rightarrow \mathbb{R}$, $x \mapsto -x$ cannot be embedded in a basic non-augmented neural ODE (see Example~\ref{ex:minusxembedding1}), but in a general non-augmented neural ODE (see Example~\ref{ex:minusxembedding2}) and in a basic augmented neural ODE (see Example~\ref{ex:minusxembedding3}). In the following example, we show how the explicit construction of Theorem~\ref{th:nodeA} embeds the map $\Psi: \mathbb{R} \rightarrow \mathbb{R}$, $x \mapsto -x$ in a general augmented neural ODE.
	
	\begin{example}   \label{ex:minusxembedding4}
		The map $\Psi: \mathbb{R} \rightarrow \mathbb{R}$, $x \mapsto -x$ can be embedded in the neural ODE architecture $\textup{NODE}_{\textup{A}}(\R,\R)$ with $m = 2$ by choosing the affine linear transformations $\lambda: \R \rightarrow \R^2$, $x \mapsto (x,0)^\top$ and $\tilde \lambda: \R^2 \rightarrow \R$, $(x_1,x_2)^\top\mapsto x_2$, and the initial value problem
		\begin{equation*}
			\frac{\dd}{\dd t}\begin{pmatrix}
				h_1 \\ h_2
			\end{pmatrix}
			= \begin{pmatrix}
				0\\ -\frac{h_1}{T}
			\end{pmatrix},
			\quad 
			\begin{pmatrix}
				h_1(0) \\ h_2(0)
			\end{pmatrix}
			= \lambda(x) = \begin{pmatrix}
				x \\ 0
			\end{pmatrix}
			\quad \Rightarrow  \quad 
			\begin{pmatrix}
				h_1(t) \\ h_2(t)
			\end{pmatrix} = 
			\begin{pmatrix}
				x  \\ -\frac{tx}{T}
			\end{pmatrix},
		\end{equation*}
		such that 
		\begin{equation*}
			\Phi(x) = \tilde \lambda(h_{\lambda(x)}(T)) = \tilde \lambda \left(\begin{pmatrix}
				x \\ -x
			\end{pmatrix}\right) = -x.
		\end{equation*}
	\end{example}

	To compare Theorem~\ref{th:nodeA} with the universal embedding result of Section~\ref{sec:suspensionflows}, consider the case $n = q$. Theorem~\ref{th:nodeA} implies that universal embedding is possible if the dimension of the vector field fulfills $m \geq 2n$. In the case of diffeomorphisms $\Psi \in C^1(\mathcal{X},\mathcal{X})$, we show via the suspension flow on the $(n+1)$-dimensional mapping torus $\mathcal{M}$ that one additional dimension is enough to embed the map $\Psi$ in the time-$T$ map of a flow on a manifold of dimension $n+1$. In Section~\ref{sec:suspensionflows} we embed the mapping torus $\mathcal{M}$ in a $(2n+2)$-dimensional Euclidean space, which results in a general augmented neural ODE architecture with two (possibly nonlinear) layers. Notably, both universal embedding theorems, Theorem~\ref{th:nodeA} for continuous functions and Theorem~\ref{th:twononlinearlayers} for smooth diffeomorphisms, require a dimension of order $2n$, namely $2n$ and $2n+2$, respectively, to embed a given map $\Psi: \mathcal{X} \rightarrow \mathbb{R}^n$, $\mathcal{X}\subset \mathbb{R}^n$. The dimension~$2n$ required by both results is of the same order as the width bounds in universal approximation theorems for MLPs and ResNets, shown for example by~\citet{Kidger2020},~\citet{Park2020} and~\citet{Tabuada2020}. The tradeoff between these two approaches, the Whitney embedding requiring twice the dimension but yielding a Euclidean phase space, versus the suspension flow with fewer dimensions but living on a non-Euclidean manifold, is discussed in detail in Remark~\ref{rem:tradeoffs}.
	
	Beyond these existence-based results, a complementary constructive approach arises from the restricted embedding problem studied in Section~\ref{sec:restrictedembedding}. By combining the explicit form of the vector fields to embed monomial maps $\Psi(x) = cx^\alpha$ for $c,x,\alpha \in \R_{>0}$, we show in Corollary~\ref{cor:approximatepolynomial} the universal approximation property of general augmented neural ODEs with respect to continuous functions $y:[a,b]\rightarrow\R^q$ defined on a compact interval $[a,b]\subset \R$. Unlike in Theorem~\ref{th:nodeA}, where the vector field depends on the target map, the vector field is here fixed a priori and independent of the target function, and the output layer is the only trainable component, see Remark~\ref{rem:augmented_comparison} for a detailed comparison.
	
	\section{The Restricted Embedding Problem}\label{sec:restrictedembedding}
	
	In this section, we discuss the restricted embedding problem of embedding a given map~$\Psi$ in the basic non-augmented neural ODE architecture $\textup{NODE}_\textup{N,b}(\mathcal{X},\R^n)$, $\mathcal{X}\subset \R^n$, as introduced in Section~\ref{sec:nodeNB}. The problem is called restricted, because the dimensions of the map $\Psi$ and the neural ODE agree. By Theorem \ref{th:orientationpreserving}, every $\Phi \in \textup{NODE}_\textup{N,b}^k(\mathcal{X},\R^n)$, $k \geq 1$, is an orientation-preserving diffeomorphism, which excludes large function classes from being embeddable in basic non-augmented neural ODEs. 
	
	In the one-dimensional case, solution curves do not cross (cf.~Example~\ref{ex:minusxembedding1}), such that every neural ODE $\Phi \in \textup{NODE}_\textup{N,b}(\mathcal{X},\R)$ is strictly increasing, and Theorem~\ref{th:orientationpreserving} additionally yields $\Phi'(x) > 0$ for all $x \in \mathcal{X}$ if $k \geq 1$. In the following theorem and remark, we give a condition and an explicit construction for embedding strictly increasing functions. Here, $C^{1,0}([0,T] \times \mathcal{X}, \mathbb{R})$ denotes the space of functions $g: [0,T] \times \mathcal{X} \rightarrow \mathbb{R}$ that are continuously differentiable in the first argument $t$ and continuous in the second argument $x$, and $C^{1,1}([0,T] \times \mathcal{X}, \mathbb{R})$ denotes the space of functions that are continuously differentiable in both arguments.
	
	\begin{theorem}\label{th:monotoneembedding}
		If there exists for a given map $\Psi \in C^1(\mathcal{X},\mathbb{R})$, $\mathcal{X}\subset \R$ open, a function $g\in C^{1,1}([0,T] \times \mathcal{X},\mathbb{R})$ with $g(0,x) = x$ and $g(T,x) = \Psi(x)$, which is for every $t \in [0,T]$ strictly increasing in $x$ with $\frac{\partial g}{\partial x}(t,x) > 0$, and whose mixed partial derivative $\frac{\partial^2 g}{\partial x \partial t}$ is continuous, then $\Psi$ can be embedded in the neural ODE architecture $\textup{NODE}^1_\textup{N,b}(\mathcal{X},\R)$.
	\end{theorem}
	
	\begin{proof}	
		As for every $t\in[0,T]$, the function $g(t,x)$ is strictly increasing in $x$, the inverse $g^{-1}(t,x)$ also exists for every fixed $t \in [0,T]$. Then the map $h_x: [0,T]\rightarrow \R$, $h_x(t) \coloneqq g(t,x)$ is a solution of the initial value problem
		\begin{equation}\label{eq:inversion}
			\frac{\dd h_x(t)}{\dd t} = \frac{\partial g}{\partial t}(t,g^{-1}(t,h_x(t))), \qquad h_x(0) = x,
		\end{equation}
		as $g^{-1}(t,h_x(t)) = g^{-1}(t,g(t,x)) = x$ for every $t \in [0,T]$ and $h_x(0) = g(0,x) = x$. Moreover, it holds $h_x(T) = g(T,x) = \Psi(x)$.

		The vector field $(t,y) \mapsto \frac{\partial g}{\partial t}(t,g^{-1}(t,y))$ of~\eqref{eq:inversion} is defined on $\mathcal{D} \coloneqq \{(t,y) \in [0,T] \times \R : y \in g(t,\mathcal{X})\}$, which contains the graphs of all solutions $h_x$, $x \in \mathcal{X}$. We extend it constantly in time by its values at $t = 0$ to $(-\infty,0) \times \mathcal{X}$ and by its values at $t = T$ to $(T,\infty) \times \Psi(\mathcal{X})$. As $\mathcal{X}$ is open and $g$ is continuous and strictly increasing in its second argument, the resulting domain $$\Omega \coloneqq ((-\infty,0) \times \mathcal{X}) \cup \mathcal{D} \cup ((T,\infty) \times \Psi(\mathcal{X}))$$ is open.
		The vector field $(t,y) \mapsto \frac{\partial g}{\partial t}(t,g^{-1}(t,y))$ of~\eqref{eq:inversion} has the derivative
		\begin{equation*}
			\frac{\partial}{\partial y} \frac{\partial g}{\partial t}\left(t,g^{-1}(t,y)\right) = \frac{\partial^2 g}{\partial x \partial t}\left(t,g^{-1}(t,y)\right) \cdot \frac{\partial g^{-1}}{\partial y}(t,y) = \frac{\frac{\partial^2 g}{\partial x \partial t}\left(t,g^{-1}(t,y)\right)}{\frac{\partial g}{\partial x}\left(t,g^{-1}(t,y)\right)}
		\end{equation*}
		with respect to $y$, which is continuous, as $\frac{\partial^2 g}{\partial x \partial t}$ is continuous and $\frac{\partial g}{\partial x} > 0$. Hence, the extended vector field is continuously differentiable in its second argument, in particular locally Lipschitz continuous, such that the solution $h_x$ is unique and depends continuously on the initial condition $x$. Consequently, the resulting neural ODE lies in $\textup{NODE}^1_\textup{N,b}(\mathcal{X},\R)$.
	\end{proof}

    \begin{remark}[Linear interpolation] \label{rem:linearinterp}
        A natural and broadly applicable choice for the interpolating map~$g$ in Theorem~\ref{th:monotoneembedding} is the linear interpolation $g(t,x) \coloneqq \left(1-\frac{t}{T}\right)x + \frac{t}{T}\Psi(x)$. If $\Psi \in C^1(\mathcal{X},\mathbb{R})$ is strictly increasing with $\Psi'(x) > 0$ for all $x\in\mathcal{X}$, then
        $$\frac{\partial g}{\partial x}(t,x) = \left(1-\frac{t}{T}\right) + \frac{t}{T}\Psi'(x) > 0
        \quad \textup{for all } t \in [0,T],$$
        so $g(t,\cdot)$ is strictly increasing for every $t$, and the conditions of the theorem are satisfied. Note that the condition $\Psi'(x)>0$ is also necessary for embeddability in $\textup{NODE}^k_\textup{N,b}(\mathcal{X},\R)$, $k \geq 1$, by Theorem~\ref{th:orientationpreserving}.
    \end{remark}
    
	While Remark~\ref{rem:linearinterp} provides a general constructive recipe, the resulting vector fields, which are derived via the inversion-based formula~\eqref{eq:inversion} in Theorem~\ref{th:monotoneembedding}, are typically non-autonomous and often lead to algebraically complex expressions. To obtain a structurally simpler and autonomous representation, we now turn to a direct construction for the class of monomial maps. The following theorem shows that the maps $\Psi(x) = cx^\alpha$ with  $c,x,\alpha\in \mathbb{R}_{>0}$ can be embedded exactly in basic non-augmented neural ODEs via explicit autonomous ODEs. The explicit form of the vector fields in Equations \eqref{eq:linODE1}-\eqref{eq:linODE2} is motivated by Julia's functional equation, as derived in Appendix~\ref{app:julia}. In the following, we verify directly that they achieve the desired time-$T$ map. 
    
	\begin{theorem}[Monomial Embedding] \label{th:julia_monome}
		Consider for $c,\alpha\in \mathbb{R}_{>0}$ the map $\Psi: \R_{>0} \rightarrow \R_{>0}$, $\Psi(x) = c x^\alpha$. Then the initial value problems
		\begin{align}
			\frac{\dd h}{\dd t} &= \frac{\ln(c)}{T} h,  && h(0) =x >0, \qquad \qquad \textup{if } \alpha = 1, \label{eq:linODE1}\\
			\frac{\dd h}{\dd t} &= \frac{\ln(\alpha)}{T} h \ln\left(c^{\frac{1}{\alpha-1}}h\right),  && h(0) =x >0, \qquad \qquad  \textup{if } \alpha \neq 1, \label{eq:linODE2}
		\end{align}
		have solutions that exist for all $t \geq 0$ and remain in $\R_{>0}$, and the corresponding basic non-augmented neural ODE has the time-$T$ map $h_x(T) = \Psi(x) = cx^\alpha$.
	\end{theorem}

   \begin{proof}
       In both cases, we verify global existence by exhibiting an explicit solution that remains in $\mathbb{R}_{>0}$ for all $t \geq 0$ and uniqueness follows from the fact that the vector field is locally Lipschitz continuous on $\mathbb{R}_{>0}$.

        In the case $\alpha = 1$, the ODE~\eqref{eq:linODE1} is linear and has the explicit solution $h_x(t) = x\exp\left(\frac{\ln(c)}{T}t\right)$. Evaluating the solution at $t = T$ yields $h_x(T) = xe^{\ln(c)} = cx = \Psi(x)$.

        In the case $\alpha \neq 1$, we verify in the following that
        \begin{equation} \label{eq:monomial_solution}
            h_x(t) = c^{\frac{1}{1-\alpha}}\left(c^{\frac{1}{\alpha-1}} x\right)^{\exp\left(\frac{\ln(\alpha)}{T}t\right)}
        \end{equation}
        solves the ODE~\eqref{eq:linODE2}. Setting $u(t) \coloneqq \exp\left(\frac{\ln(\alpha)}{T}t\right)$ and $\tilde{c}\coloneqq c^{\frac{1}{\alpha-1}}$, it holds $h_x(t) = \frac{1}{\tilde{c}} (\tilde{c} x)^{u(t)} $ and we obtain
        \begin{align*}
            \frac{\dd h_x}{\dd t}
            &= \frac{1}{\tilde{c}} (\tilde{c} x)^{u(t)} \cdot \ln( \tilde{c} x) \cdot u(t) \cdot \frac{\ln(\alpha)}{T} = h_x(t) \cdot \ln(\tilde{c} h_x(t)) \cdot \frac{\ln(\alpha)}{T},
        \end{align*}
        which matches the right-hand side of~\eqref{eq:linODE2}. Evaluating the solution~\eqref{eq:monomial_solution} at $t = T$ gives $u(T) = \alpha$, such that it follows
        \begin{equation*}
             h_x(T) = \frac{1}{\tilde{c}}(\tilde{c}x)^{\alpha}
            = c^{\frac{1}{1-\alpha}} \cdot c^{\frac{\alpha}{\alpha-1}} x^\alpha
            = cx^\alpha = \Psi(x),
        \end{equation*}
        which implies the result.
    \end{proof}
	
	As the monomial maps $\Psi(x) = cx^\alpha$ are for $c>0$ and $\alpha \in 2\mathbb{N}_0+1$ strictly increasing on all of~$\R$, the construction of Theorem~\ref{th:julia_monome} applies to all of $\mathbb{R}$ for odd integer exponents: on $\mathbb{R}_{>0}$ the vector fields of Theorem~\ref{th:julia_monome} apply directly, on $\mathbb{R}_{<0}$ an analogous construction applies using the symmetry $\Psi(-x) = -\Psi(x)$, and $x = 0$ is trivially embedded since it is a fixed point of $\Psi$. For odd $\alpha \geq 3$, the extended vector field is continuous at the origin, as $h\ln\abs{h} \rightarrow 0$ for $h \rightarrow 0$, but not Lipschitz continuous. Uniqueness of solutions nevertheless holds: for $x > 0$, the explicit solution~\eqref{eq:monomial_solution} remains in $\R_{>0}$ for all $t \geq 0$, where the vector field is locally Lipschitz continuous, and analogously for $x < 0$ by symmetry. Hence, solutions starting at $x \neq 0$ never reach the origin. Moreover, as $\ln(\alpha) > 0$, the vector field points towards $0$ from both sides near the origin, such that the solution starting at $0$ cannot leave it. Consequently, the resulting architecture lies in $\textup{NODE}_\textup{N,b}(\R,\R)$. The loss of Lipschitz continuity at the origin is unavoidable, as Theorem~\ref{th:orientationpreserving} with $\Psi'(0) = 0$ excludes an embedding in $\textup{NODE}^1_\textup{N,b}(\R,\R)$. The one-dimensional basic neural ODEs embedding $\Psi(x) = cx^\alpha$ for $c, x, \alpha \in \mathbb{R}_{>0}$ can also be combined into an augmented neural ODE approximating continuous functions on a compact interval arbitrarily well.  
    
	\begin{corollary}[Universal Approximation of Continuous Functions] \label{cor:approximatepolynomial}
		Let $[a,b]\subset \R$ be a compact interval and $y:[a,b]\rightarrow \mathbb{R}^q$ be a continuous function. Given any $\eps>0$,  there exists a choice of parameters $m\in \mathbb{N}$ and $c,w_{ij} \in \R$ with $i\in \{0,\ldots,m\}$, $j \in \{1,\ldots,q\}$, such that the augmented neural ODE $\Phi:(a-1,b+1)\rightarrow \R^q$ with vector field $f:\R^m_{>0} \rightarrow \R^m$ defined by
		\begin{align*}
			f(h_1,\ldots,h_m) = \begin{pmatrix}
				0 \\
				\frac{\ln(2)}{T} h_2 \ln(h_2)\\
				\vdots\\
				\frac{\ln(m)}{T} h_m \ln(h_m)
			\end{pmatrix}
		\end{align*}
		and first affine linear layer $\lambda: (a-1,b+1) \rightarrow \R^m$, $\lambda_i(x) = x+c$ for all $i\in \{1,\ldots,m\}$, and second affine linear layer $\tilde \lambda: \R^m \rightarrow \R^q$, $\tilde \lambda_j (y_1,\ldots, y_m) = w_{0j} + \sum_{i = 1}^m w_{ij} y_i$, $j \in \{1,\ldots, q\}$, can approximate~$y$ uniformly on $[a,b]$, i.e., $$\sup_{x \in [a,b]}\norm{ y(x)-\Phi(x)}_\infty<\eps.$$
	\end{corollary}
	
	\begin{proof}
		Fix some constant $c\in \R$, such that $a-1+c>0$, i.e., $\lambda$ maps the open interval $(a-1,b+1)$ into~$\R^m_{>0}$, and let $j \in \{1,\ldots,q\}$. As $y$ is continuous, the function $$y_{j,c}:[a+c,b+c]\rightarrow \R, \quad y_{j,c}(x) \coloneqq y_{j}(x-c)$$ is also continuous. By the Stone-Weierstraß theorem~\citep[see for example][]{Cheney1982}, there exists a number $m_j \in \mathbb{N}$ and a polynomial $$p_{j,c}:[a+c,b+c] \rightarrow \R, \quad p_{j,c}(x) = \sum_{i =0}^{m_j} v_{ij} x^i$$ of degree $m_j$ with $v_{0j},\ldots,v_{m_j j} \in \R$, such that $$\sup_{x \in [a+c,b+c]}\abs{ y_{j,c}(x)-p_{j,c}(x)} = \sup_{x \in [a,b]}\abs{y_j(x)-p_{j,c}(x+c)} <\eps.$$
		Hence, the function $p_j: [a,b]\rightarrow \R$, $p_j(x) \coloneqq p_{j,c}(x+c)$ uniformly approximates $y_j$ on the compact interval $[a,b]$. 
		
		In the following, we show that we can choose the parameters of the augmented neural ODE in such a way that $\Phi_j(x)  =p_j(x)$ for all $x\in [a,b]$ and $j \in \{1,\ldots,q\}$, which then implies together with the last estimate the result as
		\begin{equation*}
			\sup_{x \in [a,b]}\norm{ y(x)-\Phi(x)}_\infty = \max_{j \in \{1,\ldots,q\}} \sup_{x \in [a,b]}\abs{y_j(x)-p_{j}(x)} < \eps.
		\end{equation*}
		We define the parameters of the augmented neural ODE to be $m = \max\{2,\max_{j\in\{1,\ldots,q\}} m_j\}$ and $w_{ij} = v_{ij}$ for all $j \in \{1,\ldots,q\}$, $i \in \{0,\ldots,m_j\}$ and set all remaining parameters $w_{ij}$ to zero. Then it follows for $j \in \{1,\ldots,q\}$:
		\begin{equation*}
			\Phi_j(x) = \tilde{\lambda}_j(h_{\lambda(x)}(T)) = v_{0j} + \sum_{i = 1}^m v_{ij} (x+c)^i = p_j(x),
		\end{equation*}
		where we used Theorem \ref{th:julia_monome} to determine the time-$T$ map of every component of $h$.
	\end{proof}

	\begin{remark} \label{rem:augmented_comparison}
        Corollary~\ref{cor:approximatepolynomial} is logically independent of, and complementary to, the universal embedding result of Theorem~\ref{th:nodeA}: the latter guarantees the existence of an augmented neural ODE embedding a given map $\Psi$, and its proof is constructive in the sense that the vector field is explicitly specified. However, it depends on $\Psi$ itself, and is therefore not useful in practice when $\Psi$ is unknown and to be learned from data. By contrast, in Corollary~\ref{cor:approximatepolynomial} the ODE vector field $f$ is fixed a priori and independent of the target function $y$. The augmentation dimension $m$ is determined by the polynomial degree required to $\varepsilon$-approximate $y$ via Stone-Weierstraß, and the output layer $\tilde{\lambda}$ is the sole trainable component, whose coefficients are directly inherited from the polynomial approximation and can thus be learned from data. Within this specific construction, the smoothness and oscillatory complexity of the target function $y$ determine the model size. This gives a transparent and quantitative link between analytic properties of $y$ and the resources needed by this model.
    \end{remark}

  \begin{remark}
	The argument of Corollary~\ref{cor:approximatepolynomial} extends to inputs $x \in [a,b]^{n}$ via the multivariate Stone-Weierstraß theorem, which guarantees that polynomials in $n$ variables are uniformly dense in $C([a,b]^{n})$~\citep{Cheney1982}. Using the theory of ridge polynomials, each multivariate monomial $x_1^{\alpha_1}\cdots x_{n}^{\alpha_{n}}$ can be expressed as a linear combination of univariate monomials evaluated at affine linear projections of the inputs~\citep{Pinkus1999, Pinkus2015}. After shifting each affine linear projection to positive values on $[a,b]^{n}$, as with the constant $c$ in the proof of Corollary~\ref{cor:approximatepolynomial}, the resulting univariate monomials are embedded exactly via Theorem~\ref{th:julia_monome}, and the output layer $\tilde{\lambda}$ absorbs the coefficients of the linear combinations.
\end{remark}

	\section{Morse Functions: A Class of Non-Embeddable Maps}
	\label{sec:morsesection}
	
	In this section, we use topological arguments to prove results about functions that cannot be embedded in basic augmented neural ODE architectures. To that purpose, we introduce in Section~\ref{sec:morsefunctions} the Borsuk-Ulam theorem and topological Morse functions, a class that strictly extends smooth Morse functions by requiring only a homeomorphic local coordinate change near critical points. Smooth Morse functions demand a non-singular Hessian at every critical point, whereas topological Morse functions also capture degenerate critical points, such as the minimum of $x \mapsto x^4$ at the origin. The local quadratic symmetry produced by this homeomorphism, combined with the Borsuk-Ulam theorem and the injectivity of the \mbox{time-$T$} map, allows us to show in Section~\ref{sec:nonembeddable} that no embedding of a map with a topologically critical component into the basic augmented neural ODE architecture $\textup{NODE}_{\textup{A,b}}(\mathcal{X},\mathbb{R}^n)$ is possible.
 
	\subsection{Borsuk-Ulam Theorem and Topological Morse Functions}
	\label{sec:morsefunctions}
	
	The result proven in Section~\ref{sec:nonembeddable} is based upon the following Borsuk-Ulam theorem. The theorem guarantees the existence of two antipodal points with the same function value on the unit $m$-sphere $S^{m}_1 = \left\{x \in \mathbb{R}^{m+1}: \norm{x}_2 = 1\right\}$ with Euclidean norm $\norm{x}_2 \coloneqq \left(\sum_{i = 1}^n x_i^2\right)^{1/2}$ for $x \in \mathbb{R}^n$.
	
	\begin{theorem}[Borsuk-Ulam Theorem~\citep{Borsuk1933}] \label{th:borsukulam}
		Let $g \in C^0(S^{m}_1,\mathbb{R}^{m})$, $m \geq 1$. Then there exists a point $x \in S_1^m$, such that $g(x) = g(-x)$.
	\end{theorem}

	In the following, we introduce the class of topological Morse functions, which plays an important role in Section~\ref{sec:nonembeddable}. Topological Morse functions are scalar functions, which we relate to the scalar component maps $\Psi_i: \mathcal{X} \rightarrow \mathbb{R}$, $\mathcal{X} \subset \mathbb{R}^{n} $ of the neural ODE architecture. First, we define Morse functions and the index of their critical points.
	
	\begin{definition}[Morse Function~\citep{Hirsch1976, Morse1934}] \label{def:morse}
		A map $\psi \in C^2(\mathcal{X}, \mathbb{R})$ with $\mathcal{X} \subset \mathbb{R}^n$ open is called a Morse function if all critical points of $\psi$ are non-degenerate, i.e., for every critical point $p \in \mathcal{X}$ defined by a zero gradient $\nabla \psi(p) = 0$, the Hessian matrix~$H_{\psi}(p)$ is non-singular. A critical point $p \in \mathcal{X}$ has index $r$ if $r$ eigenvalues of $H_{\psi}(p)$ are negative.
	\end{definition}
	
	The following theorem from singularity theory was first introduced by~\citet{Morse1934} for analytic functions and then generalized for functions on general Banach spaces by~\citet{Palais1969}. 
	
	\begin{theorem}[Morse-Palais Lemma~\citep{Hirsch1976,Palais1969}] \label{th:morse}
		Let $\psi \in C^{k+2}(\mathcal{X},\mathbb{R})$ with $\mathcal{X} \subset \mathbb{R}^n$ open and $k \geq 1$ be a Morse function. Suppose $p \in \mathcal{X}$ is a critical point of~$\psi$ with index $r$. Then there exists a neighborhood $\mathcal{U}$ of $0 \in \mathbb{R}^n$ and a $C^k$-diffeomorphism $\mu: \mathcal{U} \rightarrow \mu(\mathcal{U})$ with $\mu(0) = p$, such that for $(u_1,\ldots,u_n) \in \mathcal{U}$
		\begin{equation} \label{eq:quadraticform}
			\psi(\mu(u_1,\ldots,u_n)) = \psi(p) - \sum_{j = 1}^r u_j^2 + \sum_{j = r+1}^n u_j^2.
		\end{equation}
	\end{theorem}
	
	\begin{example}
		The map $\psi: \mathbb{R} \rightarrow \mathbb{R}$, $x \mapsto 4x^2 -8x +1$ is a Morse function, as $\nabla \psi(x) = 8x-8$, and the only critical point $p = 1$ of $\psi$ is non-degenerate, since $H_\psi(1) = 8 \neq 0$. The smooth diffeomorphism $\mu: \mathbb{R} \rightarrow \mathbb{R}$, $u \mapsto \frac{u}{2} + 1$ with inverse $\mu^{-1}: \mathbb{R} \rightarrow \mathbb{R}$, $u \mapsto 2u-2$ transforms $\psi$ into the simple quadratic form $\psi(\mu(u)) = u^2 - 3 = \psi(p) + u^2$ with index $0$, as guaranteed by Theorem~\ref{th:morse}.
	\end{example}

	There also exist maps with degenerate critical points $p$, i.e., \mbox{$H_\psi(p) = 0$}, which can be transformed into the simple quadratic form of Theorem~\ref{th:morse}, as the following example shows.
	
	\begin{example} \label{ex:topologicallymorse}
		The map $\psi: \mathbb{R}\rightarrow \mathbb{R}$, $x \mapsto x^4$ is not a Morse function, as $\nabla \psi(x) = 4x^3$ and the only critical point $p = 0$ is degenerate, since $H_\psi(0) = 0$. Nevertheless, the homeomorphism
		\begin{equation*}
			\mu: \mathbb{R} \rightarrow \mathbb{R}, \qquad u \mapsto \begin{cases}
				\hspace{7.5pt} \sqrt{u}, \hspace{8pt} \; \textup{if } u \geq 0, \\
				- \sqrt{-u}, \; \textup{if } u<0,
			\end{cases}
			\qquad 
			\mu^{-1}: \mathbb{R} \rightarrow \mathbb{R}, \qquad u \mapsto \begin{cases}
				\hspace{7.5pt} u^2, \; \textup{if } u \geq 0, \\
				- u^2, \; \textup{if } u<0,
			\end{cases}
		\end{equation*}
		transforms $\psi$ into the simple quadratic form $\psi(\mu(u)) = u^2$.
	\end{example}
	
	The phenomenon described in Example~\ref{ex:topologicallymorse} can be made precise and defines the class of topological Morse functions, which have only topologically non-degenerate critical points.
	
	\begin{definition}[Topological Morse Function~\citep{Cantwell1968,Morse1959}] \label{def:morse_topological}
		Consider a map $\psi \in C^0(\mathcal{X}, \mathbb{R})$ with $\mathcal{X} \subset \mathbb{R}^n$ open. A point $y \in \mathcal{X}$ is called topologically ordinary if there exists a neighborhood $\mathcal{V}$ of $0 \in \mathbb{R}^n$ and a homeomorphism $\eta: \mathcal{V} \rightarrow \eta(\mathcal{V})$ with $\eta(0) = y$, such that for all $(v_1,\ldots,v_n) \in \mathcal{V}$ it holds	$\psi(\eta(v_1,\ldots,v_n)) = \psi(y) + v_n$.
		
		A point $p \in \mathcal{X}$, which is not topologically ordinary, is called topologically critical. The map~$\psi$ is called a topological Morse function if every topologically critical point $p \in \mathcal{X}$ has some index $r \in \{0,\ldots,n\}$, i.e., if there exists a neighborhood $\mathcal{U}$ of $0 \in \mathbb{R}^n$ and a homeomorphism $\mu: \mathcal{U} \rightarrow \mu(\mathcal{U})$ with $\mu(0) = p$, such that $\psi$ can be transformed into the quadratic form \eqref{eq:quadraticform}. 
	\end{definition}
	
	For $k \geq 3$, every Morse function in $C^k(\mathcal{X},\R)$, $\mathcal{X} \subset \R^n$ open, is also a topological Morse function, where the required normal form is given by Theorem~\ref{th:morse}. At every other point, the implicit function theorem~\citep[cf.][]{Forster2017} implies the existence of a homeomorphism $\eta$ as in Definition~\ref{def:morse_topological}, such that the point is topologically ordinary. For one-dimensional maps, Proposition~\ref{prop:onedtopologicalmorse} in Appendix~\ref{app:onedmorse} gives a sufficient condition in terms of the orders of the critical points under which a map is a topological Morse function. In the following, we present a theorem regarding Morse functions as perturbations of general maps, which then implies the density of Morse functions in a certain Banach space, which we prove in the upcoming Corollary \ref{cor:morsefunctionsdense}.
	
	\begin{theorem} \label{th:morseperturbation}
		Let $\psi \in C^{k}(\mathcal{X},\mathbb{R})$ with $\mathcal{X} \subset \mathbb{R}^n$ open, and $k \geq 2$. Then for all $a \in \mathbb{R}^n$, except possibly for a null set in $\mathbb{R}^n$ with respect to the Lebesgue measure, the function $\psi_a(x) \coloneqq \psi(x) + \sum_{j = 1}^n a_jx_j$
		is a Morse function.
	\end{theorem}
	
	\begin{proof}
		The idea of the proof is based on~\citet{Guillemin1974}, where the statement is proven for smooth functions $\psi: \mathcal{X} \rightarrow \mathbb{R}$, $\mathcal{X} \subset \mathbb{R}^n$ open.	Consider the map $g \in C^{k-1}(\mathcal{X},\mathbb{R}^n)$, $g(x) = \nabla \psi(x)$. By the definition of~$\psi_a$, it follows that
		\begin{equation*}
			\nabla \psi_a(x) = \nabla \psi(x) + a = g(x) +a, \qquad H_{\psi_a}(x) = H_{\psi}(x) = J_g(x), \qquad \forall \; x \in \mathcal{X}.
		\end{equation*}
		A point $p \in \mathcal{X}$ is a critical point of $\psi_a$ if and only if $g(p) = -a$. As $k-1 \geq 1$, the Morse-Sard lemma~\citep[see for example][]{Morse1939, Sard1942} implies that for $$\mathcal{D} \coloneqq \{p \in \mathcal{X}: J_g(p) = H_{\psi_a}(p) \textup{ does not have full rank} \},$$ the set $g(\mathcal{D}) \coloneqq \{g(p): p \in \mathcal{D}\}$ has Lebesgue measure zero. As the critical points of $\psi_a$ are defined by $g(p) = -a$, the critical points~$p$ are for all $a \in \mathbb{R}^n$, except for possibly a null set in $\mathbb{R}^n$, non-degenerate, and hence $\psi_a$ is a Morse function.
	\end{proof}
	
	\begin{example}
		The smooth map $\psi: \mathbb{R} \rightarrow \mathbb{R}$, $x \mapsto x^3$ has a degenerate critical point at $p = 0$, whereas the perturbed map $\psi_a: \mathbb{R} \rightarrow \mathbb{R}$, $x \mapsto x^3+ax$ has for $a<0$ the two non-degenerate critical points $\pm \sqrt{-a/3}$ and for $a>0$ no critical point at all. Hence for all parameter values $a \in \mathbb{R}\backslash\{0\}$, $\psi_a$ is a Morse function. The set $\{0\}$ for which $\psi_a$ is not a Morse function is a null set in $\mathbb{R}$, as guaranteed by Theorem~\ref{th:morseperturbation}.
	\end{example}
	
	The last theorem can be used to prove the density of Morse functions in the Banach space of \mbox{$C^k$-functions} endowed with the $C^k$-norm.
	
	\begin{corollary}\label{cor:morsefunctionsdense}
		Let $\mathcal{X} \subset \mathbb{R}^n$ be open and bounded. For $k \in \mathbb{N}_0$, the vector space
		\begin{equation*}
			C^k(\overbar{\mathcal{X}},\mathbb{R}) \coloneqq \left\{ \psi \in C^k(\mathcal{X},\mathbb{R}) : \partial^s \psi \textup{ extends continuously to } \overbar{\mathcal{X}} \textup{ for all } \abs{s} \leq k  \right\},
		\end{equation*}
		endowed with the $C^k$-norm $\norm{\psi}_{C^k(\overbar{\mathcal{X}},\R)} \coloneqq \sum_{\abs{s} \leq k} \norm{\partial^s \psi}_\infty$
		is a Banach space. Here, $\overbar{\mathcal{X}}$ denotes the closure of $\mathcal{X}$ and $\norm{f}_\infty \coloneqq \sup_{x \in \overbar{\mathcal{X}}} \abs{ f(x)}$ the supremum norm  of a continuous function $f: \overbar{\mathcal{X}} \rightarrow \mathbb{R}$. If additionally $k \geq 2$, the set of Morse functions 
		\begin{equation*}
			M \coloneqq \left\{\psi \in C^k(\overbar{\mathcal{X}},\mathbb{R}): \psi \big\vert_{\mathcal{X}} \textup{ is a Morse function}\right\}
		\end{equation*}
		is dense in $\left(C^k(\overbar{\mathcal{X}},\mathbb{R}), \norm{\cdot}_{C^k(\overbar{\mathcal{X}},\R)} \right)$.
	\end{corollary} 
	
	\begin{proof}
		The vector space $C^k(\overbar{\mathcal{X}},\mathbb{R})$ endowed with the $C^k$-norm $\norm{\cdot}_{C^k(\overbar{\mathcal{X}},\R)} $ is a Banach space~\citep[see for example][]{Alt2012}. As $\overbar{\mathcal{X}}$ is compact, it holds $\sup_{x \in \overbar{\mathcal{X}}} \norm{x}_{\infty} \eqqcolon K < \infty$,
		where $\norm{x}_{\infty} \coloneqq \max\{\abs{x_1},\ldots,\abs{x_n}\}$ is the supremum norm of a vector $x \in \overbar{\mathcal{X}}$. As for $k\geq 2$ and every $a\in \mathbb{R}^n$, except for possibly a null set in $\mathbb{R}^n$, the perturbation $\psi_a$ of $\psi \in C^k(\overbar{\mathcal{X}},\mathbb{R})$ of Theorem \ref{th:morseperturbation} is a Morse function, there exists $a \in \mathbb{R}^n$ with $\norm{a}_\infty \leq \delta \coloneqq \varepsilon/(n(K+1))$, such that the function $\psi_a$ lies in an $\varepsilon$-neighborhood of $\psi$ for an arbitrary $\eps>0$:
		\begin{equation*}
			\norm{\psi_a - \psi}_{C^k(\overbar{\mathcal{X}},\R)} = \Biggl\lVert \sum_{j=1}^n a_j x_j \Biggr\rVert_{C^k(\overbar{\mathcal{X}},\R)} = \Biggl\lVert\sum_{j=1}^n a_j x_j\Biggr\rVert_{\infty} + \sum_{j=1}^n \abs{ a_j} \leq \delta n K + \delta n = \varepsilon.
		\end{equation*}
		Hence, the subset $M$ is dense in $C^k(\overbar{\mathcal{X}},\mathbb{R})$.
	\end{proof}
	
	\subsection{Implications on Basic Augmented Neural ODEs}
	\label{sec:nonembeddable}
	
	In this section, we use the local symmetry properties of Morse functions and the Borsuk-Ulam theorem introduced in Section \ref{sec:morsefunctions} to prove a theorem about the non-embeddability of Morse functions in basic augmented neural ODEs.

	\begin{theorem} \label{th:neg}
		Let $\Psi \in C^0(\mathcal{X},\mathbb{R}^{n})$, $\mathcal{X} \subset \mathbb{R}^{n}$ open, have at least one component map $\Psi_i \in C^0(\mathcal{X},\mathbb{R})$, $i \in \{1,\ldots,n\}$, which is a topological Morse function with at least one topologically critical point. Then the map $\Psi$ cannot be embedded in the basic augmented neural ODE architecture $\textup{NODE}_{\textup{A,b}}(\mathcal{X},\R^n)$. 
	\end{theorem}

	\begin{proof}
		As defined in Section~\ref{sec:nodeAB}, the time-$T$ map of the solution of the initial value problem in dimension $m>n=q$ of a basic augmented neural ODE $\Phi \in \textup{NODE}_{\textup{A,b}}(\mathcal{X},\R^n)$ is denoted by $h_{(x,0)^\top}(T)\in \mathbb{R}^n \times \{0\}^{m-n}$. Suppose that we can embed the map $\Psi \in C^0(\mathcal{X},\mathbb{R}^{n})$, $\mathcal{X} \subset \mathbb{R}^n$ in the architecture $\textup{NODE}_{\textup{A,b}}(\mathcal{X},\R^n)$, then $\Psi(x) = \Phi(x) = [h_{(x,0)^\top}(T)]_{1,\ldots,n}$ for all $x \in \mathcal{X}$. 
		
		As by Assumptions~\ref{assA:ode_existence} and \ref{assA:ode_unique} the solution $h_{(x,0)^\top}: [0,T] \rightarrow \mathbb{R}^m$ is unique and exists for $t \in [0,T]$, the solution curves do not cross and the time-$T$ map $H: \mathcal{X} \times \{0\}^{m-n} \rightarrow \mathbb{R}^n \times \{0\}^{m-n}$, $H((x,0)^\top) \coloneqq h_{(x,0)^\top}(T)$ is injective. 

		Consider first the case $n = 1$. As $H$ is injective with values in $\R \times \{0\}^{m-1}$, the map $\Psi(x) = [H((x,0)^\top)]_1$ is injective. By Definition~\ref{def:morse_topological}, the topological Morse function $\Psi_1 = \Psi$ with topologically critical point $p$ admits a homeomorphism $\mu_1$ from a neighborhood of $0 \in \R$ onto a neighborhood of~$p$ with $\mu_1(0) = p$ and $\Psi_1(\mu_1(u)) = \Psi_1(p) \pm u^2$ for $u$ near~$0$. Hence, $\Psi_1$ is not injective on any neighborhood of $p$, which is a contradiction. In the following, let $n \geq 2$.

		The considered map $\Psi \in C^0(\mathcal{X},\mathbb{R}^{n})$ has by assumption a component map $\Psi_i \in C^0(\mathcal{X},\mathbb{R})$, which is a topological Morse function with topologically critical point $p \in \mathcal{X}$. Definition~\ref{def:morse_topological} implies that there exists a neighborhood $\mathcal{U}$ of $0 \in \mathbb{R}^{n}$ and a homeomorphism $\mu_1: \mathcal{U} \rightarrow \mu_1(\mathcal{U})$ with $\mu_1(0) = p$, such that for all $(u_1,\ldots,u_n) \in \mathcal{U}$ it holds
		\begin{equation*}
			\Psi_i(\mu_1(u_1,\ldots,u_n)) = \Psi_i(p) - \sum_{j = 1}^r u_j^2 + \sum_{j = r+1}^n u_j^2
		\end{equation*}
		with some index $r \in \{0,\ldots,n\}$. As $\mathcal{U}$ is a neighborhood of the origin, there exists $\varepsilon >0$, such that $B_\varepsilon^{n} \coloneqq \{u \in \mathbb{R}^{n}: \norm{u}_2 < \varepsilon\} \subset \mathcal{U}$. For $u \in B_\varepsilon^n$ it holds $\Psi_i(\mu_1(u)) = \Psi_i(\mu_1(-u))$.
		
		For $0 < \delta < \varepsilon$ it holds $S_\delta^{n-1}\coloneqq \{u \in \mathbb{R}^{n}: \norm{u}_2 = \delta\} \subset B_\varepsilon^n$. We now define a second homeomorphism $\mu_2:S_1^{n-1} \rightarrow S_\delta^{n-1}$, $u \mapsto \delta u$ with continuous inverse $\mu_2^{-1} :S_\delta^{n-1} \rightarrow S_1^{n-1}$, $u \mapsto \delta^{-1} u$. To transfer initial conditions in $\mathcal{X}$ to initial conditions in the augmented space $\mathcal{X} \times \{0\}^{m-n}$, a third homeomorphism $\mu_3: \mathbb{R}^n \rightarrow \mathbb{R}^n \times \{0\}^{m-n}$, $x \mapsto (x,0)^\top$ with continuous inverse $\mu_3^{-1}: \mathbb{R}^n \times \{0\}^{m-n} \rightarrow \mathbb{R}^n$, $(x,0)^\top \mapsto x$ is defined. By Assumption \ref{assA:continuous_dependence}, the solution of the neural ODE depends continuously on the initial condition $(x,0)^\top \in \mathcal{A} = \mathcal{X} \times \{0\}^{m-n}$, such that the time-$T$ map $H: \mathcal{X} \times \{0\}^{m-n} \rightarrow \mathbb{R}^n \times \{0\}^{m-n}$ is continuous in $x \in \mathcal{X}$. Consider now the map 
		\begin{equation*}
			\widetilde{H}: S_1^{n-1} \rightarrow \mathbb{R}^{n-1}, \qquad \widetilde{H}(u) \coloneqq [\mu_3^{-1}(H(\mu_3(\mu_1(\mu_2(u)))))]_{1,\ldots,i-1,i+1,\ldots,n},
		\end{equation*}
		which transfers an input $u \in S_1^{n-1}$ to an input $ \mu_3(\mu_1(\mu_2 (u))) \in \mathcal{X} \times \{0\}^{m-n}$ of the neural ODE. The time-$T$ map $H(\mu_3(\mu_1(\mu_2(u)))) \in \mathbb{R}^n \times \{0\}^{m-n}$ is then restricted to its first $n$ components by $\mu_3^{-1}$ and afterwards the $i$-th component is removed. As all occurring maps are continuous, the map $\widetilde{H}$ is continuous in $u$. As $n-1 \geq 1$, by the Borsuk-Ulam Theorem~\ref{th:borsukulam}, a point $\tilde{u} \in S_1^{n-1}$ exists, such that $\widetilde{H}(\tilde{u}) = \widetilde{H}(-\tilde{u})$. 
		
		As we suppose that the map $\Psi$ is embedded in the basic augmented neural ODE architecture $\textup{NODE}_{\textup{A,b}}(\mathcal{X},\R^n)$, it holds for the component map~$\Psi_i$ that
		\begin{equation*}
			\Psi_i(x) = \Phi_i(x) = [h_{(x,0)^\top}(T)]_i = [H((x,0)^\top)]_i = [\mu_3^{-1}(H(\mu_3(x)))]_i, \qquad x \in \mathcal{X}. 
		\end{equation*}
		Inserting the point $\mu_1(\mu_2(\tilde{u})) \in \mathcal{X}$ with $\tilde{u} \in S_1^{n-1}$ into the map $\Psi$ leads to 
		\begin{equation*}
			[\mu_3^{-1}(H(\mu_3(\mu_1(\mu_2(\tilde{u})))))]_{i} = \Psi_i(\mu_1(\mu_2(\tilde{u}))) =  \Psi_i(\mu_1(\mu_2(-\tilde{u}))) = [\mu_3^{-1}(H(\mu_3(\mu_1(\mu_2(-\tilde{u})))))]_{i} 
		\end{equation*}
		as $\mu_2(\tilde{u}) = -\mu_2(-\tilde{u}) \in B_\varepsilon^n$ and for $u \in B_\varepsilon^n$ it holds $\Psi_i(\mu_1(u)) =\Psi_i(\mu_1(-u))$. Together with the property $\widetilde{H}(\tilde{u}) = \widetilde{H}(-\tilde{u})$ it follows that
		\begin{equation*}
			\mu_3^{-1}(H(\mu_3(\mu_1(\mu_2(\tilde{u}))))) = \mu_3^{-1}(H(\mu_3(\mu_1(\mu_2(-\tilde{u}))))).
		\end{equation*}
		By definition of $\mu_3^{-1}$, it follows that also $H(\mu_3(\mu_1(\mu_2(\tilde{u})))) = H(\mu_3(\mu_1(\mu_2(-\tilde{u}))))$, such that the time-$T$ map $H$ is not injective, which contradicts Assumption~\ref{assA:ode_unique}. Hence, the map $\Psi$ cannot be embedded in a basic augmented neural ODE as defined in Section~\ref{sec:nodeAB}.
	\end{proof}

    \begin{remark} \label{rem:assumptions_morse}
    Theorem~\ref{th:neg} remains valid if the assumption that $\Psi_i$ is a topological Morse function is relaxed to a broader symmetry condition: there exists a neighborhood $\mathcal{U}$ of $0 \in \mathbb{R}^n$ and a homeomorphism $\mu: \mathcal{U} \rightarrow \mu(\mathcal{U})$ with $\mu(0) = p$, such that for $u \in \mathcal{U}$ it holds $-u \in \mathcal{U}$ and $\Psi_i(\mu(u)) = \Psi_i(\mu(-u))$. This condition is satisfied, for example, if $\Psi_i$ can be transformed into the generalized quadratic form
    \begin{equation*}
        \Psi_i(\mu(u_1,\ldots,u_n)) = \Psi_i(p) + \sum_{j = 1}^n s_j u_j^2 
    \end{equation*}
    with $s_j \in \{-1,0,1\}$. While smooth Morse theory requires a non-singular Hessian to establish local forms via diffeomorphisms, our topological approach requires only a homeomorphism. This allows Theorem~\ref{th:neg} to handle smoothly degenerate cases like $\Psi_i(x) = x^4$, where $H_{\Psi_i}(0) = 0$. 
    Furthermore, this condition captures maps that are not classifiable as topological Morse functions at all. For instance, $\Psi_i(x_1, x_2) = x_1^4$ has a non-isolated set of topologically critical points (the entire $x_2$-axis) and thus is not a topological Morse function. Nevertheless, it possesses the antipodal symmetry ($s_1 = 1, s_2 = 0$) necessary to trigger the Borsuk-Ulam contradiction. Consequently, the non-embeddability obstruction is not confined to maps with well-defined Morse indices, but extends to a genuinely larger class of maps with degenerate or flat local geometry.
\end{remark}

	\section{Suspension Flows and Differential Geometry}
	\label{sec:suspensionflows}
	
	In Theorem~\ref{th:suspension}, the suspension flow on the $(n+1)$-dimensional mapping torus $\mathcal{M}$ was introduced. Via the suspension flow, it is possible to embed every diffeomorphism $\Psi\in C^1( \mathcal{X},\mathcal{X})$, $\mathcal{X} \subset \mathbb{R}^n$ open, in the time-$T$ map of a flow on a manifold of dimension $n+1$. As it is often not practical in machine learning applications to work on a general topological manifold $\mathcal{M}$, it is possible to embed $\mathcal{M}$ as a submanifold in $\mathbb{R}^{2n+2}$, which we prove in Theorem \ref{th:twononlinearlayers} for smooth diffeomorphisms. The resulting neural ODE architecture is then an augmented neural ODE with two additional, possibly nonlinear layers. The idea of embedding the suspension flow in a Euclidean space was mentioned but not proven by~\citet{Zhang2020a}. To rigorously prove this statement for smooth diffeomorphisms, we need the following two results from differential geometry.
	
	\begin{theorem}[Whitney Embedding Theorem~\citep{Whitney1944}] \label{th:whitney}
		Let $\mathcal{N}$ be a $p$-dimensional smooth manifold with $p \geq 1$. Then there exists a smooth embedding of $\mathcal{N}$ into $\mathbb{R}^{2p}$.
	\end{theorem} 
	
	\begin{theorem}[Quotient Manifold Theorem~\citep{Lee2013,Lee2018}] \label{th:quotientmanifold}
		Let $G$ be a Lie group acting smoothly, freely, and properly on a smooth manifold $M$. Then the quotient space $M/G$ is a topological manifold with dimension $\dim{M}-\dim{G}$ and it has a smooth structure, such that the quotient map $\pi: M \rightarrow M/G$ is a smooth submersion. 
	\end{theorem}

    Here, the group action of $G$ is called \emph{free} if $g \cdot x = x$ for some $x \in M$ and $g\in G$ implies $g = e$ (the identity element of $G$), and \emph{proper} if for every convergent sequence $(x_i)_{i \in \mathbb{N}}$ in $M$ such that $(g_i \cdot x_i)_{i \in \mathbb{N}}$ also converges, the sequence $(g_i)_{i \in \mathbb{N}}$ has a convergent subsequence in~$G$~\citep{Lee2013}. \medskip
	
	The first step to prove the embedding of the mapping torus $\mathcal{M}$ in the Euclidean space~$\mathbb{R}^{2n+2}$ is to show that $\mathcal{M}$ is a smooth manifold if $\Psi \in C^\infty(\mathcal{X},\mathcal{X})$ is a smooth diffeomorphism. The embedding of $\mathcal{M}$ in $\mathbb{R}^{2n+2}$ follows  afterwards from Whitney's Embedding Theorem~\ref{th:whitney}. 
	
	\begin{proposition}
		Let $\Psi \in C^\infty(\mathcal{X},\mathcal{X})$, $\mathcal{X} \subset \mathbb{R}^n$ open, be a diffeomorphism. Then the mapping torus~$\mathcal{M}$ is a smooth $(n+1)$-dimensional manifold.
	\end{proposition}
    
    \begin{proof}
        Consider the group action of $\mathbb{Z}$ on $\mathcal{X} \times \mathbb{R}$ defined by
        \begin{equation*}
            k \cdot (x, t) \;=\; \left(\Psi^k(x),\; t - kT\right), \qquad k \in \mathbb{Z}.
        \end{equation*}
        As $\Psi \in C^\infty(\mathcal{X},\mathcal{X})$ is a diffeomorphism, all iterates $\Psi^k$, $k \in \mathbb{Z}$, are smooth and this group action is smooth. It is also \emph{free}, as $k \cdot (x,t) = (x,t)$ implies $t - kT = t$, which forces $k = 0$ as $T > 0$. Furthermore, the group action is \emph{proper}: if $(x_i, t_i)$ and $\bigl(\Psi^{k_i}(x_i), t_i - k_i T\bigr)$ both converge, then $k_i T = t_i - (t_i - k_i T)$ converges in $\mathbb{R}$, so $k_i$ is eventually constant and hence convergent in $\mathbb{Z}$. 
        
        By the Quotient Manifold Theorem~\ref{th:quotientmanifold}, and as $\dim \mathbb{Z} = 0$, the quotient $(\mathcal{X}  \times \mathbb{R})/\mathbb{Z}$ is a smooth manifold of dimension $n+1$. The equivalence relation induced by this group action is given by
        \begin{equation*}
            (x, t) \sim \left(\Psi^k(x), t - kT\right) \qquad \textup{for} \; k \in \mathbb{Z}
        \end{equation*}
        which restricted to $\mathcal{X}  \times [0,T]$ is precisely the mapping torus $\mathcal{M}$. Hence, $\mathcal{M}$ is a smooth $(n+1)$-dimensional manifold.
    \end{proof}

	For $\Psi$ being a smooth diffeomorphism, we can now apply Whitney's Embedding Theorem~\ref{th:whitney} to the $(n+1)$-dimensional mapping torus~$\mathcal{M}$. The following theorem shows how the embedding of the suspension manifold leads to the embedding of the map $\Psi$ in a neural ODE architecture with two additional (possibly nonlinear) layers.
	
	\begin{theorem} \label{th:twononlinearlayers}
		Let $\Psi \in C^\infty(\mathcal{X},\mathcal{X})$, $\mathcal{X} \subset \mathbb{R}^n$ open, be a diffeomorphism. Then $\Psi$ can be embedded in a neural ODE in dimension $2n+2$ with two additional (possibly nonlinear) layers, i.e., there exist a smooth autonomous vector field $f$ defined on an open subset of $\mathbb{R}^{2n+2}$, a smooth embedding $\lambda: \mathcal{X} \rightarrow \mathbb{R}^{2n+2}$, and a smooth map $\tilde{\lambda}: \mathcal{T} \rightarrow \mathcal{X}$ defined on an open set $\mathcal{T} \subset \mathbb{R}^{2n+2}$ containing $\mathcal{S} \coloneqq \lambda(\mathcal{X})$, such that $\Psi(x) = \tilde{\lambda}(h_{\lambda(x)}(T))$ for all $x \in \mathcal{X}$. Moreover, the two layers satisfy $\tilde{\lambda} \circ \lambda = \textup{id}_{\mathcal{X}}$, and the time-$T$ map $a \mapsto h_a(T)$ maps the embedded copy $\mathcal{S}$ of $\mathcal{X}$ onto itself, such that $\Psi^k(x) = \tilde{\lambda}(h_{\lambda(x)}(kT))$ for all $x \in \mathcal{X}$ and $k \in \mathbb{N}$.
	\end{theorem}
	
	\begin{proof}
		By the previous proposition, the mapping torus $\mathcal{M}$ is a smooth $(n+1)$-dimensional manifold, such that by Whitney's Embedding Theorem~\ref{th:whitney} there exists a smooth embedding $\mu \in C^\infty(\mathcal{M}, \mathbb{R}^{2n+2})$, i.e., a diffeomorphism onto the embedded submanifold $\mu(\mathcal{M}) \subset \mathbb{R}^{2n+2}$. The suspension flow $$\left(\frac{\dd x}{\dd t},\frac{\dd \tilde{t}}{\dd t}\right)^\top = (0,1)^\top \in \R^n \times \R$$ defines a smooth vector field on $\mathcal{M}$, whose solution with initial condition $(x_0,t_0)^\top$ is $x(t) = x_0$ and $\tilde{t}(t) = t_0 + t$. Its pushforward under $\mu$ is the smooth vector field on $\mu(\mathcal{M})$ given in the coordinates $(x,\tilde{t})$ of $\mathcal{X} \times \R$ by
		\begin{equation*}
			J_\mu(x,\tilde{t}) \cdot \begin{pmatrix} 0 \\ 1 \end{pmatrix} = \frac{\partial \mu}{\partial \tilde{t}}(x,\tilde{t}) \in \mathbb{R}^{2n+2}.
		\end{equation*}
		As $\mu(\mathcal{M})$ is an embedded submanifold, this vector field extends to a smooth vector field $f$ on an open neighborhood of $\mu(\mathcal{M})$ in $\mathbb{R}^{2n+2}$~\citep{Lee2013}. As $f$ is tangent to $\mu(\mathcal{M})$, the image under $\mu$ of every trajectory of the suspension flow is a solution of the ODE with vector field $f$, and by uniqueness it is the solution with the corresponding initial condition. For $(x_0,t_0)^\top \in \mathcal{M}$, the time-$T$ map on $\mathcal{M}$ is
		\begin{equation*}
			\begin{pmatrix} x_0 \\ t_0 \end{pmatrix} + \int_0^T \begin{pmatrix} 0 \\ 1 \end{pmatrix} \dd t = \begin{pmatrix} x_0 \\ t_0 + T \end{pmatrix} \sim \begin{pmatrix} \Psi(x_0) \\ t_0 \end{pmatrix},
		\end{equation*}
		where $\sim$ denotes the identification of the mapping torus. The time-$T$ map of $f$ with initial condition $\mu((x_0,t_0)^\top)$ is given by
		\begin{equation*}
			\mu\begin{pmatrix} x_0 \\ t_0 \end{pmatrix} + \int_{t_0}^{t_0+T} \frac{\partial \mu}{\partial \tilde{t}}(x_0,s) \, \dd s = \mu\begin{pmatrix} x_0 \\ t_0 + T \end{pmatrix} = \mu\begin{pmatrix} \Psi(x_0) \\ t_0 \end{pmatrix}.
		\end{equation*}
		Hence, the time-$T$ map on $\mathcal{M}$ is conjugated by $\mu$ to the time-$T$ map of $f$ on $\mu(\mathcal{M})$.

		The quotient map restricted to $\mathcal{X} \times (-T/2,T/2)$ is a diffeomorphism onto an open subset of~$\mathcal{M}$, such that the section $\mathcal{M}_0 \coloneqq \{\tilde{t} = 0\} \subset \mathcal{M}$ is an embedded submanifold diffeomorphic to $\mathcal{X}$, and $\mu_0: \mathcal{X} \rightarrow \mathbb{R}^{2n+2}$, $\mu_0(x) \coloneqq \mu((x,0)^\top)$, is a smooth embedding with image $\mathcal{S} \coloneqq \mu(\mathcal{M}_0)$. Let $\mathcal{T} \subset \mathbb{R}^{2n+2}$ be a tubular neighborhood of the embedded submanifold $\mathcal{S}$. Consider a smooth retraction of $\mathcal{T}$ onto~$\mathcal{S}$, which is a smooth map $r: \mathcal{T} \rightarrow \mathcal{S}$ that restricts to the identity on~$\mathcal{S}$, i.e., $r(s) = s$ for all $s \in \mathcal{S}$. Such a retraction exists by~\citet{Lee2013}. 
		
		We take as the layer before the neural ODE the smooth map $\lambda \coloneqq \mu_0$ and as the layer after the neural ODE the smooth map $\tilde{\lambda} \coloneqq \mu_0^{-1} \circ r: \mathcal{T} \rightarrow \mathcal{X}$. By the conjugacy, it holds $h_{\lambda(x)}(T) = \mu((\Psi(x),0)^\top) = \mu_0(\Psi(x)) \in \mathcal{S}$ for all $x \in \mathcal{X}$, hence
		\begin{equation*}
			\Phi(x) = \tilde{\lambda}(h_{\lambda(x)}(T)) = \mu_0^{-1}(\mu_0(\Psi(x))) = \Psi(x) \qquad \textup{for all } x \in \mathcal{X}.
		\end{equation*}
		In contrast to the affine linear layers in~\eqref{eq:NODE}, the two (possibly nonlinear) layers are only defined where they are evaluated: $\lambda$ on $\mathcal{X}$ and $\tilde{\lambda}$ on the open neighborhood $\mathcal{T}$ of the time-$T$ images $h_{\lambda(\mathcal{X})}(T) \subset \mathcal{S}$.

		Furthermore, as $r$ is a retraction onto $\mathcal{S}$, it holds $$\tilde{\lambda}(\lambda(x)) = \mu_0^{-1}(r(\mu_0(x))) = \mu_0^{-1}(\mu_0(x)) = x$$ for all $x \in \mathcal{X}$. As the solutions $x(t) = x_0$, $\tilde{t}(t) = t_0 + t$ of the suspension flow exist for all $t \in \R$, the solution of the ODE with vector field $f$ and initial condition $\mu((x_0,t_0)^\top)$ is $\mu((x_0,t_0+t)^\top)$ for all $t \in \R$. With the identification $(x,kT)^\top \sim (\Psi^k(x),0)^\top$ in the mapping torus, it follows for all $x \in \mathcal{X}$ and $k \in \mathbb{N}$ that
		\begin{equation*}
			h_{\lambda(x)}(kT) = \mu((x,kT)^\top) = \mu((\Psi^k(x),0)^\top) = \lambda(\Psi^k(x)) \in \mathcal{S},
		\end{equation*}
		hence, $\tilde{\lambda}(h_{\lambda(x)}(kT)) = \Psi^k(x)$. For $k = 1$, the surjectivity of $\Psi: \mathcal{X} \rightarrow \mathcal{X}$ implies that the time-$T$ map maps $\mathcal{S}$ onto itself.
	\end{proof}

	The last theorem has shown that the idea of the suspension flow can be transferred to a Euclidean space. The disadvantage is that $2n+2$ instead of $n+1$ dimensions are needed, and that the neural ODE architecture of Theorem \ref{th:twononlinearlayers} is more complicated than the augmented neural ODE of the suspension flow in Theorem \ref{th:suspension}.
    
    \begin{remark}\label{rem:tradeoffs}
    Theorem~\ref{th:twononlinearlayers} is an existence result: it proceeds via the Whitney embedding $\mu: \mathcal{M} \to \mathbb{R}^{2n+2}$, but makes no claim about the complexity of representing the input and output transformations. However, the Whitney construction is non-constructive: for a general diffeomorphism $\Psi$ the embedding can be geometrically complicated, and approximating $\mu$ by a neural network may require high representational complexity. As the layers $\lambda$ and $\tilde{\lambda}$ constructed from $\mu$ are smooth maps between open subsets of Euclidean spaces, they can in principle be approximated to arbitrary precision by MLPs via the universal approximation theorem, but the difficulty lies in the representational complexity that may be required. 
    
    An alternative is to cover $\mathcal{M}$ by a finite atlas of local charts and conjugate the suspension flow chart-by-chart. Within each chart the flow reduces to a standard ODE in $\mathbb{R}^{n+1}$, avoiding the doubling of dimension, and the local coordinate maps are typically simpler than a global Whitney embedding. The cost is that chart transitions must be handled explicitly via a partition-of-unity or switching mechanism, introducing time-dependence into the architecture.

    Compared with the universal embedding Theorem~\ref{th:nodeA} for $n = q$, which embeds every $\Psi \in C^k(\mathcal{X},\R^n)$, $k \geq 0$, already in dimension $m = 2n$ with affine linear layers, Theorem~\ref{th:twononlinearlayers} requires dimension $2n+2$, possibly nonlinear layers, and a smooth diffeomorphism $\Psi$. The added value of Theorem~\ref{th:twononlinearlayers} lies in its dynamical consistency: the time-$T$ map leaves the embedded copy $\mathcal{S}$ of $\mathcal{X}$ invariant, such that the iterates of $\Psi$ are represented by the flow itself, i.e., $\Psi^k(x) = \tilde{\lambda}(h_{\lambda(x)}(kT))$ for all $k \in \mathbb{N}$. In the construction of Theorem~\ref{th:nodeA}, the time-$T$ map does in general not map $\lambda(\mathcal{X})$ into itself.
\end{remark}

	\section{Conclusion and Outlook}

	Neural ordinary differential equations are a class of neural networks that have gained particular interest in recent years. The advantages are that neural ODEs can be trained with constant memory cost and that they can represent either input-output relations or time-series data. In this work, we focused on the input-output map of different neural ODE architectures and their embedding capability. Even though universal approximation theorems are quite helpful in practice, the study of embeddings via a dynamical systems viewpoint helps to understand and compare the structure and capabilities of different neural ODE architectures.

	\medskip

	In Section~\ref{sec:overview}, we introduced four neural ODE architectures, illustrated their behavior in low-dimensional examples, refined and generalized already existing results, and then stated several new structure theorems. In particular, we focused on three different fundamental questions: the performance in low dimensions, the existence of non-embeddable function classes, and the universal embedding property. Here, we assumed that the solution of the initial value problem contained in the neural ODE architecture exists on the time interval~$[0,T]$ in order to have a well-defined time-$T$ map. Furthermore, we assumed that the solution is unique and that the time-$T$ map depends continuously on the initial condition. Very general conditions guaranteeing our Assumptions~\ref{assA:ode_existence},~\ref{assA:ode_unique}, and~\ref{assA:continuous_dependence} are for example the Carathéodory conditions. As typical activation functions used for feed-forward neural networks are Lipschitz continuous, the corresponding initial value problem used in the neural ODE architecture fulfills the required assumptions.

	The simplest neural ODE architecture is a basic non-augmented neural ODE, which maps an initial condition of an initial value problem to its time-$T$ map. In other contexts, this problem is also called the restricted embedding problem, discussed in Section~\ref{sec:restrictedembedding}. The derivation of Julia's functional equation via the Jabotinsky equations, which yields candidate vector fields~$f$ for embedding a given map $\Psi$, is carried out in Appendix~\ref{app:julia}. As a concrete instance of this approach, Theorem~\ref{th:julia_monome} shows how to embed monomial functions exactly, and Corollary~\ref{cor:approximatepolynomial} extends this to a polynomial approximation result.

	We have seen that basic non-augmented neural ODEs have restricted embedding capability, as they are in the differentiable case orientation-preserving diffeomorphisms. To overcome these restrictions, we studied two advanced neural ODE architectures: general non-augmented neural ODEs and basic augmented neural ODEs. In both cases, we showed via one-dimensional examples that these architectures perform better than basic non-augmented neural ODEs. Nevertheless, there exist function classes that cannot be embedded in these two neural ODE architectures. In the case of general non-augmented neural ODEs with full-rank weight matrices and $q \leq m$, Theorem~\ref{th:nocriticalpoint}, which follows from~\citet{Kuehn2024}, shows the non-embeddability of functions with critical points. In the basic augmented case, we used Morse functions to define a non-embeddable function class. We introduced Morse functions in Section~\ref{sec:morsesection} and showed in Corollary~\ref{cor:morsefunctionsdense} that they are dense in the Banach space of $k$-times continuously differentiable functions, $k \geq 2$. We defined the suspension flow on the mapping torus to show that one additional dimension is enough to embed any diffeomorphism $\Psi \in C^1(\mathcal{X},\mathcal{X})$, $\mathcal{X}\subset \R^n$ open, in the time-$T$ map of a flow on a manifold of dimension $n+1$. In Section~\ref{sec:suspensionflows}, this result is transferred to Euclidean spaces to show the embedding of smooth diffeomorphisms in general augmented neural ODEs with two (possibly nonlinear) layers and $2n+2$ dimensions.
	
	In Section~\ref{sec:nodeA} we showed an even stronger statement, that already the combination of a general architecture with an augmented phase space, i.e., a general augmented neural ODE, is sufficient to embed any continuous function $\Psi:\mathcal{X}\rightarrow \R^q$, $\mathcal{X}\subset \R^n$, if $m \geq n+q$, see Theorem~\ref{th:nodeA}. Together, these results establish a coherent picture of how augmentation and the presence of affine layers jointly govern the embedding capabilities of neural ODE architectures, and clarify why certain architectures are strictly more expressive than others.
 
    \medskip
 
    The present work focuses on existence results for neural ODE embeddings. The following two questions arise naturally from it. The embedding theorems established here provide a starting point for a perturbation analysis aimed at deriving approximation results. While the non-embeddability of a function class does not directly imply it cannot be approximated arbitrarily well, embedding results do have implications for the universal approximation property, as already shown in~\citet{Kuehn2024} for architectures which cannot represent critical points. Developing a transparent framework that connects embedding capabilities to approximation results thus remains an important open problem. 
    
    Beyond approximation, we note that Theorem~\ref{th:nodeA} achieves universal embedding already with autonomous dynamics, which raises a natural quantitative question. Non-autonomous dynamics in~$\R^m$ are strictly more expressive than autonomous dynamics in $\R^m$, yet form a strict subclass of autonomous dynamics in~$\R^{m+1}$. Since universal embedding is impossible (for autonomous and non-autonomous systems) for $m \leq n$, our result that it is already achieved by autonomous systems at $m = n+q$ is sharp up to the intermediate regime $n < m < n+q$, where the effect of non-autonomous dynamics remains open for $q > 1$. 
	
	\vspace{10mm}
	
	\textbf{Acknowledgements:} This work was supported by the Deutsche Forschungsgemeinschaft (DFG) via the SPP 2298 `Theoretical Foundations of Deep Learning' and via the Research Training Group GRK 3081 `Machine Learning and Control Theory: Exploring Synergies, Complementarities, and Mutual Benefits' (METEOR). CK would like to thank the VolkswagenStiftung for support via a Lichtenberg Professorship. SVK would like to thank the Munich Data Science Institute (MDSI) for partial support via a Linde doctoral fellowship. We also thank two anonymous referees for their valuable feedback and suggestions, which led to significant improvements in the article.

	\vspace{10mm}

	\bibliographystyle{abbrvnat}
\bibliography{sample}
	
\newpage

    \appendix

    \section{Jabotinsky and Julia's Functional Equations} \label{app:julia}
 
    The following material is used to derive the explicit vector fields in Theorem~\ref{th:julia_monome} via a systematic heuristic. It is presented here for completeness. The theorem itself does not depend on these results.
    We focus on autonomous initial value problems of the form
	\begin{equation} \label{eq:IVPauto}
		\frac{\dd h}{\dd t} = f(h(t)), \qquad h(0) = x \in \mathcal{X} \subset \mathbb{R}^n, \tag{$\textup{IVP}_\textup{auto}$}
	\end{equation}
	where the vector field $f: \Omega_h \rightarrow \mathbb{R}^n$, $\Omega_h \subset \R^n$ open, does not explicitly depend on the time variable $t$. As in Section \ref{sec:overview}, we denote the maximal time interval of existence by~$\mathcal{I}_x$ for $x \in \mathcal{X}$. In the following subsections, we show how to construct candidates for the vector field $f$ in order to embed a function $\Psi$ in a basic non-augmented neural ODE. As the autonomous initial value problem \eqref{eq:IVPauto} is a special case of the general initial value problem \eqref{eq:IVP}, the vector fields derived here are also candidates for the basic non-augmented neural ODE architecture $\textup{NODE}_\textup{N,b}(\mathcal{X},\R^n)$ as defined in Section \ref{sec:nodeNB}.
 
    In this work, we consider vector fields $f$ fulfilling Assumptions \ref{assA:ode_existence}, \ref{assA:ode_unique}, and \ref{assA:continuous_dependence}. As motivated in Section \ref{sec:overview}, a large class of vector fields fulfilling these three assumptions is given by the vector fields that satisfy the Carathéodory conditions, which allow for discontinuities with respect to the time variable~$t$. As autonomous initial value problems \eqref{eq:IVPauto} are independent of time, we assume in the remaining part of this appendix that the vector field $f$ of \eqref{eq:IVPauto} is continuous. In Section~\ref{sec:jabotinsky} below we derive the Jabotinsky functional equations characterizing general solutions of~\eqref{eq:IVPauto}. Taking additionally into account the embedding condition $h_x(T) = \Psi(x)$, Julia's functional equation is derived. In Section~\ref{sec:juliafunctional}, we analyze and solve Julia's functional equation and derive for the monomial maps $\Psi(x) = cx^\alpha$ the vector fields of Theorem~\ref{th:julia_monome}.

    Already for one-dimensional maps $\Psi\in C^0(\mathbb{R},\mathbb{R})$ it is of interest to determine whether these can be embedded in autonomous basic neural ODEs. As shown in Theorem~\ref{th:monotoneembedding} and Remark~\ref{rem:linearinterp}, embedding via non-autonomous vector fields is always possible for strictly increasing maps $\Psi \in C^1(\mathcal{X},\R)$ with $\Psi' > 0$ via linear interpolation, but the resulting vector fields are typically non-autonomous and algebraically complex. The autonomous case is therefore of particular interest, as it yields structurally simpler representations. A necessary condition for such an embedding is that Julia's functional equation is fulfilled. Here, we focus on the class of monomials $\Psi(x) = cx^\alpha$ with $c, x, \alpha \in \mathbb{R}_{>0}$, as these are the building blocks of polynomials, which can approximate every continuous function on a real closed interval by the Stone-Weierstraß theorem~\citep{Cheney1982}. As the map $\Psi$ has by Theorem~\ref{th:orientationpreserving} to be strictly increasing in $x$ to be embeddable in a basic non-augmented neural ODE, the restriction to $c, x, \alpha \in \mathbb{R}_{>0}$ is natural. The following sections illustrate how Julia's functional equation provides a systematic route to finding autonomous vector fields $f$ for which the time-$T$ map of the basic non-augmented neural ODE coincides with $\Psi$.

    \subsection{Jabotinsky Equations}
	\label{sec:jabotinsky}
 
	Under the assumption that $h_x \in C^{1}([0,T],\mathbb{R}^n)$ is a solution of the autonomous initial value problem \eqref{eq:IVPauto}, the extended map $g \in C^{1,0}([0,T] \times \mathcal{X}, \mathbb{R}^n)$, $g(t,x) \coloneqq h_x(t)$, which explicitly takes into account the initial condition $x \in \mathcal{X}$, fulfills the translation equation
	\begin{equation} \label{eq:translation_equation} \tag{T}
		g(s+t,x) = g(t,g(s,x)),
	\end{equation}
	with $s,t,s+t \in \mathcal{I}_x$, $x \in \mathcal{X}'$ and $\mathcal{X}' \subset \mathcal{X}$ such that $g(s,x) \in \mathcal{X}$~\citep{Chicone2006}.
 
	\begin{definition}[Flow Map~\citep{Chicone2006}]
		A map $g \in C^{0,0}(\mathcal{I} \times \mathcal{X}, \mathbb{R}^n)$, $\mathcal{X} \subset \mathbb{R}^n$, is called a flow if $g(0,x) = x$ and the translation equation~\eqref{eq:translation_equation} is fulfilled for all $s,t \in \mathcal{I}$ and $x \in \mathcal{X}$ for which both sides of the equation are well-defined.
	\end{definition}
 
	The problem of finding a basic non-augmented neural ODE $\Phi \in \textup{NODE}_\textup{N,b}(\mathcal{X},\R^n)$ with autonomous vector field, which embeds a given map $\Psi\in C^0(\mathcal{X},\mathbb{R}^n)$, is equivalent to finding a flow $g \in C^{1,0}([0,T]\times \mathcal{X}, \mathbb{R}^n)$, with $g(T,x) = \Psi(x)$ for all $x \in \mathcal{X}$. The autonomous initial value problem \eqref{eq:IVPauto} defining the basic neural ODE is obtained by differentiating the translation equation~\eqref{eq:translation_equation} with respect to the time~$t$, evaluating at $t = 0$, and renaming $s$ to~$t$:
	\begin{equation} \label{eq:Tdiff}
		\frac{\partial g(s+t,x)}{\partial t} = \frac{\partial g(t,g(s,x))}{\partial t} \qquad \Rightarrow \qquad \frac{\partial g(t,x)}{\partial t} = \frac{\partial g(0,g(t,x))}{\partial t} = f(g(t,x)),
	\end{equation}
	where $f(y) \coloneqq \frac{\partial g}{\partial t}(0,y)$ is continuous. \medskip
 
	Under the additional assumption that the flow map $g$ is differentiable both with respect to the initial condition $x$ and the time $t$, the relationship between the flow map $g$ and the map $\Psi$ to embed can be characterized via the three Jabotinsky functional equations.
 
	\begin{lemma}[Jabotinsky Equations~\citep{Aczel1988}] \label{lem:jabotinsky}
		Let $g \in C^{1,1}(\mathcal{I}\times \mathcal{X}, \mathbb{R}^n)$ be a flow map fulfilling the translation equation~\eqref{eq:translation_equation} and the initial condition $g(0,x) = x$ for all $x \in \mathcal{X}'$ with $\mathcal{X}' \subset \mathcal{X} \subset \mathbb{R}^n$, such that $g(s,x) \in \mathcal{X}$. Then $g$ fulfills for $x \in \mathcal{X}'$ and $t\in \textup{int}(\mathcal{I})$ the three Jabotinsky equations
		\begin{align}
			\hspace{4cm}&\frac{\partial g(t,x)}{\partial t} &&=&& \frac{\partial g(t,x)}{\partial x} \cdot f(x), \hspace{4cm} \tag{J1} \label{eq:J1} \\
			&\frac{\partial g(t,x)}{\partial t}  &&=&& f(g(t,x)), \tag{J2} \label{eq:J2} \\
			&\frac{\partial g(t,x)}{\partial x} \cdot f(x) &&=&& f(g(t,x)), \tag{J3} \label{eq:J3}
		\end{align}
		with differential initial condition $f(x) = \frac{\partial g(t,x)}{\partial t} \big\vert_{t=0}$.
	\end{lemma}
 
	\begin{proof}
		The first Jabotinsky equation is obtained by differentiating the translation equation~\eqref{eq:translation_equation} with respect to $s$ and then setting $s=0$. Analogously, the second Jabotinsky equation is obtained by differentiating~\eqref{eq:translation_equation} with respect to $t$ and then setting $t=0$, as already shown in \eqref{eq:Tdiff}. The third Jabotinsky equation is a combination of \eqref{eq:J1} and \eqref{eq:J2}.
 
		The differential initial condition $f(x) = \frac{\partial g(t,x)}{\partial t} \big\vert_{t=0}$ follows by evaluating \eqref{eq:J2} at $t = 0$ and inserting the initial condition $g(0,x) = x$.
	\end{proof}
 
	The general solutions of the Jabotinsky Equations \eqref{eq:J1}, \eqref{eq:J2}, and \eqref{eq:J3} are characterized in the one-dimensional case by~\citet{Aczel1988}. To study under Assumption \ref{assA:ode_existence}, that $[0,T]\subset \mathcal{I}$, the embedding of a map $\Psi$ in the solution of an autonomous initial value problem~\eqref{eq:IVPauto}, the constraint $h_x(T) = g(T,x) = \Psi(x)$ needs to be combined with the results of Lemma~\ref{lem:jabotinsky}. As the embedding considers the map $g(t,x)$ at the fixed time $t = T$, only the third Jabotinsky equation~\eqref{eq:J3} is of major interest, as~\eqref{eq:J1} and~\eqref{eq:J2} contain partial derivatives of $g$ with respect to $t$. Under the assumption that the map $\Psi$ is differentiable, inserting $t=T$ in the third Jabotinsky equation \eqref{eq:J3} leads to Julia's functional equation
	\begin{equation} \label{eq:julia}
		J_\Psi(x) \cdot f(x) = f(\Psi(x)), \qquad x \in \mathcal{X}, \tag{J}
	\end{equation}
	where $J_\Psi(x) = \frac{\partial \Psi(x)}{\partial x}$ is the Jacobian matrix of the differentiable map $\Psi$. In the next section, we study solutions of Julia's functional equation \eqref{eq:julia}, characterizing the relationship between the map $\Psi$ and the vector field $f$ for the embedding in basic non-augmented neural ODEs.

    \subsection{Julia's Functional Equation} \label{sec:juliafunctional}
 
	In the last section, we derived via the third Jabotinsky equation \eqref{eq:J3} Julia's functional equation \eqref{eq:julia} as a necessary condition for the embedding of a differentiable map $\Psi$ in a basic non-augmented neural ODE based on an autonomous initial value problem \eqref{eq:IVPauto}. Hence, every pair of solution functions of a differentiable map $\Psi$ and a continuous map $f$ of Julia's functional equation \eqref{eq:julia} is a candidate to define via the autonomous initial value problem \eqref{eq:IVPauto} a basic non-augmented neural ODE which embeds the map $\Psi$. Whether it does must be verified independently.
 
	From the structure of the functional equation \eqref{eq:julia}, we can directly conclude that for every differentiable function $\Psi: \mathbb{R}^n \rightarrow \mathbb{R}^n$, the zero function $f: \mathbb{R}^n \rightarrow \mathbb{R}^n$, $x \mapsto 0$, is a trivial solution. In the context of neural ODEs we are interested in non-trivial solutions to Julia's functional equation, as the trivial ordinary differential equation $\frac{\dd h}{\dd t} = 0$, $h(0) = x$ has only the constant solution $h_x(t) = x$, which embeds the time-$T$ map $\Psi(x) = x$. Furthermore, we remark that for every differentiable function $\Psi: \mathbb{R}^n \rightarrow \mathbb{R}^n$ and continuous solution $f: \mathbb{R}^n \rightarrow \mathbb{R}^n$ of Julia's functional equation, also $af: x \mapsto a f(x)$ solves~\eqref{eq:julia} for $a \in \mathbb{R}$. Hence, the solution $f$ of Julia's functional equation is only defined up to a multiplicative constant. In the literature, some additional results regarding the general structure of one-dimensional solutions to Julia's functional equation~\eqref{eq:julia} exist~\citep{Kuczma1990,Zdun1974}. \medskip

    Before determining explicit solutions of Julia's functional equation, we show in the following lemma that its role as a necessary condition extends beyond autonomous ODEs to the broader class of separable initial value problems, where the vector field takes the product form $f(h(t))\cdot\beta(t)$ with $\beta$ an arbitrary continuous function of time. Hence, the vector fields derived below via the autonomous case are candidates for the separable setting as well.
 
\begin{lemma} \label{lem:separable}
	Let $f,\beta \in C^0(\mathbb{R},\mathbb{R})$ with $f$ locally Lipschitz continuous, let $\mathcal{X} \subset \mathbb{R}$ be open, and assume that for every $x \in \mathcal{X}$ the solution $h_x$ of the one-dimensional separable initial value problem
	\begin{equation*}
		\frac{\dd h}{\dd t} = f(h(t)) \cdot \beta(t), \qquad h(0) = x,
	\end{equation*}
	exists on $[0,T]$. If $h_x(T) = \Psi(x)$ for all $x \in \mathcal{X}$ for a differentiable map $\Psi: \mathcal{X} \rightarrow \mathbb{R}$, then $f$ and $\Psi$ satisfy Julia's functional equation~\eqref{eq:julia} on $\mathcal{X}$.
\end{lemma}
 
	\begin{proof}
		As $f$ is locally Lipschitz continuous and $\beta$ is continuous, solutions of the separable ODE are unique for every initial time and initial value. Let $x \in \mathcal{X}$. If $f(x) = 0$, then $h_x \equiv x$, such that $\Psi(x) = x$ and both sides of~\eqref{eq:julia} vanish at $x$.

		Let now $f(x) \neq 0$, let $\mathcal{C} \subset \mathbb{R}$ be the connected component of $\{y \in \mathbb{R} : f(y) \neq 0\}$ containing $x$, which is an open interval, and let $\mathcal{U} \subset \mathcal{X} \cap \mathcal{C}$ be an open interval containing $x$. For every $\tilde{x} \in \mathcal{U}$, the solution $h_{\tilde{x}}$ cannot reach a zero $y_0$ of $f$, as the constant function $y_0$ is a solution and solutions are unique, also backward in time. Hence, $h_{\tilde{x}}(t) \in \mathcal{C}$ for all $t \in [0,T]$. Let $F \in C^1(\mathcal{C},\mathbb{R})$ be an antiderivative of $\frac{1}{f}$ on $\mathcal{C}$. By the chain rule, $\frac{\dd}{\dd t} F(h_{\tilde{x}}(t)) = \beta(t)$, and integration over $[0,T]$ yields
		\begin{equation*}
			\int_{\tilde{x}}^{\Psi(\tilde{x})} \frac{1}{f(y)} \; \dd y = F(\Psi(\tilde{x})) - F(\tilde{x}) = \int_0^T \beta(t) \; \dd t \qquad \textup{for all } \tilde{x} \in \mathcal{U}.
		\end{equation*}
		As the right-hand side is independent of $\tilde{x}$, differentiating both sides with respect to $\tilde{x}$ at $\tilde{x} = x$ yields by the Leibniz integral rule~\citep[see for example][]{Protter1985}
		\begin{equation*}
			\frac{1}{f(\Psi(x))} \cdot \frac{\dd \Psi(x)}{\dd x} - \frac{1}{f(x)} = 0.
		\end{equation*}
		Multiplying by $f(x) f(\Psi(x)) \neq 0$ yields $\frac{\dd \Psi(x)}{\dd x} f(x) = f(\Psi(x))$, which is Julia's functional equation~\eqref{eq:julia} in the one-dimensional case.
	\end{proof}

    We now apply Julia's functional equation to the monomial maps $\Psi(x) = cx^\alpha$ with $c, x, \alpha \in \mathbb{R}_{>0}$, deriving the vector fields appearing in Theorem~\ref{th:julia_monome}. The derivation proceeds in two cases.
 
	If $\alpha = 1$, then $\Psi(x) = cx$ and Julia's functional equation~\eqref{eq:julia} reduces to $c\, f(x) = f(cx)$, which is solved by every linear function $f(x) = ax$ with $a \in \mathbb{R}$. Choosing $a = \frac{\ln(c)}{T}$ yields the vector field in equation~\eqref{eq:linODE1} in Section~\ref{sec:restrictedembedding}.
 
	If $\alpha \neq 1$, then $\Psi(x) = cx^\alpha$ and Julia's functional equation~\eqref{eq:julia} reduces to
	\begin{equation*}
		c\alpha x^{\alpha-1} f(x) = f(cx^\alpha).
	\end{equation*}
	With the ansatz $f(x) = x\tilde{f}(x)$, this becomes $\alpha\tilde{f}(x) = \tilde{f}(cx^\alpha)$. Defining $\nu(x) \coloneqq \tilde{f}\!\left(c^{\frac{1}{1-\alpha}}e^x\right)$ for $x \in \mathbb{R}$, one computes
	\begin{equation*}
		\alpha\,\nu(x) = \alpha\,\tilde{f}\!\left(c^{\frac{1}{1-\alpha}}e^x\right) = \tilde{f}\!\left(c\left(c^{\frac{1}{1-\alpha}}e^x\right)^\alpha\right) = \tilde{f}\!\left(c^{\frac{1}{1-\alpha}}e^{\alpha x}\right) = \nu(\alpha x),
	\end{equation*}
	so that the functional equation for $\nu$ is again solved by every linear function $\nu(x) = ax$, $a \in \mathbb{R}$. Substituting back gives
	\begin{equation*}
		f(x) = x\,\tilde{f}(x) = x\,\nu\!\left(\ln\!\left(c^{\frac{-1}{1-\alpha}}x\right)\right) = ax\ln\!\left(c^{\frac{1}{\alpha-1}}x\right).
	\end{equation*}
	Choosing $a = \frac{\ln(\alpha)}{T}$ yields the vector field in~\eqref{eq:linODE2} in Section~\ref{sec:restrictedembedding}. That the autonomous initial value problems \eqref{eq:linODE1}-\eqref{eq:linODE2} defined by these vector fields indeed have the time-$T$ map $h_x(T) = cx^\alpha = \Psi(x)$ is verified directly in the proof of Theorem~\ref{th:julia_monome}.

    \section{One-Dimensional Topological Morse Functions} \label{app:onedmorse}

    For a one-dimensional input space $\mathcal{X} \subset \R$, both the existence and the structure of critical points are more restrictive than in the general setting of Section~\ref{sec:morsefunctions}. The first result concerns the existence of critical points and shows that differentiable functions without critical points on an open interval are injective.

    \begin{proposition} \label{prop:criticalpoint}
		Let $\mathcal{X} \subset \R$ be an open interval and let $\psi: \mathcal{X} \rightarrow \R$ be differentiable and non-injective. Then $\psi$ has at least one critical point, i.e., there exists $p \in \mathcal{X}$ with $\nabla \psi(p) = 0$.
	\end{proposition}

	\begin{proof}
		As $\psi$ is non-injective, there exist $x_1, x_2 \in \mathcal{X}$ with $x_1 < x_2$ and $\psi(x_1) = \psi(x_2)$. As $\mathcal{X}$ is an interval, the compact interval $[x_1,x_2]$ is contained in $\mathcal{X}$, and the continuous function $\psi$ attains both its minimum and its maximum on $[x_1,x_2]$. If both extrema equal $\psi(x_1) = \psi(x_2)$, then $\psi$ is constant on $[x_1,x_2]$ and every point of $(x_1,x_2)$ is a critical point. Otherwise, one extremum differs from $\psi(x_1) = \psi(x_2)$ and is therefore attained at an interior point $p \in (x_1,x_2)$. As $\psi$ is differentiable at $p$, it follows $\nabla \psi(p) = 0$.
	\end{proof}

	\begin{remark} \label{rem:criticalpoint}
		Equivalently, Proposition~\ref{prop:criticalpoint} states that every differentiable function without critical points on an open interval $\mathcal{X} \subset \R$ is injective. Both the restriction to an interval and the restriction to a one-dimensional input space are necessary. On the disconnected open set $\mathcal{X} = (0,1) \cup (1,2)$, the differentiable map $\psi(x) = x$ for $x \in (0,1)$ and $\psi(x) = x-1$ for $x \in (1,2)$ is non-injective and has no critical point. For $n \geq 2$, the map $\psi: \R^n \rightarrow \R$, $x \mapsto \sum_{j=1}^n x_j$, is non-injective, as
		\begin{equation*}
			\psi(x_1+t, x_2-t, x_3, \ldots, x_n) = \sum_{j=1}^n x_j = \psi(x)
		\end{equation*}
		for all $x \in \R^n$ and all $t \in \R$, but it has no critical point, as $\nabla \psi(x) = (1,\ldots,1) \neq 0$ for all $x \in \R^n$.
	\end{remark}

	While Proposition~\ref{prop:criticalpoint} decides whether critical points exist, the following lemma determines their structure. A degenerate critical point has no normal form in the sense of the Morse-Palais Lemma~\ref{th:morse}. Example~\ref{ex:topologicallymorse} shows that the degenerate critical point of the map $x \mapsto x^4$ can still be transformed into the quadratic form~\eqref{eq:quadraticform} by a homeomorphism. In one dimension, every critical point at which not all derivatives of $\psi$ vanish admits a local normal form.

	\begin{lemma} \label{lem:onedhigherorder}
		Let $\psi \in C^{k+1}(\mathcal{X},\R)$ with $\mathcal{X} \subset \R$ open and $k \geq 2$, and let $p \in \mathcal{X}$ be a critical point of~$\psi$ satisfying $\psi^{(j)}(p) = 0$ for all $1 \leq j < k$ and $\psi^{(k)}(p) \eqqcolon s \neq 0$, where $\psi^{(j)}$ denotes the $j$-th derivative of~$\psi$. Then there exist a neighborhood $\mathcal{U}$ of $0 \in \R$ and a $C^1$-diffeomorphism $\mu: \mathcal{U} \rightarrow \mu(\mathcal{U})$ with $\mu(0) = p$ such that for all $u \in \mathcal{U}$
		\begin{equation} \label{eq:onednormalform}
			\psi(\mu(u)) = \psi(p) + (\textup{sign}(s))^{k-1} u^k.
		\end{equation}
	\end{lemma}

	\begin{proof}
		The idea of the proof is based on~\citet{Castrigiano2019}, where the statement is outlined for smooth functions vanishing at the origin. Let $\mathcal{U}_1 \subset \mathcal{X}$ be an open interval containing $p$. As $\psi^{(j)}(p) = 0$ for all $1 \leq j < k$, Taylor's formula with integral remainder yields for all $x \in \mathcal{U}_1$
		\begin{equation*}
			\psi(x) = \psi(p) + g(x)(x-p)^k, \qquad g(x) \coloneqq \int_0^1 \frac{(1-t)^{k-1}}{(k-1)!}\, \psi^{(k)}\left(p + t(x-p)\right) \dd t,
		\end{equation*}
		where $g \in C^1(\mathcal{U}_1,\R)$, as $\psi \in C^{k+1}(\mathcal{X},\R)$. Evaluating $g$ at $p$ yields $g(p) = \frac{s}{k!}$. As $s \neq 0$ and $g$ is continuous, there exists an open interval $\mathcal{U}_2 \subset \mathcal{U}_1$ containing $p$ such that $\textup{sign}(s)\, g(x) > 0$ for all $x \in \mathcal{U}_2$. Hence, the map
		\begin{equation*}
			\nu: \mathcal{U}_2 \rightarrow \R, \qquad \nu(x) \coloneqq \textup{sign}(s) \sqrt[k]{\textup{sign}(s)\, g(x)}\, (x-p),
		\end{equation*}
		is well-defined and continuously differentiable with $\nu(p) = 0$ and
		\begin{equation*}
			\nu'(x) = \textup{sign}(s) \sqrt[k]{\textup{sign}(s)\, g(x)} + \frac{1}{k} (\textup{sign}(s)\, g(x))^{\frac{1}{k}-1} g'(x) (x-p),
		\end{equation*}
		such that $\nu'(p) = \textup{sign}(s) \sqrt[k]{\textup{sign}(s)\, g(p)} \neq 0$. By the inverse function theorem, there exists an open interval $\mathcal{U}_3 \subset \mathcal{U}_2$ containing $p$ such that $\nu: \mathcal{U}_3 \rightarrow \nu(\mathcal{U}_3)$ is a $C^1$-diffeomorphism. Set $\mathcal{U} \coloneqq \nu(\mathcal{U}_3)$, which is a neighborhood of $0$ as $\nu(p) = 0$, and let $\mu \coloneqq \nu^{-1}: \mathcal{U} \rightarrow \mathcal{U}_3$, which satisfies $\mu(0) = p$. Using $(\textup{sign}(s))^{2k} = 1$, it holds for all $x \in \mathcal{U}_3$
		\begin{align*}
			\psi(p) + (\textup{sign}(s))^{k-1} \nu(x)^k &= \psi(p) + (\textup{sign}(s))^{k-1} (\textup{sign}(s))^{k} \left(\textup{sign}(s)\, g(x)\right) (x-p)^k \\
			&= \psi(p) + g(x)(x-p)^k = \psi(x).
		\end{align*}
		Substituting $x = \mu(u)$ for $u \in \mathcal{U}$ gives the normal form~\eqref{eq:onednormalform}.
	\end{proof}

	The order $k$ of the first non-vanishing derivative decides whether the critical point is topologically ordinary or topologically critical. If $k$ is odd, the normal form~\eqref{eq:onednormalform} is a strictly monotone function of $u$ and can be transformed into the linear normal form of a topologically ordinary point. If $k$ is even, it can be transformed into the quadratic form~\eqref{eq:quadraticform}, and the critical point has an index. In both cases, the requirements of Definition~\ref{def:morse_topological} are fulfilled, as the following proposition shows.

	\begin{proposition} \label{prop:onedtopologicalmorse}
		Let $\psi \in C^{k+1}(\mathcal{X},\R)$ with $\mathcal{X} \subset \R$ open and $k \geq 2$. Suppose that for every critical point $p \in \mathcal{X}$ of $\psi$ there exists an order $k_p \leq k$ such that $\psi^{(j)}(p) = 0$ for all $1 \leq j < k_p$ and $\psi^{(k_p)}(p) \neq 0$. Then $\psi$ is a topological Morse function.
	\end{proposition}

	\begin{proof}
		Let $y \in \mathcal{X}$ be a point with $\nabla \psi(y) \neq 0$. By the inverse function theorem, there exists an open interval containing $y$ on which $\psi$ is a $C^1$-diffeomorphism onto its image, such that $\eta(v) \coloneqq \psi^{-1}(\psi(y)+v)$ defines a homeomorphism on a neighborhood $\mathcal{V}$ of $0$ with $\eta(0) = y$ and $\psi(\eta(v)) = \psi(y) + v$ for all $v \in \mathcal{V}$. Hence, $y$ is topologically ordinary, and every topologically critical point of $\psi$ is a critical point of $\psi$.

		Let now $p \in \mathcal{X}$ be a critical point of $\psi$ with order $k_p$ and denote the corresponding derivative by $s \coloneqq \psi^{(k_p)}(p) \neq 0$. As $\nabla \psi(p) = 0$, it holds $k_p \geq 2$, and as $k_p \leq k$, the assumptions of Lemma~\ref{lem:onedhigherorder} are satisfied. Hence, there exist a neighborhood $\mathcal{U}$ of $0$ and a $C^1$-diffeomorphism $\mu: \mathcal{U} \rightarrow \mu(\mathcal{U})$ with $\mu(0) = p$ and
		\begin{equation*}
			\psi(\mu(u)) = \psi(p) + (\textup{sign}(s))^{k_p-1} u^{k_p} \qquad \textup{for all } u \in \mathcal{U}.
		\end{equation*}
		In analogy to Example~\ref{ex:topologicallymorse}, define for $\varrho > 0$ the homeomorphism
		\begin{equation*}
			\sigma_\varrho: \R \rightarrow \R, \qquad z \mapsto \textup{sign}(z) \abs{z}^\varrho,
		\end{equation*}
		with inverse $\sigma_{1/\varrho}$, and distinguish two cases.

		If $k_p$ is odd, then $(\textup{sign}(s))^{k_p-1} = 1$ and $\sigma_{1/k_p}(z)^{k_p} = \textup{sign}(z)^{k_p} \abs{z} = z$. Choosing a neighborhood~$\mathcal{U}'$ of $0$ with $\sigma_{1/k_p}(\mathcal{U}') \subset \mathcal{U}$ yields
		\begin{equation*}
			\psi\left(\mu(\sigma_{1/k_p}(z))\right) = \psi(p) + z \qquad \textup{for all } z \in \mathcal{U}',
		\end{equation*}
		such that $p$ is topologically ordinary with homeomorphism $\mu \circ \sigma_{1/k_p}$.

		If $k_p$ is even, then $(\textup{sign}(s))^{k_p-1} = \textup{sign}(s)$ and $\sigma_{2/k_p}(z)^{k_p} = \textup{sign}(z)^{k_p} \abs{z}^2 = z^2$. Choosing a neighborhood $\mathcal{U}'$ of $0$ with $\sigma_{2/k_p}(\mathcal{U}') \subset \mathcal{U}$ yields
		\begin{equation*}
			\psi\left(\mu(\sigma_{2/k_p}(z))\right) = \psi(p) + \textup{sign}(s)\, z^2 \qquad \textup{for all } z \in \mathcal{U}',
		\end{equation*}
		which is the quadratic form~\eqref{eq:quadraticform} with index $0$ if $s > 0$ and index $1$ if $s < 0$, with the homeomorphism $\mu \circ \sigma_{2/k_p}$.

		Consequently, every topologically critical point of $\psi$ has an index, and $\psi$ is a topological Morse function.
	\end{proof}

	\begin{remark} \label{rem:onedtopologicalmorse}
		The assumption of Proposition~\ref{prop:onedtopologicalmorse} is satisfied, for example, by every non-constant polynomial, where $k$ can be chosen as the degree. More generally, the assumption is satisfied by every analytic function on an interval with finitely many critical points, since a non-constant analytic function has at every point a non-vanishing derivative of finite order. In particular, the map $\psi(x) = x^4$ of Example~\ref{ex:topologicallymorse} is not a Morse function, but it is a topological Morse function with critical point $p = 0$ of order $k_p = 4$ and index $0$.
	\end{remark}
\end{document}